\documentclass{article}
\usepackage[lang = british]{ems-jems} 

\usepackage{ mathrsfs }
\usepackage{dsfont}
\usepackage{mathtools}
\usepackage{varioref}

\theoremstyle{plain}
\newtheorem{theorem}{Theorem}[section]
\newtheorem{proposition}[theorem]{Proposition}
\newtheorem{lemma}[theorem]{Lemma}
\theoremstyle{definition}

\newtheorem{remark}[theorem]{Remark}

\newtheorem{conditionA}{A\kern-0.1mm}
\labelformat{conditionA}{\textbf{A\kern-0.1mm#1}}

\DeclarePairedDelimiter\floor{\lfloor}{\rfloor}
\DeclareMathOperator{\supp}{supp}

\numberwithin{equation}{section}

\begin{document}

\title{Berry--Esseen bound and precise moderate deviations for products of random matrices}
\titlemark{Limit theorems for products of random matrices}


\emsauthor{1}{Hui Xiao}{H.~Xiao}
\emsauthor{2}{Ion Grama}{I.~Grama}
\emsauthor{3}{Quansheng Liu}{Q.~Liu}



\emsaffil{1}{Universit\"{a}t Hildesheim, Institut f\"{u}r Mathematik und Angewandte Informatik, Hildesheim, Germany \email{xiao@uni-hildesheim.de}}
\emsaffil{2}{Universit\'{e} de Bretagne-Sud, LMBA UMR CNRS 6205, Vannes, France \email{ion.grama@univ-ubs.fr}}
\emsaffil{3}{Universit\'{e} de Bretagne-Sud, LMBA UMR CNRS 6205, Vannes, France \email{quansheng.liu@univ-ubs.fr}}


\classification[60B20, 60J05]{60F10}




\begin{abstract}
Let $(g_{n})_{n\geq 1}$ be a sequence of independent and identically distributed (i.i.d.) 
$d\times d$ real random matrices. 
For $n\geq 1$ set $G_n = g_n \ldots g_1$. Given any starting point $x=\mathbb R v\in\mathbb{P}^{d-1}$,  
consider the Markov chain $X_n^x = \mathbb R G_n v $ on the projective space $\mathbb P^{d-1}$ and  
the norm cocycle $\sigma(G_n, x)= \log \frac{|G_n v|}{|v|}$,
for an arbitrary norm $|\cdot|$ on $\mathbb R^{d}$.
Under suitable conditions we prove  
a Berry--Esseen type theorem and an Edgeworth expansion for the couple $(X_n^x,  \sigma(G_n, x))$. 
These results are established using a brand new smoothing inequality on complex plane, 
the saddle point method and additional spectral gap properties of the transfer operator related to the Markov chain $X_n^x$.
Cram\'{e}r type moderate deviation expansions as well as a local limit theorem with moderate deviations
are proved for the couple $(X_n^x, \sigma(G_n, x))$ with a target function $\varphi$ on the Markov chain $X_n^x$.  
\end{abstract}

\maketitle


\section{Introduction}
\subsection{Background and objectives}
For any integer $d\geq 2$, denote by $GL(d, \mathbb{R})$ the 
general linear group of $d\times d$ invertible matrices. 
Equip $\mathbb{R}^{d}$ with any norm $|\cdot|$
 and let $\lVert g \rVert = \sup_{v\in \mathbb{R}^{d}\setminus\{0\}} \frac{|g v|}{|v|}$ 
be the operator norm for $g \in GL(d,\mathbb{R})$.
Denote by $\mathbb{P}^{d-1}$ the projective space of $\mathbb{R}^{d}$.
Let $(g_{n})_{n\geq 1}$ be a sequence of i.i.d. $d\times d$ real random matrices
of the same law $\mu$ on $GL(d,\mathbb{R})$.
For any $n\geq 1$, consider the product $G_n = g_n \ldots g_1$ and the process $X_n^x = \mathbb R G_n v \in \mathbb P^{d-1}$, 
with the starting point $x=\mathbb R v \in \mathbb{P}^{d-1}$.
The norm cocycle is defined by $\sigma(G_n,x)= \log \frac{|G_n v|}{|v|}$, 
where  $x=\mathbb R v \in \mathbb P^{d-1}$.

The study of the asymptotic properties of the Markov chain $(X_n^x)_{n\geq 1}$ and of the product $(G_n)_{n\geq 1}$ 
has attracted a good deal of attention since the 
groundwork of Furstenberg and Kesten \cite{FK60}, where
the strong law of large numbers (LLN) for the operator norm $\lVert G_n \rVert$ has been established. 
In the same context, Furstenberg \cite{Fur63} proved the LLN 
for the norm cocycle $\sigma(G_n,x)$: for any $x\in \mathbb P^{d-1}$,  
\begin{align*}
\lim_{n\to\infty} \frac{1}{n} \sigma(G_n,x)  
= \lim_{n\to\infty} \frac{1}{n} \mathbb E\, \sigma(G_n,x) 
 = \lambda
  \quad  \mathbb{P}\mbox{-a.s.,}
\end{align*} 
where $\lambda$ is a real number called upper Lyapunov exponent associated with the product $G_n$.
Another cornerstone result is the central limit theorem (CLT) for the couple $(X_n^x, \sigma(G_n,x))$, 
established under contracting type assumptions by Le Page \cite{LeP82}:
for any fixed $y \in \mathbb{R}$ and any H\"older continuous function 
$\varphi: \mathbb P^{d-1}\mapsto \mathbb R$, it holds uniformly in $x \in \mathbb{P}^{d-1} $ that
\begin{align*}
\lim_{ n \to \infty } \mathbb{E} 
\Big[  \varphi(X_n^x) \mathds{1}_{ \big\{ \frac{\sigma(G_n,x) - n \lambda }{\sigma \sqrt{n}} \leq y \big\} } \Big]
= \nu (\varphi) \Phi(y),
\end{align*}
where $\nu$ is the unique stationary probability measure of the Markov chain $X_n^x$ on $\mathbb{P}^{d-1}$, 
$ \sigma^2 = \lim_{n\to\infty} \frac{1}{n} \mathbb{E} \big[(\sigma(G_n,x) - n\lambda)^{2} \big]$ is the asymptotic variance
independent of $x$, and $\Phi$ is the standard normal distribution function.
The optimal conditions for the CLT to hold true have been established recently by Benoist and Quint \cite{BQ16a}.  

The next step in these studies is to know how precise are the approximations in  the LLN and the CLT.
The asymptotic of the large deviation probabilities describes  the rate of convergence in the LLN, and 
the Berry--Esseen bound characterizes  that  
in the CLT. For sums of independent random variables these topics 
have been extensively studied over many decades, 
and have been proved to play the key role for many problems in probability theory and mathematical statistics.  
For deep and optimal results in this direction we refer to the 
pioneering works of Cram\'{e}r \cite{Cra38}, 
Esseen \cite{Ess45}, Bahadur and Rao \cite{BR60}, Petrov \cite{Pet65}
and to the monographs of Petrov \cite{Pet75}, Stroock \cite{Str84}, Varadhan \cite{Var84}, 
Dembo and Zeitouni \cite{DZ09} and Borovkov and Borovkov \cite{BB08}. 

For products of random matrices 
the known results about the rate of convergence in the LLN and the CLT are far from being optimal, although
there are already important studies on the topic.
The main goal of the present paper is to fill in this gap 
by proving large deviation asymptotics and Berry--Esseen type bounds which are close to definitive. 
Precise large deviation asymptotics 
originate from the work of Le Page \cite{LeP82} 
and more recently have been considered e.g.\  by 
Guivarc'h  \cite{Gui15}, Benoist and Quint \cite{BQ16b}, Buraczewski and Mentemeier \cite{BM16},  Sert \cite{Ser19}, 
Xiao, Grama and Liu \cite{XGL18}.   
For moderate deviations, very few results are known. 
Benoist and Quint \cite{BQ16b} have recently established the 
moderate deviation principle for reductive groups, which in our setting reads as follows:
for any interval $B \subseteq \mathbb{R}$, and positive sequence $(b_n)_{n\geq 1}$ satisfying
$\frac{b_n}{n}\rightarrow0$ and $\frac{b_n}{\sqrt{n}}\to \infty$ as $n\to\infty$,  
it holds uniformly in $x \in \mathbb{P}^{d-1}$ that
\begin{align} \label{intro MDPintro000}
\lim_{n\to \infty} \frac{n}{b_n^{2}}
\log  \mathbb{P}  \Big( \frac{\sigma(G_n,x) - n\lambda }{b_n} \in B  \Big)
= - \inf_{y\in B} \frac{y^2}{2\sigma^2}.
\end{align}
A functional moderate deviation principle 
has  been established by Cuny, Dedecker and Jan \cite{CDM17}.

The first objective of our paper is to improve on the result \eqref{intro MDPintro000} 
by establishing a Cram\'{e}r type  
moderate deviation expansion for $\sigma(G_n,x)$:
we prove that uniformly in $x\in \mathbb{P}^{d-1}$ and $y \in [0, o(\sqrt{n} )]$, as $n\to\infty$,
\begin{align} \label{ThmNormTarget01b}
\frac{\mathbb{P} \big( \sigma(G_n,x) - n\lambda \geq \sqrt{n}\sigma y  \big) }{1-\Phi(y) }
=  e^{ \frac{y^3}{\sqrt{n}}\zeta ( \frac{y}{\sqrt{n}} ) }
\Big[ 1 + O\Big( \frac{y+1}{\sqrt{n}} \Big) \Big], 
\end{align}
where $t\mapsto \zeta(t)$ is the Cram\'{e}r series of 
the logarithm of the eigenfunction related to the transfer operator of 
the Markov walk associated to the product of random matrices (see Section \ref{SecNormCocy}).

In many important models it is useful 
to extend  the moderate deviation expantion \eqref{ThmNormTarget01b}
for the couple $(X_n^x, \sigma(G_n,x))$
which describes completely the random walk $(G_n v)_{n\geq 1}$.
We prove that, for any H\"{o}lder continuous function $\varphi$ on $\mathbb{P}^{d-1}$,
uniformly in $x\in \mathbb{P}^{d-1}$ and $y \in [0, o(\sqrt{n} )]$, as $n\to\infty$,
\begin{align} \label{ThmNormTarget01c}
  \frac{\mathbb{E}
\big[ \varphi(X_n^x) \mathds{1}_{ \{ \sigma(G_n,x) - n\lambda \geq \sqrt{n}\sigma y \} }  \big] }
{1-\Phi(y) }
=  e^{ \frac{y^3}{\sqrt{n}}\zeta ( \frac{y}{\sqrt{n}} ) }
\Big[ \nu(\varphi) + O\Big( \frac{y+1}{\sqrt{n}} \Big) \Big], 
\end{align}
see Theorem \ref{MainThmNormTarget} for a slightly stronger statement.

Our second objective, which is also the key point in proving \eqref{ThmNormTarget01c}, is 
a Berry--Esseen bound for the couple $(X_n^x, \sigma(G_n,x))$:
for any H\"{o}lder continuous function $\varphi$ on $\mathbb{P}^{d-1}$, as $n\to\infty$,
\begin{align}\label{Intro-BerEs001}
\sup_{x \in \mathbb{P}^{d-1},  \, y \in \mathbb{R} }   
\Big| \mathbb{E}
 \Big[  \varphi(X_n^x) \mathds{1}_{ \big\{ \frac{\sigma(G_n,x) - n \lambda }{\sigma \sqrt{n}} \leq y \big\} } \Big]
-  \nu(\varphi)  \Phi(y) \Big|
  =  O \Big(\frac{1}{\sqrt{n}} \Big),
\end{align}
see Theorem \ref{Thm-BerryEsseen-Norm:s=0}. 
This extends the result of Le Page \cite{LeP82}  established for the particular target function $\varphi = \mathbf{1}$ (see also Jan \cite{Jan01}).
 We further upgrade \eqref{Intro-BerEs001} to an Edgeworth expansion under 
a non-arithmeticity condition, see Theorem \ref{Thm-Edge-Expan:s=0},
which is new even for $\varphi=\mathbf 1$.  

Our third objective is to establish the following local limit theorem with moderate deviations: 
for any real numbers $-\infty < a_1 < a_2 < \infty$,
we have,  uniformly in $x \in \mathbb{P}^{d-1}$ and  $|y| = o(\sqrt{n})$, as $n\to\infty$,
\begin{align}\label{IntroLLT005}
\mathbb{E} 
\Big[ \varphi(X_n^x) \mathds{1}_{ \{ \sigma(G_n,x) - n\lambda \in [a_1, a_2] + \sqrt{n}\sigma y \} } \Big]
= \nu (\varphi) 
\frac{a_2 - a_1}{\sigma\sqrt{2\pi n} }
e^{ -\frac{y^2}{2} +  \frac{y^3}{\sqrt{n}}\zeta(\frac{y}{\sqrt{n}} ) } (1+o(1)).
\end{align}
For a more general version of \eqref{IntroLLT005}, see Theorem \ref{ThmLocal01}, 
where a target function $\psi$ on $\sigma(G_n,x)$ is considered. 
When $|y| =o(n^{1/6})$,  
the term $\frac{y^3}{\sqrt{n}} \zeta(\frac{y}{\sqrt{n}})$ tends to $0$ and can be removed in \eqref{IntroLLT005}.  
In this case, \eqref{IntroLLT005} improves the local limit theorem of \cite[Theorem 17.10]{BQ16b}
established for $|y| = O(\sqrt{\log n})$. 
Local limit theorems with moderate deviations of type  \eqref{IntroLLT005} are used 
for instance in \cite{BQ13} 
for studying dynamics of group actions on finite volume homogeneous spaces. 
As an important application of \eqref{IntroLLT005}
we establish a new local limit theorem with moderate deviations for the operator norm $\lVert G_n\rVert$, 
see Theorem \ref{Thm_LLT_Norm_0a}.

All the results stated above concern invertible matrices, 
but we also establish analogous theorems for positive matrices.
Some limit theorems for $\sigma(G_n,x)$ 
in case of positive matrices such as 
central limit theorem and Berry--Esseen theorem 
have been established earlier by Furstenberg and Kesten \cite{FK60},
Hennion \cite{Hen97}, and Hennion and Herv\'{e} \cite{HH04}.
Here, we extend the Berry--Esseen theorem of \cite{HH04} to the couple $(X_n^x, \sigma(G_n,x))$
with a target function $\varphi$ on the Markov chain $X_n^x$. 
We also complement the results in \cite{FK60, Hen97, HH04}    
by giving a Cram\'er type moderate deviation expansion
and  a local limit theorem with moderate deviations.

The results of the paper can be useful in number of models of growing interest in 
probability and statistics. 
In particular, our study has been motivated by applications 
to branching random walks
and multitype branching processes in random environment;
we refer to 
\cite{BGL20a, BGL20b, GLP20a, GLP20b}
where large deviation asymptotics have been obtained in these settings using the results of this paper.  
For an application to moderate deviations for the operator norm and the 
spectral radius of products of random matrices we refer to \cite{XGL20}. 
Other fields of application include the financial mathematics, among them multidimensional stochastic recursions and perpetuity sequences.

On the other hand with the approach developed in the paper, 
one can also study limit theorems for Markov chains, dynamical systems, 
random walks on hyperbolic groups and homogeneous spaces; 
for these topics we refer to Hennion and Herv\'e \cite{HH01}, 
Parry and Pollicott \cite{ParryPollicottBook90}, Gou\"ezel \cite{Gou09}, 
Guivarc'h \cite{Gui15}, Benoist and Quint \cite{BQhyperbolic16}. 
For example, combining our approach with the techniques from Guivarc'h and Hardy \cite{GH88},
it is possible to obtain extensions of our results to the setting of Anosov's diffeomorphisms 
and more general dynamical systems allowing a coding by mixing sub-shifts. 
As another example, one can establish  
the analogs of the results of the paper 
for Markov chains with compact state spaces. These aspects will  
be not considered here because of the limitation of the length of the paper.

\subsection{Key ideas of the approach}

For the moderate deviation expansions \eqref{ThmNormTarget01b} and \eqref{ThmNormTarget01c}, 
our proof is different from those in \cite{BQ16b} and \cite{CDM17}:
in \cite{BQ16b} the moderate deviation principle \eqref{intro MDPintro000} is obtained by following
the strategy of Kolmogorov \cite{Kol29} suited to show the law of iterated logarithm (see also de Acosta \cite{DeA83} and Wittman \cite{Wit85}); 
in \cite{CDM17} the proof of the functional moderate deviation principle
is based on the martingale approximation method developed in \cite{BQ16a}.

In order to prove \eqref{ThmNormTarget01c}, we need to rework 
the spectral gap theory for the transfer  operators $P_z$ and  $R_{s,z}$, 
by considering the case when $s$ can take values in the interval $(-\eta,\eta)$ with $\eta>0$ small, and $z$ 
belongs to a small complex ball centered at the origin,
see Section \ref{sec:spec gap norm}.
This allows to define the change of measure $\mathbb Q_{s}^{x}$ and to extend 
the Berry--Esseen bound \eqref{Intro-BerEs001} 
for the changed measure $\mathbb Q_{s}^{x}$, see Theorem \ref{Thm-BerryEsseen-Norm}. 
The moderate deviation expansion \eqref{ThmNormTarget01c} is established 
by adapting the techniques from Petrov \cite{Pet75}. 

It is surprising that the proof of the Berry--Esseen bound and of the Edgeworth expansion
with a non-trivial target function $\varphi\not = \mathbf 1$ is way more difficult than the analogous results with $\varphi=\mathbf 1$. 
This can be seen from the sketch of the proof which we  give below.

For simplicity, we assume that $\sigma=1$.
Introduce the transfer operator $P_{z}$: 
for any H\"{o}lder continuous function $\varphi$ on $\mathbb{P}^{d-1}$ and $z \in \mathbb{C}$,
\begin{align} \label{Fourier transf001}
P_{z}\varphi(x)
= \mathbb{E} \Big[ e^{z \sigma(g_1,x) } \varphi(X_1^x) \Big],
\quad   x \in \mathbb{P}^{d-1}.
\end{align}
Let $F$ be the distribution function of $\frac{\sigma(G_n,x) - n \lambda }{\sqrt{n}}$ and
$f$ be its Fourier transform: 
$f(t) = e^{it\sqrt{n} \lambda}(P^n_{{-it}/{\sqrt{n}}} \mathbf{1} ) (x),$ $t\in \mathbb R$.
The Berry--Esseen bound \eqref{Intro-BerEs001} with target function $\varphi= \mathbf{1}$
is usually proved using Esseen's smoothing inequality: 
there exists a constant $C>0$ such that for all $T>0$,
\begin{align}\label{Intro-Berry-Ine}
\sup_{y \in \mathbb{R}} | F(y) - \Phi(y) |
\leq  \frac{1}{ \pi } 
\int_{-T}^T  \Big| \frac{ f(t) - e^{- t^2 /2 } }{ t }  \Big| \,dt  + \frac{C}{T}.
\end{align}
Inserting the spectral gap decomposition 
\begin{align}\label{Rnsw-Decomp}
P^{n}_{z} = \kappa^{n}(z) M_{z} + L^{n}_{z}  \quad  (n \geq 1)
\end{align}
into \eqref{Intro-Berry-Ine} allows us to obtain the Berry--Esseen bound \eqref{Intro-BerEs001}
with $\varphi= \mathbf{1}$:
after some straightforward calculations, it reduces to showing that,
with $T=c\sqrt{n},$ as $n\to\infty$,
\begin{align} \label{crucial001} 
\int_{-T}^T \frac{1}{|t|}\big|(L^{n}_{-it/\sqrt{n}} \mathbf{1}) (x)\big|  \, dt = O\Big(\frac{1}{\sqrt{n}}\Big).
\end{align}
The bound \eqref{crucial001} is proved using 
Taylor's expansion $L^{n}_{z} \mathbf{1}= L^{n}_{0} \mathbf{1} + z\frac{d}{dz} (L^{n}_{z} \mathbf{1}) + o(z)$ with $z=-it/\sqrt{n}$, 
 and  the fact that $L^{n}_{0} \mathbf{1} = 0$. 
 However, when we replace the unit function $\mathbf{1}$ by a target function $\varphi$ for which 
in general $L^{n}_{0} \varphi \neq 0$,  instead of \eqref{crucial001}, we have   
\begin{align}\label{INT_infty001}
\int_{-T}^T \frac{1}{|t|} |L^{n}_{-it/\sqrt{n}} \varphi(x)| \, dt = \infty,
\end{align}
even though $|L^{n}_{0} \varphi(x)|$ 
decays exponentially fast to $0$ as $n \to \infty$.
To overcome this difficulty,
we have elaborated a new approach based on smoothing inequality on complex contours and on the saddle point method, 
see Proposition \ref{Prop-BerryEsseen}.
More precisely, we formulate our smoothing inequality as follows: 
there exists a constant $C> 0$ such that for any $ T \geq r >0$, 
\begin{align}
\sup_{y \in \mathbb{R}} | F(y) - \Phi(y) | 
& \leq \frac{1}{\pi } \sup_{y \leq 0 }  \Big|  \int_{ \mathcal{C}_{r}^{-} } 
\overline f(z) e^{izy}  e^{-i b \frac{z}{T}} \,dz \Big|   +  \frac{1}{\pi } \sup_{y > 0 }   \Big|  \int_{ \mathcal{C}_{r}^{+} }
\overline f(z) e^{izy}  e^{-i b \frac{z}{T}}  \,dz \Big|  \notag\\
& \quad  +  \frac{1}{\pi  } \sup_{y \leq 0 }  \Big|  \int_{ \mathcal{C}_{r}^{-} }
\overline f(z) e^{izy} e^{i b \frac{z}{T}} \,dz \Big| + \frac{1}{\pi  } \sup_{y > 0 }   \Big|  \int_{ \mathcal{C}_{r}^{+} }
\overline f(z) e^{izy} e^{i b \frac{z}{T}} \,dz \Big| \notag \\ 
& \quad  +  \frac{1}{\pi }
\int_{ r \leq |t| \leq T } \big| \overline f(t) \big| \,dt + \frac{2}{\pi T} \int_{-T}^T  \big| t \overline f(t)  \big|  \,dt  +  \frac{C}{T}, 
\label{eqBEint001}
\end{align}
where
$\overline f(z) = \frac{f(z) - e^{- z^2 /2}}{z}$,
$b>0$ is a fixed constant, $\mathcal{C}_{r}^{-}$ and $\mathcal{C}_{r}^{+}$ 
are semicircles in the complex plane given by 
\begin{align*}
\mathcal{C}_{r}^{-} = \{ z \in \mathbb{C}: |z| = r, \Im z <0 \},   \quad
\mathcal{C}_{r}^{+} = \{ z \in \mathbb{C}: |z| = r, \Im z >0 \}. 
\end{align*}
Using the new smoothing inequality, together with the spectral gap property 
\eqref{Rnsw-Decomp}, leads to the estimation of the following integrals: 
\begin{align}
& \int_{ \mathcal{C}_{r}^{+} \cup \; \mathcal{C}_{r}^{-} }
\frac{ \kappa^{n}(z)M_{z} \varphi(x) - e^{- z^2 / 2 } }{ z }
e^{izy} e^{\pm i b \frac{z}{T}} \,dz,  \label{INT_infty002a} \\ 
& \int_{ \mathcal{C}_{r}^{+} \cup \; \mathcal{C}_{r}^{-} }
\frac{ L^{n}_{z} \varphi(x) }{ z } e^{izy} e^{\pm i b \frac{z}{T}} \,dz.  \label{INT_infty002b}
\end{align}
The integral \eqref{INT_infty002a} is handled by using the saddle point method 
choosing a suitable path for the integration in Section \ref{Thm-Edge-Expan-bb}, 
which is one of the challenging parts of the proof.  
For the integral \eqref{INT_infty002b} we use the facts that  
$| L^{n}_{z}\varphi(x) |$ decays exponentially fast as $n\to\infty$ 
and that $|\frac{e^{izy}}{z}| \leq \frac{1}{r}$ for $z\in \mathcal{C}_{r}^{-}$, $y\leq 0$
and $r = c\sqrt{n}$.
In contrast to \eqref{INT_infty001}, the intergral \eqref{INT_infty002b} is bounded  by $Ce^{-cn}$ uniformly in $y$.
The case $y>0$ is treated similarly, which allows us to establish \eqref{Intro-BerEs001}.
Note that the non-arithmeticity condition is not needed for the validity of \eqref{Intro-BerEs001}. 
Under the non-arithmeticity condition, in Theorem \ref{Thm-Edge-Expan:s=0} 
we obtain an Edgeworth expansion for $(X_n^x, \sigma(G_n, x) )$ with the target function $\varphi$ on  $X_n^x$,
which is of independent interest.

\section{Main results} \label{SecMain}
\subsection{Notation and conditions}  \label{SecCondition}

Let $\mathbb N = \{0, 1,2,\ldots\}$ and $\mathbb{N}^* = \mathbb N \!\setminus\! \{0\}$.
The real part,  imaginary part and  the conjugate  of a complex number $z$ are denoted by $\Re z$, $\Im z$ and $\overline z$ respectively.
For $y \in \mathbb{R}$, we write $\phi(y) = \frac{1}{\sqrt{2\pi}}  e^{- y^2 / 2 }$
and $\Phi(y) = \int_{-\infty}^y \phi(t) \,dt$. 
For any $\eta>0$, set $B_\eta(0) = \{ z \in \mathbb{C}: |z| < \eta \}$
for the ball with center $0$ and radius $\eta$ in the complex plane $\mathbb{C}$.
We denote by $c$, $C$, 
positive constants whose values may change from line to line.
By $c_\alpha$, $C_{\alpha}$ we mean positive constants depending only on the index $\alpha.$
We write $\mathds{1}_A$ for the indicator function of an event $A$. 
For a measure $\nu$ and a function $\varphi$ we denote $\nu(\varphi) = \int \varphi \,d\nu$.

For $d\geq 2$, let $M(d,\mathbb{R})$ be the set of $d\times d$ matrices  with entries in $\mathbb R$.
We shall work with products of invertible or non-negative matrices. 
Denote by $\mathscr G = GL(d,\mathbb R)$ the group of invertible matrices of $M(d,\mathbb{R})$. 
A non-negative matrix $g\in M(d,\mathbb{R})$ is said to be \emph{allowable},
if every row and every column of $g$ has a strictly positive entry.
Denote by $\mathscr G_+$ the multiplicative semigroup of allowable non-negative matrices of 
$M(d,\mathbb{R})$,
which will be called simply positive.
We write $\mathscr G_+^\circ $ for the subsemigroup of $\mathscr G_+$ with strictly positive entries.

The space $\mathbb{R}^d$ is equipped  with  any given norm $|\cdot|$. 
Let $\mathbb{P}^{d-1} = \{x=\mathbb Rv:  v \in \mathbb{R}^{d}\setminus\{0\}\}$ 
be the projective space of $\mathbb{R}^{d}$. 
Let $\mathbb{R}^d_+$ be the positive quadrant of $\mathbb{R}^d$, 
and $\mathbb{P}^{d-1}_{+} = \{x = \mathbb R v: v \in \mathbb{R}^{d}_+\setminus\{0\} \}$
be the set of directions corresponding to nonzero vectors in $\mathbb{R}^d_+$.
To unify the exposition, we use the symbol $\mathcal{S}$ to denote
$\mathbb{P}^{d-1}$ in case of invertible matrices
and $\mathbb{P}^{d-1}_{+}$ in case of positive matrices.
For any matrix $g$ in $\mathscr G$ or $\mathscr G_+$ and $x = \mathbb R  v \in \mathcal S$, 
we write $g \cdot x = \mathbb R g v$ for the projective action of $g$ on $\mathcal{S}$.  
The space $\mathcal S$ is endowed with the metric $\mathbf d$:
for invertible matrices, $\mathbf{d}$ is the angular distance, i.e.,
for any $x=\mathbb R v, y=\mathbb R u \in \mathbb{P}^{d-1}$, $\mathbf{d}(x,y)= |\sin \theta(v,u)|$,
where $\theta(v,u)$ is the angle between $v$ and $u$; 
for positive matrices, 
$\mathbf{d}$ is the Hilbert cross-ratio metric, i.e.,
for any $x = \mathbb R v \in \mathbb P^{d-1}_+$ and $y = \mathbb R u \in \mathbb P^{d-1}_+$ with $|v| = |u| =1$,
$\mathbf{d}(x,y) = \frac{ 1- m(v,u)m(u,v) }{ 1 + m(v,u)m(u,v) }$, 
where $m(v,u)=\sup\{\alpha >0 :  \alpha u_i \leq v_i,  \forall i=1,\ldots, d    \}.$
In both cases, 
there exists a constant $C>0$ such that 
\begin{align} \label{Ineq-Distan}
|v-u|\leq C \mathbf{d}(x,y),\quad \mbox{for any} \ x = \mathbb R v, \,  y = \mathbb R u \in \mathcal{S}, \,  |v| = |u| =1. 
\end{align}
We refer to \cite{GL16} and \cite{Hen97} for more details of the metric $\mathbf{d}$.

Let $\mathcal{C}(\mathcal{S})$ be the space of continuous complex-valued functions on $\mathcal{S}$
and $\mathbf{1}$ be the constant function with value $1$.
Let $\gamma>0$.  For any $\varphi\in \mathcal{C(S)}$, set
\begin{align*}
\lVert \varphi \rVert_{\gamma}:= \lVert \varphi \rVert_{\infty} + [\varphi]_{\gamma}, \ \  
\lVert \varphi \rVert_{\infty}:= \sup_{x\in \mathcal{S}}|\varphi(x)|,  \ \
[\varphi]_{\gamma}: = \sup_{x,y\in \mathcal{S}}\frac{|\varphi(x)-\varphi(y)|}{\mathbf{d}^{\gamma }(x,y)}.
\end{align*}
 Introduce the Banach space
$\mathcal{B}_{\gamma}:= \{ \varphi\in \mathcal{C(S)}: \lVert \varphi \rVert_{\gamma}< + \infty\}.$

Assume that on some probability space $(\Omega, \mathscr{F}, \mathbb{P})$ we are given 
a sequence of i.i.d.\  random matrices  $(g_{n})_{n\geq 1}$ of the same law $\mu$ 
on $\mathscr G$ or $\mathscr G_+$. 
Set $G_n = g_{n} \ldots g_{1}$, $n \geq 1$, then for any starting point $x \in \mathcal{S}$, the process
\begin{align*}
X_0^x=x, \quad X_{n}^x = G_n \!\cdot\! x, \quad n\geq 1 
\end{align*}
forms a Markov chain on $\mathcal{S}$.
Let  $\sigma(g,x)= \log \frac{|g v|}{|v|}$ be the norm cocycle, 
where $g\in \mathscr G$ and $x = \mathbb R  v \in \mathbb P^{d-1} $ or $g\in\mathscr G_+$ and $x = \mathbb R  v \in \mathbb P_+^{d-1} $.
The goal of the present paper is to establish a Berry--Esseen bound  
and a Cram\'{e}r type moderate deviation expansion
for the couple $(X_n^x, \sigma(G_n, x))$ with a target function $\varphi$ on the Markov chain $(X_n^x)$, 
for both invertible matrices and positive matrices.

For any $g \in M(d,\mathbb{R})$, set 
$\lVert g\rVert = \sup_{v \in \mathbb R^d \setminus \{0\} } \frac{|g v|}{|v|}$ and 
$\iota(g) = \inf_{v \in \mathbb R^d \setminus \{0\} } \frac{|g v|}{|v|}$, 
where $\iota(g) >0$ for both $g \in \mathscr G$ and $g \in \mathscr G_+$.
In the following we denote $N(g) = \max \{ \lVert g \rVert, \iota(g)^{-1} \}$.
From the Cartan decomposition it follows that the norm $\lVert g\rVert$ coincides with the largest singular value of $g$,
i.e.\ $\lVert g \rVert$ is the square root of the largest eigenvalue of $g^{\mathrm{T}} g$, 
where $g^{\mathrm{T}}$ denotes the transpose of $g$.
For an invertible matrix $g \in \mathscr G$, $\iota(g) = \lVert g^{-1} \rVert^{-1}$, 
hence $\iota(g)$ is the smallest singular value of $g$ and 
 $N(g) = \max \{ \lVert g \rVert, \lVert g^{-1} \rVert \}$.
We need the two-sided exponential moment condition:
\begin{conditionA}\label{CondiMoment}
There exists a constant $\eta_0 \in (0,1)$ such that $\mathbb{E} [N(g_1 )^{\eta_0}] < +\infty$.
\end{conditionA}

We denote by $\Gamma_{\mu}:=[\supp\mu]$
the smallest closed subsemigroup of $M(d,\mathbb{R})$ generated by $\supp \mu$, 
the support of the measure $\mu$.

For invertible matrices, we need the strong irreducibility and proximality conditions.
Recall that a matrix $g$ is said to be \emph{proximal}
if $g$ has an eigenvalue $\lambda_{g}$ satisfying $|\lambda_{g}| > |\lambda_{g}'|$
for all other eigenvalues $\lambda_{g}'$ of $g$.
The normalized eigenvector $v_g$ ($|v_g| = 1$) corresponding to the eigenvalue $\lambda_{g}$
is called the dominant eigenvector. It is easy to verify that $\lambda_{g} \in \mathbb{R}$.

\begin{conditionA}\label{CondiIP}
{\rm (i)(Strong irreducibility) }
No finite union of proper subspaces of $\mathbb{R}^d$ is $\Gamma_{\mu}$-invariant.

\  {\rm (ii)(Proximality) } $\Gamma_{\mu}$ contains at least one proximal matrix. 
\end{conditionA}

For positive matrices, we use the allowability and positivity conditions:
\begin{conditionA}\label{CondiAP}
{\rm (i) (Allowability) }
Every $g\in\Gamma_{\mu}$ is allowable.

\ {\rm (ii) (Positivity) }
$\Gamma_{\mu}$ contains at least one matrix belonging to $\mathscr G_+^\circ$.
\end{conditionA}

It follows from the Perron-Frobenius theorem that every  $g \in \mathscr G_+^\circ$
has a dominant eigenvalue $\lambda_g>0$, with the corresponding eigenvector $v_g \in \mathbb{P}_+^{d-1}$.

Under conditions \ref{CondiMoment} and \ref{CondiIP}
for invertible matrices, or conditions \ref{CondiMoment} and \ref{CondiAP} for positive matrices,
there exists a unique $\mu$-stationary probability measure $\nu$ on $\mathcal{S}$ (\cite{GL16, BDGM14}): 
for any $\varphi \in \mathcal{C(S)}$,
\begin{align} \label{mu station meas}
(\mu*\nu)(\varphi ) = \int_{\mathcal{S}} \int_{\Gamma_{\mu}} \varphi(g_1\!\cdot\! x) \, \mu(dg_1)  \, \nu(dx)
= \int_{\mathcal{S}} \varphi(x) \, \nu(dx) = \nu(\varphi). 
\end{align}
Moreover, for invertible matrices, $\supp \nu$ (the support of $\nu$) is given by
\begin{equation} \label{def-VGamma-inv}
V(\Gamma_{\mu})=\overline{\{ v_{g}\in \mathbb P^{d-1}:  g\in\Gamma_{\mu}, \ g \mbox{ is proximal} \}}; 
\end{equation}
for positive matrices, $\supp \nu$ is given by
\begin{equation} \label{def-VGamma-pos}
V(\Gamma_{\mu}) = \overline{\{v_{g} \in \mathbb P^{d-1}_+: g\in \Gamma_{\mu} \cap \mathscr G_+^{\circ} \}}. 
\end{equation}
In addition, for both cases, $V(\Gamma_{\mu})$ is the unique minimal $\Gamma_{\mu}$-invariant subset
(see \cite{GL16} and \cite{BDGM14}).

For positive matrices, 
it will be shown in Proposition \ref{Prop-lambTaylor} that under conditions \ref{CondiMoment} and \ref{CondiAP}, 
the asymptotic variance 
\begin{align*}
\sigma^2 = \lim_{n\to\infty} \frac{1}{n} \mathbb{E} \Big[ (\sigma (G_n, x) - n\lambda)^{2} \Big]
\end{align*}
exists with value in $[0,\infty)$.
To establish the Berry--Esseen theorem and the moderate deviation expansion,
we need the following condition:
\begin{conditionA}\label{Condi-Variance}
The asymptotic variance $\sigma^2$ satisfies $\sigma^2>0$.
\end{conditionA}

We say that the measure $\mu$ 
is \emph{arithmetic}, if there exist $t>0$, $\beta \in[0,2\pi)$ and a function
$\vartheta: \mathcal{S} \to \mathbb{R}$ such that
$\exp[it\sigma(g, x)-i\beta + i\vartheta(g\!\cdot\!x)-i \vartheta(x)]=1 $
for any $g\in \Gamma_{\mu}$ and $x\in V(\Gamma_{\mu})$.
To establish the Edgeworth expansion for positive matrices, we impose the following condition:
\begin{conditionA}\label{CondiNonarith}
{\rm (Non-arithmeticity) }
The measure $\mu$ is non-arithmetic.
\end{conditionA}
A simple sufficient condition introduced in \cite{Kes73} for the measure $\mu$
to be non-arithmetic is that the additive subgroup of $\mathbb{R}$ generated by the set
$\{ \log \lambda_{g} : g \in \Gamma_{\mu} \cap \mathscr G_+^\circ \}$
is dense in $\mathbb{R}$, see \cite[Lemma 2.7]{BM16}.

We end this subsection by giving some implications  among the above conditions.
For invertible matrices, it was proved in \cite{GU05} that condition \ref{CondiIP} implies
condition \ref{CondiNonarith}.
For positive matrices, 
conditions \ref{CondiMoment}, \ref{CondiAP} and \ref{CondiNonarith} imply condition \ref{Condi-Variance},
see Proposition \ref{Prop-lambTaylor}. 


\subsection{Berry--Esseen bound and Edgeworth expansion}
In this subsection we formulate the Berry--Esseen theorem and the Edgeworth expansion   
for the couple $(X_n^x, \sigma(G_n, x))$.
We first state the Berry--Esseen theorem with a target function on $X_n^x$.
Through the rest of the paper 
we assume that $\gamma>0$ is a fixed small enough constant 
so that the spectral properties stated in Proposition \ref{transfer operator} hold true. 

\begin{theorem}\label{Thm-BerryEsseen-Norm:s=0}
Assume either conditions \ref{CondiMoment} and \ref{CondiIP} for invertible matrices,
or conditions \ref{CondiMoment}, \ref{CondiAP} and \ref{Condi-Variance} for positive matrices.
Then, there exists a constant 
$C>0$ such that for all $n \geq 1$, 
$x \in \mathcal{S}$, $y \in \mathbb{R}$ and   $\varphi \in \mathcal{B}_{\gamma}$, 
\begin{align}\label{Th-BerEssMatr001:s=0} 
\Big| \mathbb{E}
\Big[  \varphi(X_n^x) \mathds{1}_{ \big\{ \frac{\sigma(G_n, x) - n \lambda }{\sigma \sqrt{n}} \leq y \big\} } \Big]
-  \nu(\varphi)  \Phi(y) \Big|   
 \leq  \frac{C}{\sqrt{n}} \lVert \varphi \rVert_{\gamma}. 
\end{align}
\end{theorem}

The proof of this theorem  
follows the same line as the proof of the Edgeworth expansion in Theorem \ref{Thm-Edge-Expan:s=0}  
formulated below, and 
will be sketched at the end of Section \ref{sec:proof berry-esseen-edgeworth}.
The presence of the target function in Theorem \ref{Thm-BerryEsseen-Norm:s=0} turns out to be crucial
in the study of the asymptotic of moderate deviations of the logarithm of the coefficients
$\log | \langle f, G_n v \rangle |$ with $f \in (\mathbb R^d)^*$ and $v \in \mathbb R^d$, 
which will be done in a forthcoming paper.

Theorem \ref{Thm-BerryEsseen-Norm:s=0}
extends the Berry--Esseen bounds 
 from \cite{LeP82, Jan01} for invertible matrices,
and \cite{HH04} for positive matrices to versions with target functions on $X_n^x$.
Note that  
the results in  \cite{Jan01, HH04}   have been  established under some polynomial moment conditions.
However, 
proving \eqref{Th-BerEssMatr001:s=0} with the target function $\varphi \not= \mathbf{1}$ under the polynomial moments is still an open problem.

The next result gives an Edgeworth expansion for $\sigma(G_n, x)$
with a target function $\varphi$ on $X_n^x$.
To formulate it, we introduce the necessary notation.
Consider the following transfer operator:
for any $s \in (-\eta, \eta)$ with $\eta>0$ small, and $\varphi \in \mathcal{C}(\mathcal{S})$,
\begin{align*}
P_s \varphi(x) = \mathbb E \Big[ e^{s \sigma(g_1, x)} \varphi( g_1 \! \cdot \! x ) \Big],  
\quad  x \in \mathcal{S}.  
\end{align*}
It will be shown in Proposition \ref{transfer operator} that
there exists a unique H\"{o}lder continuous function $r_s$ on $\mathcal{S}$ such that
\begin{align} \label{Def-kappa01} 
P_s r_s = \kappa(s) r_s,
\end{align}
where $\kappa(s)$ is the unique dominant eigenvalue of $P_s$.  
Set $\Lambda(s)=\log \kappa(s)$.  
We shall show in Lemma \ref{Lem-Bs} that for any 
$\varphi \in \mathcal{B}_{\gamma}$, the function
\begin{align}\label{func-phi-001}
b_{\varphi}(x) = \lim_{n \to \infty}
\mathbb{E} \big[ ( \sigma(G_n, x) - n \lambda ) \varphi(X_n^x) \big],
\quad   x \in \mathcal{S}
\end{align}
is well defined, belongs to $\mathcal B_{\gamma}$ and  
has an equivalent expression \eqref{Def2-bs}  
in terms of derivative of 
the projection operator $\Pi_{0,z}$, see Proposition \ref{perturbation thm}.

\begin{theorem}\label{Thm-Edge-Expan:s=0}
Assume either conditions \ref{CondiMoment} and \ref{CondiIP} for invertible matrices,
or conditions \ref{CondiMoment}, \ref{CondiAP} and \ref{CondiNonarith} for positive matrices.
Then, as $n \to \infty$, uniformly in 
$x \in \mathcal{S}$, $y \in \mathbb{R} $ and $\varphi \in \mathcal{B}_{\gamma}$,
\begin{align}
& \bigg| \mathbb{E}
   \Big[  \varphi(X_n^x) \mathds{1}_{ \big\{ \frac{\sigma(G_n, x) - n \lambda }{\sigma \sqrt{n}} \leq y \big\} } \Big]
   \label{MyEdgExp001} \\
&  \qquad  - \nu(\varphi) \Big[  \Phi(y) + \frac{\Lambda'''(0)}{ 6 \sigma^3 \sqrt{n}} (1-y^2) \phi(y) \Big]
  + \frac{ b_{\varphi}(x) }{ \sigma \sqrt{n} } \phi(y)   \bigg|
  =  \lVert \varphi \rVert_{\gamma} o \Big( \frac{ 1 }{\sqrt{n}} \Big). \notag 
\end{align}
\end{theorem}

The proof of this theorem 
is postponed to Section \ref{sec:proof berry-esseen-edgeworth} and
is based on
a new smoothing inequality (Proposition \ref{Prop-BerryEsseen})
and the saddle point method.
Even for $\varphi = \mathbf{1}$, Theorem \ref{Thm-Edge-Expan:s=0} is new.


\subsection{Moderate deviation expansions}  \label{SecNormCocy}

Denote $\gamma_k = \Lambda^{(k)}(0)$, $k \geq 1$, where $\Lambda = \log \kappa$ with the function
$\kappa$ defined in \eqref{Def-kappa01}. 
In particular, $\gamma_1 =\lambda$ and $\gamma_2 = \sigma^2$, 
see Propositions \ref{Prop-LLN} and \ref{Prop-lambTaylor}, where we also give an expression for $\gamma_3.$
Throughout the paper, we write $\zeta$ for the Cram\'{e}r series of 
$\Lambda$ (see \cite{Cra38} and \cite{Pet75}): 
\begin{align}\label{Def-CramSeri}
\zeta(t)=\frac{\gamma_3}{6\gamma_2^{3/2} } + \frac{\gamma_4\gamma_2-3\gamma_3^2}{24\gamma_2^3}t
+ \frac{\gamma_5\gamma_2^2-10\gamma_4\gamma_3\gamma_2 + 15\gamma_3^3}{120\gamma_2^{9/2}}t^2 + \cdots, 
\end{align}
which converges for $|t|$ small enough.

Now we formulate a Cram\'{e}r type moderate deviation expansion for the couple $(X_n^x, \sigma(G_n, x) )$
with target function on $X_n^x,$ for both invertible matrices and positive matrices.

\begin{theorem} 
\label{MainThmNormTarget}
Assume either conditions \ref{CondiMoment} and \ref{CondiIP} for invertible matrices,
or conditions \ref{CondiMoment}, \ref{CondiAP} and \ref{Condi-Variance} for positive matrices.
Then, 
uniformly in $x\in \mathcal{S}$,  $y \in [0, o(\sqrt{n} )]$ and
$\varphi \in \mathcal{B}_{\gamma}$, 
as $n \to \infty$,
\begin{align*}
&  \frac{\mathbb{E}
\Big[ \varphi(X_n^x) \mathds{1}_{ \big\{ \sigma(G_n, x) - n\lambda \geq \sqrt{n}\sigma y \big\} } \Big] }
{1-\Phi(y)}
=  e^{ \frac{y^3}{\sqrt{n}}\zeta ( \frac{y}{\sqrt{n}} ) }
\Big[ \nu(\varphi) + \lVert \varphi \rVert_{\gamma} O\Big( \frac{y+1}{\sqrt{n}} \Big) \Big], \\
& \frac{\mathbb{E}
\Big[ \varphi(X_n^x) \mathds{1}_{ \big\{ \sigma(G_n, x) - n\lambda \leq - \sqrt{n}\sigma y  \big\} } \Big] }
{ \Phi(-y)    }
 =   e^{ - \frac{y^3}{\sqrt{n}}\zeta (-\frac{y}{\sqrt{n}} ) }
\Big[  \nu(\varphi)  + \lVert \varphi \rVert_{\gamma} O\Big( \frac{y+1}{\sqrt{n}} \Big) \Big]. 
\end{align*}
\end{theorem}

Note that the  above  asymptotic expansions remain valid even when $\nu(\varphi) = 0$. 
In this case, for example, the first expansion becomes, as $n\to\infty$,
\begin{align*} 
\mathbb{E} \Big[ \varphi(X_n^x) \mathds{1}_{ \big\{ \sigma(G_n, x) - n\lambda \geq  \sqrt{n}\sigma y  \big\} } \Big]  
= \big[ 1-\Phi(y) \big] e^{ \frac{y^3}{\sqrt{n}}\zeta ( \frac{y}{\sqrt{n}} ) }
\lVert \varphi \rVert_{\gamma} O\Big( \frac{y+1}{\sqrt{n}} \Big). 
\end{align*}
It is an open question to extend  
the results of Theorem  \ref{MainThmNormTarget}  
to higher order expansions under the additional condition of non-arithmeticity.   
We refer to Saulis \cite{Sau69} and  Rozovsky \cite{Roz86} 
for relevant results in the i.i.d.\ real-valued case.   
In the case of products of random matrices this problem seems to us interesting 
because of the presence of the derivatives in $s$ of the eigenfunction $r_s$ 
and of the linear functional $\nu_s$  in the higher order terms.

In particular, under conditions of Theorem \ref{MainThmNormTarget}, with $\varphi = \mathbf{1}$ we obtain:
as $n\to\infty$,
\begin{align}
\frac{ \mathbb{P} \Big(\frac{ \sigma(G_n, x) - n\lambda }{ \sigma \sqrt{n} } \geq y \Big)}{ 1-\Phi(y) }
& =    e^{ \frac{y^3}{\sqrt{n}} \zeta (\frac{y}{\sqrt{n}} ) }
\Big[ 1 + O\Big( \frac{y+1}{\sqrt{n}} \Big) \Big],  \notag \\
\frac{ \mathbb{P} \Big( \frac{ \sigma(G_n, x) -n\lambda }{ \sigma \sqrt{n} } \leq -y\Big)}{ \Phi(-y) }
& =    e^{ -\frac{y^3}{\sqrt{n}} \zeta (-\frac{y}{\sqrt{n}} ) }
\Big[ 1 + O\Big( \frac{y+1}{\sqrt{n}} \Big) \Big].    \notag
\end{align}
When $\varphi \in \mathcal{B}_{\gamma}$ is a real-valued function satisfying $\nu(\varphi) > 0$, 
Theorem \ref{MainThmNormTarget} clearly implies the following moderate deviation principle 
for $\sigma(G_n, x)$ with target function on $X_n^x$: 
for any Borel set $B \subseteq \mathbb{R}$,
and positive sequence $(b_n)_{n\geq 1}$ satisfying
$\frac{b_n}{n}\rightarrow0$ and $\frac{b_n}{\sqrt{n}}\to \infty$ as $n\to\infty$, 
uniformly in $x \in \mathcal{S}$, 
\begin{align} 
- \inf_{ y\in B^{\circ} } \frac{y^2}{2\sigma^2} 
& \leq \liminf_{n\to \infty} \frac{n}{b_n^{2}} \log \mathbb{E} 
 \Big[ \varphi(X_n^x)  \mathds{1}_{ \big\{ \frac{\sigma(G_n, x) - n\lambda }{b_n} \in B  \big\} }  \Big]  \label{intro MDP000}\\
& \leq  \limsup_{n\to \infty}\frac{n}{b_n^{2}} \log \mathbb{E}  
 \Big[   \varphi(X_n^x)  \mathds{1}_{ \big\{ \frac{\sigma(G_n, x) - n\lambda }{b_n} \in B  \big\} }   \Big] 
\leq - \inf_{y\in \bar{B}} \frac{y^2}{2\sigma^2},  \notag
\end{align}
where $B^{\circ}$ and $\bar{B}$ are respectively the interior and the closure of $B$.
In fact it is enough to show \eqref{intro MDP000} only for the case where $B$ is an interval, 
the result for general $B$ can be established using Lemma 4.4 of Huang and Liu \cite{HL12}.  
With $\varphi = \mathbf 1$, \eqref{intro MDP000} implies the moderate deviation principle \eqref{intro MDPintro000} 
established in \cite[Proposition 12.12]{BQ16b} for invertible matrices. 
The moderate deviation principle \eqref{intro MDP000} with target function on $X_n^x$ 
is new for both invertible matrices and positive matrices; 
\eqref{intro MDPintro000} is new for positive matrices.
Note that in \eqref{intro MDP000} the function $\varphi$ is not necessarily strictly positive.


\subsection{Local limit theorem with moderate deviations}  \label{SecModerLocal}
In this subsection we state a local limit theorem with moderate deviations for  $\sigma(G_n, x)$, 
 which is of independent interest and can not be deduced directly from Theorem \ref{MainThmNormTarget}.

\begin{theorem}\label{ThmLocal01}
Assume either conditions \ref{CondiMoment} and \ref{CondiIP} for invertible matrices,
or conditions \ref{CondiMoment}, \ref{CondiAP} and \ref{Condi-Variance} for positive matrices.
Then, for any $\varphi \in \mathcal{B}_{\gamma}$
and any directly Riemann integrable function $\psi$ with compact support on $\mathbb R$, 
we have, as $n\to\infty$, uniformly in $x \in \mathcal{S}$ and $|y| = o(\sqrt{n})$,
\begin{align*}
\mathbb{E} \Big[ \varphi(X_n^x) \psi \big( \sigma(G_n, x) - n\lambda - \sqrt{n}\sigma y \big) \Big]  
 = \frac{e^{ -\frac{y^2}{2} +  \frac{y^3}{\sqrt{n}}\zeta(\frac{y}{\sqrt{n}} ) }}{\sigma   \sqrt{2 \pi n}}  
  \Big[ \nu(\varphi) \int_{\mathbb R} \psi(u) \, du  +  o(1) \Big].
\end{align*}
In particular,
for any $\varphi \in \mathcal{B}_{\gamma}$ and real numbers $-\infty < a_1 < a_2 < \infty$, 
we have, as $n \to \infty$, uniformly in $x \in \mathcal{S}$ and $|y| = o(\sqrt{n})$,
\begin{align*}
\mathbb{E} \Big[ \varphi(X_n^x)
       \mathds{1}_{\{ \sigma(G_n, x) - n\lambda \in [a_1, a_2] + \sqrt{n}\sigma y \} } \Big]  
=  \frac{e^{ -\frac{y^2}{2} +  \frac{y^3}{\sqrt{n}}\zeta(\frac{y}{\sqrt{n}} ) }}{\sigma   \sqrt{2 \pi n}}  
  \Big[ (a_2 - a_1) \nu(\varphi)  +  o(1) \Big].
\end{align*}
With $\varphi = \mathbf{1}$, we have, as $n \to \infty$, 
uniformly in $x \in \mathcal{S}$ and $|y| = o(\sqrt{n})$, 
\begin{align*}
\mathbb{P} \Big( \sigma(G_n, x) - n\lambda \in [a_1, a_2] +  \sqrt{n}\sigma y \Big)
 =  \frac{ e^{ - \frac{y^2}{2} + \frac{y^3}{\sqrt{n}} \zeta(\frac{y}{\sqrt{n}}  )} }{ \sigma \sqrt{2 \pi n} }
\Big[ a_2 -  a_1 + o(1) \Big]. 
\end{align*}
\end{theorem}

In the case of invertible matrices,
a similar local limit theorem has been established in \cite{BQ16b} in a more general setting
and plays an important role in studying dynamics of group actions on finite volume homogeneous spaces, 
see \cite[Proposition 4.7]{BQ13}.
Specifically, from \cite[Theorem 17.10]{BQ16b}, 
by simple calculations we deduce that for any $a_1 < a_2$,
it holds uniformly in $x \in \mathbb{P}^{d-1}$ and 
$|y| = O(\sqrt{\log n})$
that, as $n \to \infty$, 
\begin{align}\label{BQLLT001} 
\mathbb{P} \Big( \sigma(G_n, x) - n\lambda \in [a_1, a_2] + \sqrt{n}\sigma y \Big)
 =  \frac{ e^{ - \frac{y^2}{2} } }{ \sigma \sqrt{2 \pi n} }
\Big[ a_2 -  a_1 + o(1) \Big]. 
\end{align}
Theorem \ref{ThmLocal01} extends the range of $y$ in \eqref{BQLLT001} 
beyond $O(\sqrt{\log n})$ and 
moreover, allows a target function $\varphi$ on the Markov chain $X_n^x$.
Note also that in \cite{BQ16b} the group
$SL(d, \mathbb{R})$ is considered instead of $GL(d, \mathbb{R})$, 
and the proximality condition \ref{CondiIP}(ii) is replaced by the condition that the semigroup $\Gamma_{\mu}$ is unbounded.  
For positive matrices, Theorem \ref{ThmLocal01} and its consequence \eqref{BQLLT001} are new. 

As an application of Theorem \ref{ThmLocal01}, 
we can establish a local limit theorem with moderate deviations
for the operator norm $\lVert G_n \rVert$ in the case of invertible matrices. 

\begin{theorem}\label{Thm_LLT_Norm_0a}
Assume conditions \ref{CondiMoment} and \ref{CondiIP} for invertible matrices. 
Let $-\infty < a_1 < a_2 < \infty$ be real numbers. 
Then, for any $\varphi \in \mathcal{B}_{\gamma}$,
we have, as $n \to \infty$, uniformly in $x \in \mathbb{P}^{d-1}$ and $|y|=o(n^{1/6})$,  
\begin{align*}
\mathbb{E} \Big[ \varphi(X_n^x)
       \mathds{1}_{\{ \log \lVert G_n \rVert - n\lambda \in [a_1, a_2] + \sqrt{n}\sigma y \} } \Big]  
= \frac{e^{ -\frac{y^2}{2}  }}{\sigma   \sqrt{2 \pi n}}  
  \Big[ (a_2 - a_1) \nu(\varphi)  +  o(1) \Big].
\end{align*}
With $\varphi = \mathbf{1}$, we have, as $n \to \infty$, 
uniformly in $x \in \mathbb{P}^{d-1}$ and $|y| = o(n^{1/6})$, 
\begin{align}\label{LLT_Norm_00a}
\mathbb{P} \Big(  \log \lVert G_n \rVert  - n\lambda \in  [a_1, a_2] + \sqrt{n}\sigma y \Big)  
= \frac{ e^{- \frac{y^2}{2} } }{ \sigma  \sqrt{2 \pi n} } \Big[ a_2 - a_1 + o(1) \Big].
\end{align}
\end{theorem}

In the smaller range $|y| = O(\sqrt{\log n})$,  
the result \eqref{LLT_Norm_00a} has been established
for the general framework of semisimple real Lie groups in \cite[Theorem 17.7]{BQ16b},
under some assumptions which reduce to ours for the general linear group $GL(d, \mathbb R)$.  
Thus Theorem \ref{Thm_LLT_Norm_0a} extends the results in \cite{BQ16b}
to the wider range $|y| =o(n^{1/6})$, 
and to the couple $(X_n^x, \log \lVert G_n \rVert)$ 
with a target function $\varphi$ on the Markov chain $X_n^x$. 
Note that it is an open question to establish local limit theorem with moderate deviation for $\log \lVert G_n \rVert$
in the whole range $|y| = o(\sqrt{n})$.


\section{Spectral gap theory} \label{sec:spec gap norm}
This section is devoted 
to investigating the spectral gap properties of some linear operators to be introduced below:
the transfer operator $P_z$,  its normalization $Q_s$ which is a Markov operator,  
and the perturbed operator $R_{s,z}$, for real-valued $s$ and complex-valued $z$.
The properties for these operators have been studied in recent years, 
for instance in  \cite{LeP82, BDGM14,GL16, BM16, BQ16b},
where various results have been established under different restrictions on $s$ and $z$,
which are not enough for obtaining the results of the paper.   
We shall complete these results by investigating the case
when $s\in (-\eta,\eta)$ with $\eta>0$ small, and $z$ belongs to a small ball of the complex plane centered at the origin.
The case of $s<0$ turns out to be more difficult than the case $s\geq 0$ and 
requires a deeper analysis. 
We also complement the previous results with some new properties 
to be used in the proofs of the main results of the paper. 


\subsection{Properties of the transfer operator $P_z$} \label{subsec a change of measure}

Recall that 
the Banach space $\mathcal{B}_{\gamma}$ consists of all the $\gamma$-H\"{o}lder continuous complex-valued
functions on $\mathcal{S}$.
We write $\mathcal{B}_{\gamma}'$ for the topological dual of $\mathcal{B}_{\gamma}$ endowed with the norm
$\lVert \nu \rVert_{ \mathcal{B}_{\gamma}^{\prime} }
= \sup_{\varphi \in \mathcal{B}_{\gamma}: \lVert \varphi \rVert_{\gamma} = 1} | \nu( \varphi ) |,$ for any linear functional 
$\nu \in \mathcal{B}_{\gamma}^{\prime }.$
Let $\mathcal{L(B_{\gamma}, B_{\gamma})}$ be
the set of all bounded linear operators from $\mathcal{B}_{\gamma}$ to $\mathcal{B}_{\gamma}$
equipped with the operator norm
$\lVert \cdot \rVert_{\mathcal B_{\gamma} \to \mathcal B_{\gamma}}$.
Denote by $\varrho(Q)$ the spectral radius of an operator $Q \in \mathcal{L(B_{\gamma},B_{\gamma})}$,
and by $Q|_E$ its restriction to the subspace $E \subseteq \mathcal B_{\gamma}$.

For any $z \in \mathbb{C}$ with $|z|<\eta_0$, where $\eta_0$ is given in condition \ref{CondiMoment},
define the transfer operator $P_z$ as follows: for any $\varphi \in \mathcal{C(S)}$, 
\begin{align} \label{def-Pz}
P_z \varphi(x) = \mathbb E \Big[ e^{z \sigma(g_1, x)} \varphi(g_1 \!\cdot\! x) \Big],
\quad  x \in \mathcal{S}.
\end{align}
The transfer operator $P_z$ 
acts from $\mathcal {C(S)}$ to the space of bounded functions on $\mathcal S.$
The proposition stated below gives the spectral gap properties of the operator $P_z$ for $z$ in a small enough neighborhood of $0$ in the complex plane.  
In the sequel, even if it is not stated explicitly, we assume that $\gamma>0$ is a sufficiently small constant.

\begin{proposition}  \label{transfer operator}
Assume that $\mu$ satisfies
either conditions \ref{CondiMoment} and \ref{CondiIP} for invertible matrices,
or conditions \ref{CondiMoment} and \ref{CondiAP} for positive matrices.
Then, $P_z \in \mathcal{L}(\mathcal{B}_\gamma,\mathcal{B}_\gamma)$ for any $z\in B_{\frac{\eta_0}{2} }(0)$, and
the mapping $z \mapsto P_z:$
$B_{\frac{\eta_0}{2} }(0) \to \mathcal{L}(\mathcal{B}_\gamma,\mathcal{B}_\gamma)$
is analytic for $\gamma >0$ small enough,
where $\eta_0 > 0$ is given in condition \ref{CondiMoment}.
Moreover, there exists a constant $\eta > 0$ such that for any
$z \in B_{\eta}(0)$ and $n \geq 1$, we have the decomposition
\begin{align}\label{Pzn-decom}
P_z^n = \kappa^n(z) M_z + L_z^n,
\end{align}
where the operator $ M_z : = \nu_z \otimes r_z$ is a rank one projection on $\mathcal{B}_\gamma$ defined by
$M_z \varphi = \frac{ \nu_z(\varphi) }{ \nu_z(r_z) }  r_z$
for any $\varphi \in \mathcal{B}_\gamma$, and the mappings on $B_{\eta}(0)$
\begin{align*}
z \mapsto \kappa(z) \in \mathbb{C}, \quad   z \mapsto r_z \in \mathcal{B}_{\gamma} ,
\quad   z \mapsto \nu_z \in \mathcal{B}_{\gamma}' ,
\quad   z \mapsto  L_z \in \mathcal{L}(\mathcal{B}_\gamma,\mathcal{B}_\gamma)
\end{align*}
are unique under the normalization conditions $\nu (r_z) =  1$  and $\nu_z ( {\bf 1} ) =1$, 
where $\nu$ is defined in \eqref{mu station meas}; 
all these mappings are analytic in $B_{\eta}(0)$, and possess the following properties:
\begin{itemize}
\item[{\rm(a)}]
for any $z \in B_{\eta}(0)$, it holds that $M_z L_z = L_z M_z =0$;
\item[{\rm(b)}] 
for any $z \in B_{\eta}(0)$,   
$P_z r_z = \kappa(z) r_z$ and $\nu_z P_z = \kappa(z) \nu_z$; 
\item[{\rm(c)}] 
$\kappa(0) = 1$, $r_0 = \mathbf{1}$, $\nu_0 = \nu$,   
    and $\kappa(s)$ and $r_s$ are real-valued and satisfy $\kappa(s)>0$ and $r_s(x)>0$ 
    for any $s \in (-\eta, \eta)$ and $x \in \mathcal{S}$;
\item[{\rm(d)}] 
for any $k\in \mathbb{N}$, 
  there exist constants $C_k >0$ and $0< a_1 < a_2< 1$ such that $|\kappa(z)| > 1 - a_1$  
  and  $\lVert \frac{d^k}{d z^k} L_z^n \rVert_{\mathcal B_{\gamma} \to \mathcal B_{\gamma}}
    \leq C_k(1 - a_2)^n$ for all $z \in B_{\eta }(0)$. 
\end{itemize}
\end{proposition}

Let us point out the differences between Proposition \ref{transfer operator} and the previous results in 
\cite{LeP82, BDGM14, BQ16b}.  
Firstly, we complement the results in \cite{LeP82, BQ16b} 
by giving the explicit formula 
$M_z \varphi = \frac{ \nu_z(\varphi) }{ \nu_z(r_z) }  r_z$ in \eqref{Pzn-decom}, 
for $z \in B_{\eta }(0)$,
which is  one of the crucial points in the proofs of the results of the paper. 
Basically, it permits us to deduce the spectral gap properties of the Markov operator $Q_s$ and 
as well as the perturbed operator $R_{s,z}$ from those of $P_z$. 
In particular, this will enable us to obtain an explicit formula for the operators $N_s$ and $N_{s,z}$ in 
Propositions \ref{ProSpecGapNega} and \ref{perturbation thm}, 
and the uniformity of the bounds \eqref{SpGapContrPi} and \eqref{SpGapContrN}. 
Secondly, 
for positive matrices, some points of Proposition \ref{transfer operator} have been obtained in
\cite{BDGM14} only for real $z \geq 0$.
The difficulty here is the case when $z \in \mathbb R$ is negative and when $z$ is not real,
so Proposition \ref{transfer operator} is new for positive matrices when $|z|\leq \eta.$
Thirdly, we show that $\kappa(z)$ and $r_z$ take real positive values when  $z$ is real, 
which allows to define the change of measure $\mathbb Q_s^x$ for real $s$, 
for both invertible matrices and positive matrices. 
Previously it was shown in \cite{BQ16b} that $\kappa(z)$ is real-valued for real $z\in (-\eta,\eta)$
for invertible matrices.  

\begin{remark}  \label{rem-decom-Pz*}
 Define the conjugate transfer operator $P_z^*$ by
\begin{align*}
P_z^* \varphi(x) =  \mathbb E \big[ e^{z \sigma (g_1^{\mathrm{T}}, x)} \varphi(g_1^{\mathrm{T}} \cdot x)  \big],
\quad  x \in \mathcal{S}^*,
\end{align*}
where $\mathcal{S}^*$ is the dual projective space of $\mathcal S$, 
$z \in \mathbb{C}$ with $\Re z \in (-\eta_0, \eta_0)$,
and $g_1^{\mathrm{T}}$ denotes the transpose of the matrix $g_1$.
One can verify that $P_z^*$ satisfies all the properties of Proposition \ref{transfer operator}: 
under conditions of Proposition \ref{transfer operator}, we have the decomposition
\begin{align}
P_z^{*n}  = \kappa^{*n} (z) \nu_z^* \otimes r_z^* + L_z^{*n},  \quad z \in B_{\eta}(0), \;  n \geq 1,
\end{align}
and all the assertions in Proposition \ref{transfer operator} hold for $P_z^*$,  $\kappa^*(z)$, $\nu_z^*$, $r_z^*$, $L_z^*$ instead of  $P_z$, $\kappa(z)$, $\nu_z$, $r_z$, $L_z$.
\end{remark}

\begin{proof}[Proof of Proposition \ref{transfer operator}]
We split the proof into three steps.
In steps 1 and 2 we concentrate on the case of positive matrices,
since for invertible matrices the results of these steps have been proved in \cite{LeP82, BQ16b}. 
In  step 1 we follow the same lines as in \cite{LeP82, BQ16b}.
In  step 2 we follow \cite{HH08} to prove the spectral gap property of the operator $P_0$ 
and we use the perturbation theory to extend it to $P_z$.    
In step 3 the proof is new and is provided for both invertible and positive matrices
by  complementing the results in \cite{LeP82, BDGM14, BQ16b}. 

\textit{Step 1.}
We only need to consider the case of positive matrices. 
We will show that
there exists $\gamma \in (0, \frac{\eta_0}{6})$ such that
$P_z \in \mathcal{L}(\mathcal{B}_\gamma,\mathcal{B}_\gamma)$,
and that the mapping  $z \mapsto P_z$ is   analytic on $ B_{\frac{\eta_0}{2} }(0)$.
For any 
$m\geq0$,
$z \in B_{\frac{\eta_0}{2} }(0)$ and $\varphi \in \mathcal{B}_{\gamma}$, let 
\begin{align*}
P^{(m)}_z \varphi(x) = \mathbb E \big[ (\sigma(g_1, x))^m e^{z \sigma(g_1, x)} \varphi(g_1 \!\cdot\! x)  \big], 
\quad  x \in  \mathbb{P}^{d-1}_{+}.
\end{align*}
It suffices to show that
for any $z \in B_{\frac{\eta_0}{2} }(0)$ and $\theta\in B_{\frac{\eta_0}{6} }(0)$,
\begin{align} \label{Panalytic01}
P_{z+\theta}=\sum_{m=0}^{\infty} \frac{\theta^m}{m!} P_{z}^{(m)},
\end{align}
and that there exists a constant $C>0$ not depending on $\theta$ and $z$ such that
\begin{align} \label{InequaAnaly}
\sum_{m=0}^{\infty}\frac{|\theta|^m}{m!} \lVert P^{(m)}_z \varphi \rVert_\gamma \leq C \lVert \varphi \rVert_\gamma.
\end{align}
From 
 \eqref{InequaAnaly} 
 we deduce that
  $P^{(0)}_z = P_z \in \mathcal{L}(\mathcal{B}_\gamma,\mathcal{B}_\gamma)$. 
Moreover, the bound \eqref{InequaAnaly} ensures the validity of \eqref{Panalytic01} 
which  implies  the analyticity of the  mapping  $z \mapsto P_z$  on $ B_{\frac{\eta_0}{2} }(0)$.

It remains to prove \eqref{InequaAnaly}.
We first give a control of $\lVert P^{(m)}_z \varphi \rVert_{\infty}$.
Since   $|\sigma(g,x) |\leq \log N(g)$ for any $g \in \Gamma_{\mu}$ and $x\in \mathbb{P}^{d-1}_{+}$,
 we get
\begin{align}\label{sum infinity norm}
\sum_{m=0}^{\infty} \frac{ |\theta|^m }{m!} \lVert P^{(m)}_z \varphi \rVert_{\infty}
\leq \lVert \varphi \rVert_{\infty}  
 \mathbb E \Big[ e^{( |\theta| +|\Re z|)\log N(g_1)}  \Big]
\leq C\lVert \varphi \rVert_{\infty}.
\end{align}
To control
$[ P^{(m)}_z \varphi ]_{\gamma}$,
note that for any 
 $\varphi \in \mathcal{B}_{\gamma}$,
\begin{align}\label{I holder expan}
[ P^{(m)}_z \varphi ]_{\gamma} 
\leq   &  \  \sup_{x, y\in \mathbb{P}^{d-1}_{+}, x\neq y} \bigg|  
 \mathbb E \bigg[
\frac{( \sigma(g_1, x) )^m - ( \sigma(g_1, y) )^m}{\mathbf{d}^{\gamma}(x,y)} e^{z \sigma(g_1, x)} \varphi(g_1 \!\cdot\! x)
\bigg] \bigg| \notag\\
 &  \  +  \sup_{x, y\in \mathbb{P}^{d-1}_{+}, x\neq y}  \bigg|  
\mathbb E \bigg[ ( \sigma(g_1, y) )^m
\frac{e^{z \sigma(g_1, x)} - e^{z \sigma(g_1, y)}}{\mathbf{d}^{\gamma}(x,y)} \varphi(g_1 \!\cdot\! x) \bigg] \bigg| \notag\\
 &  \  +  \sup_{x, y\in \mathbb{P}^{d-1}_{+}, x\neq y}  \bigg| 
 \mathbb E \bigg[ (\sigma(g_1, y))^m  e^{z \sigma(g_1, y)}
\frac{\varphi(g_1 \!\cdot\! x)- \varphi(g_1 \!\cdot\! y)}{\mathbf{d}^{\gamma}(x,y)} \bigg] \bigg|  \notag\\
=: & \  I_{1,m}+I_{2,m}+I_{3,m}.
\end{align}
We then control each of the three terms  $ I_{1,m}, I_{2,m}, I_{3,m}$.

\textit{Control of $I_{1,m}$.}
Since for any $a, b\in \mathbb{C}$,  $m\in \mathbb{N}$ and $0< \gamma <1$,
\begin{align}\label{basic inequ1}
|a^m-b^m|\leq  2 m \max \{|a|^{m-\gamma}, |b|^{m-\gamma}\} |a-b|^\gamma,
\end{align}
we get
\begin{align*}
I_{1,m} \leq    2 m \lVert \varphi \rVert_{\infty}   \sup_{x, y\in \mathbb{P}^{d-1}_{+}, x\neq y}
\mathbb E \bigg[  \frac{ (\log N(g_1))^{m-\gamma} N(g_1)^{|\Re z|} }{\mathbf{d}^{\gamma}(x,y)}
| \sigma(g_1, x) - \sigma(g_1, y) |^\gamma  \bigg].
\end{align*}
Using \eqref{Ineq-Distan}, 
we deduce that for any $g \in \Gamma_{\mu}$,
\begin{align} \label{InelogGx}
\big| \sigma(g, x) - \sigma(g, y) \big|
\leq  C \lVert g\rVert \iota(g)^{-1} \mathbf{d}(x,y),
\end{align}
and hence
\begin{align} \label{sum I1m}
  \sum_{m=0}^{\infty} \frac{ |\theta|^m }{m!}  I_{1,m}  
\leq   2  \lVert \varphi \rVert_{\infty} 
  \mathbb E \Big[ ( \log N(g_1) )^{1-\gamma} e^{( |\theta| + |\Re z| + 2\gamma)\log N(g_1)}  \Big].  
\end{align}

\textit{Control of $I_{2,m}$.}
Using \eqref{basic inequ1}, 
we deduce that for any $z_1, z_2\in \mathbb{C}$,
\begin{align}\label{TO-Ine-ex}
|e^{z_1}-e^{z_2}|\leq
2 \max \{|z_1|^{1-\gamma}, |z_2|^{1-\gamma}\} \max\{e^{\Re z_1}, e^{\Re z_2}\} |z_1-z_2|^\gamma.
\end{align}
By this inequality, we find that for any $g \in \Gamma_{\mu}$,
\begin{align*}
\big| e^{z \sigma(g, x)}-e^{z \sigma (g,y)} \big|
\leq 2 (\log N(g))^{1-\gamma}e^{|\Re z|\log N(g)} |\sigma(g, x) - \sigma(g, y) |^\gamma .
\end{align*}
Combining this with \eqref{InelogGx} implies that
\begin{align}\label{sum I2m}
\sum_{m=0}^{\infty} \frac{ |\theta|^m }{m!}  I_{2,m}  
\leq     2 \lVert \varphi \rVert_{\infty} 
 \mathbb E \Big[ (\log N(g_1))^{1-\gamma}e^{( |\theta| + |\Re z| + 2\gamma)\log N(g_1)}  \Big]. 
\end{align}

\textit{Control of $I_{3,m}$.}
Since $\varphi \in \mathcal{B}_{\gamma}$ and $\mathbf{d}(g \!\cdot\! x, g \!\cdot\! y) \leq \mathbf{d}(x,y)$
for any $g \in \Gamma_{\mu}$,
we get
\begin{align*} 
\sum_{m=0}^{\infty}\frac{ |\theta|^m }{m!} I_{3,m}
\leq   \lVert \varphi \rVert_{\gamma}
\mathbb E \Big[ e^{(|\theta| + |\Re z| + 2\gamma)\log N(g_1)}  \Big].
\end{align*}
Combining this with  \eqref{sum infinity norm},
\eqref{I holder expan}, \eqref{sum I1m} and  \eqref{sum I2m},  
we obtain \eqref{InequaAnaly}.

\textit{Step 2.}
Again we only need to consider the case of positive matrices. 
We will prove the decomposition formula \eqref{Pzn-decom} together with parts (a), (b) and (d). 
Our proof follows closely \cite{HH08}.
Define the operator $M$ on $\mathcal{B}_{\gamma}$ by  $M \varphi = \nu(\varphi) \mathbf{1}$,
$\varphi \in \mathcal{B}_{\gamma}$.
Set $E = \ker M \cap \mathcal{B}_{\gamma}$.
We first  show that
$\lVert \varphi \rVert_{\infty}  \leq [\varphi]_{\gamma}$ for any $\varphi \in E$.
Since $\nu(\varphi)=0$ for any $\varphi \in E$,
there exist $x_1, x_2 \in \mathbb{P}^{d-1}_{+}$ such that $\Re \varphi (x_1) = \Im \varphi(x_2) = 0$.
Since $\mathbf{d}(x,y) \in [0,1]$,  it follows that
\begin{align}\label{CompaNorm}
\lVert \varphi \rVert_{\infty}
   \leq 
\sup_{x\in \mathbb{P}^{d-1}_{+} }  | \Re \varphi(x) - \Re \varphi (x_1) |
   +   \sup_{x\in \mathbb{P}^{d-1}_{+} }  | \Im \varphi(x) - \Im \varphi (x_2) |  
\leq  2 [\varphi]_{\gamma}.
\end{align} 
We next show that $\varrho(P|_{E})<1,$ where $P=P_0$ (see \eqref{def-Pz}).
For any $x,y \in \mathbb{P}^{d-1}_{+}, x\neq y$, and $\varphi \in \mathcal{B}_{\gamma}$,
there exists $a \in (0,1) $ such that  for large $n \geq 1$,
\begin{align*} 
\frac{| P^n \varphi(x) - P^n \varphi(y)|}{\mathbf{d}^{\gamma}(x,y)}
\leq  \lVert \varphi \rVert_{\gamma } \mathbb E \Big[
\frac{ \mathbf{d}^{\gamma} (G_n \!\cdot\! x, G_n \!\cdot\! y)}{ \mathbf{d}^{\gamma}(x,y)}  \Big] 
\leq  \lVert \varphi \rVert_{\gamma }  a^n,
\end{align*}
where for the last inequality we use \cite[Lemma 3.2]{Hen97}.
Observe that for any $\varphi \in \mathcal{B}_{\gamma}$, we have $\varphi - M \varphi \in E $,
thus  $P^n( \varphi - M \varphi )\in E$ for any $n \geq 1$ since $\nu P = \nu$.
Combining this with \eqref{CompaNorm} and the above inequality,   
we get
\begin{align*}
\lVert P^n(\varphi - M \varphi ) \rVert_{\gamma}
\leq    2[P^n(\varphi - M \varphi )]_{\gamma}
\leq     2 a^n  [\varphi]_{\gamma}
\leq    2 a^n  \lVert \varphi \rVert_{\gamma},
\end{align*}
which implies $\varrho(P|_{E})<1.$
This, together with the definition of $E$ and
the fact that $P \mathbf{1} = \mathbf{1}$, shows
that $1$ is the isolated dominant eigenvalue of the operator $P$.
Using this and the analyticity of $P_z \in \mathcal{L}(\mathcal{B}_\gamma,\mathcal{B}_\gamma)$ shown in step 1,
and applying the perturbation theorem (see \cite[Theorem III.8]{HH01}),
we obtain the decomposition formula \eqref{Pzn-decom}  with $M_z(\varphi) = c_1 \nu_z(\varphi) r_z$
for some constant $c_1 \neq 0$,
as well as parts (a), (b) and (d).
 Using $P_z r_z = \kappa(z) r_z$, we get $c_1 = 1/\nu_z(r_z)$
and thus $M_z \varphi = \frac{ \nu_z(\varphi) }{ \nu_z(r_z) }  r_z$
for any $\varphi \in \mathcal{B}_\gamma$.

\textit{Step 3.} We prove part (c) for both invertible matrices and positive matrices.
From $P \mathbf{1} = \mathbf{1}$, we see  that $\kappa(0) = 1$ and $r_0 = \mathbf{1}$.
Letting $z=0$ in $\nu_z P_z = \kappa(z) \nu_z$, we get
$\nu_0 P = \nu_0$ and thus $\nu_0 = \nu$ since $\nu$ is the unique $\mu$-stationary probability measure.
Now we fix $z \in (-\eta, \eta)$ and
we show that $\kappa(z)$ and $r_z$ are real-valued. 
Taking the conjugate in the equality  $P_z r_z = \kappa(z) r_z$, we get 
$P_{ z } \overline{ r_z }  = \overline{ \kappa(z)} \overline{ r_z}$,
so that $ \overline{ \kappa(z)} $ is an eigenvalue of the operator $P_{ z }$.
By the uniqueness of the dominant eigenvalue of $P_z$, it follows that
$\overline{ \kappa(z) } = \kappa( z)$, showing that    $\kappa(z)$ is real-valued for $z \in  (-\eta, \eta)$.
We now  prove that
 $r_z$ is real-valued.   Write $r_z$ in the form
 $r_z = u_z + i v_z$, where $u_z$ and $v_z$ are real-valued functions on $\mathcal{S}$.
From the normalization condition $\nu(r_z) =1$, we get
 $\nu(u_z) = 1$ and $\nu(v_z) = 0$.  From the equation
$P_z r_z = \kappa(z) r_z$ and the fact that  $\kappa(z)$ is real-valued,  
we get that $P_z u_z = \kappa(z) u_z$ and $P_z v_z = \kappa(z) v_z$.
By part  (a), the space of eigenvectors corresponding to the eigenvalue  $\kappa(z) $ is one dimensional. Therefore,
we have either $u_z = c v_z$ for some constant   $c \in \mathbb{R} $, or $v_z=0$. However, the equality $u_z = c v_z$ is impossible because we have seen that  $\nu(u_z) = 1$ and $\nu(v_z) = 0$.
Hence  $v_z = 0$ and $r_z$ is real-valued for $z \in (-\eta, \eta)$.
The positivity of $\kappa(z)$ and $r_z$ then follows from $\kappa(0) = 1$, $r_0 = \mathbf{1}$
and the analyticity of the  mappings $z \mapsto \kappa(z)$ and $z \mapsto r_z$.
This ends the proof of part (c), as well as
the proof of Proposition \ref{transfer operator}. 
\end{proof}


\subsection{Definition of the change of measure \texorpdfstring{$\mathbb Q_s^x$}{Qsx}} 
\label{sec:Def change mes}

Proposition \ref{transfer operator} allows us to perform a change of measure.
Note that this change of measure for positive $s$
has been studied in \cite{BDGM14,BM16,GL16};
however, for negative $s$ it is new.
For any $s \in (-\eta, \eta)$, $x\in\mathcal{S}$ and $g \in \Gamma_{\mu}$, denote
\begin{align}\label{Def-qns}
q_{n}^{s}(x,g)=\frac{e^{s \sigma (g,x)}}{\kappa^{n}(s)}\frac{r_{s}(g \cdot x)}{r_{s}(x)},  \quad  n \geq 1.
\end{align}
Then $(q_n^s)$ satisfies the cocycle property: for any $n, m \geq 1$ and $g_1, g_2 \in \Gamma_{\mu}$,
\begin{align}\label{CocycleEqua}
q_{n}^{s}(x,g_1) q_{m}^{s}( g _1 \!\cdot\! x, g_2) = q_{n+m}^{s}(x,g_2g_1).
\end{align}
Since $\kappa(s)$ and $r_s$ are strictly positive,
$q_{n}^{s}(x, G_n )\mu(dg_1)\dots\mu(dg_n),$ $n\geq 1,$
is a sequence of probability measures,
and forms a projective system
on $M(d,\mathbb{R})^{\mathbb{N}^*}$. 
By the Kolmogorov extension theorem,
there is a unique probability measure  $\mathbb Q_s^x$ on $M(d,\mathbb{R})^{\mathbb{N}^*}$ 
with marginals $q_{n}^{s}(x, G_n )\mu(dg_1)\dots\mu(dg_n)$.
Denote by $\mathbb{E}_{\mathbb Q_s^x}$ the corresponding expectation.
For any $n \in \mathbb{N}$ and
any bounded measurable function $h$ on $(\mathcal S \times \mathbb R)^{n}$,
it holds that for any $s \in (-\eta, \eta)$ and $x\in\mathcal{S}$, 
\begin{align}\label{basic equ1}
&  \frac{1}{\kappa^{n}(s)r_{s}(x)} \mathbb{E} \Big[ r_{s}(X_{n}^{x}) e^{s \sigma(G_n, x)} 
h\big( X_{1}^{x}, \sigma(G_1, x),   \dots, X_{n}^{x}, \sigma(G_n, x) \big) \Big]    \notag\\
& =  \mathbb{E}_{\mathbb{Q}_{s}^{x}} \big[ h\big( X_1^x, \sigma(G_1, x),\dots, X_n^x, \sigma(G_n, x) \big) \big].
\end{align}


\subsection{Properties of the  Markov  operator \texorpdfstring{$Q_s$}{Qs}}\label{seQS}

For any $s \in (-\eta, \eta)$, define the Markov operator $Q_s$ as follows: 
for any $\varphi \in \mathcal{B}_{\gamma}$, 
\begin{align*}
Q_{s}\varphi(x) = \frac{1}{\kappa(s)r_{s}(x)}P_s(\varphi r_{s})(x),  \quad   x \in \mathcal{S}.
\end{align*}
Under the changed measure $\mathbb Q_s^x$, the process
$(X_{n}^x)_{n\in \mathbb{N}}$ is a Markov chain with the transition operator given by $Q_s$.

The next assertion will be useful to prove
that the function $\kappa$ is strictly convex (see Lemma \ref{Lem-Lamb-Conv}).
Recall that $V(\Gamma_{\mu} )$ is the support 
of the measure $\nu$ (cf.\eqref{def-VGamma-inv} and \eqref{def-VGamma-pos}).

\begin{lemma}\label{Lem-QsIne}
Assume the conditions of Proposition \ref{transfer operator}.
Let $s \in (-\eta, \eta)$, where $\eta > 0$ is a small constant. 
If $\varphi \leq Q_s \varphi$ for some real-valued function $\varphi \in \mathcal{C(S)}$,
then $\varphi(x) = \sup_{y \in \mathcal{S}}  \varphi(y)$ for any $x \in V(\Gamma_{\mu} )$.
\end{lemma}

\begin{proof} We use the approach developed  in \cite{GL16}.
Set $\mathcal{M} = \sup_{y \in \mathcal{S}} \varphi(y)$
and $\mathcal{S}^+ = \{ x \in \mathcal{S}:  \varphi(x)  = \mathcal{M} \}$.
From the condition  $\varphi \leq Q_s \varphi$
and the fact that $\int_{\Gamma_{\mu}} q_1^s(x, g) \mu(dg) = 1$,
we get that if $x \in \mathcal{S}^+$, then $g \!\cdot\! x \in \mathcal{S}^+$ for any $g \in \Gamma_{\mu}$, 
so that $\Gamma_{\mu} \mathcal{S}^+ \subseteq \mathcal{S}^+$.
Since $V(\Gamma_{\mu} )$ is the unique minimal $\Gamma_{\mu}$-invariant set (see \cite{GL16} and \cite{BDGM14}),
we obtain $V(\Gamma_{\mu} ) \subseteq \mathcal{S}^+ $ and the claim follows.
\end{proof}

We state the spectral gap property of the Markov operator $Q_s$, whose proof is postponed to Section \ref{sec-spgappert}.
\begin{proposition}\label{ProSpecGapNega}
Assume the conditions of Proposition \ref{transfer operator}.
Then there exists  $\eta>0$ such that for any $s\in (-\eta, \eta)$ and $n \geq 1$,
we have
\begin{align*}
Q_s^n = \Pi_s  + N_s^n,
\end{align*}
where the mappings $s \mapsto \Pi_s$, $s \mapsto N_s \in \mathcal{L}(\mathcal{B}_\gamma,\mathcal{B}_\gamma)$ 
are analytic on $(-\eta, \eta)$ and satisfy the following properties:
\begin{itemize}
\item[{\rm(a)}]
with $\pi_{s}(\varphi) := \frac{\nu_{s}(\varphi r_{s})}{\nu_{s}(r_{s})}$,  we have
for any $\varphi \in \mathcal{B}_{\gamma}$, 
\begin{align*} 
\Pi_s (\varphi)(x) = \pi_{s}(\varphi) \mathbf{1},  \quad
N_s^n (\varphi)(x) =  \frac{1}{ \kappa^n(s) }  \frac{ L_s^n (\varphi r_s) (x) }{ r_s(x) },  
\quad  x \in \mathcal{S}, 
\end{align*}
where $\nu_s$, $r_s$, $L_s$ are given in Proposition \ref{transfer operator};
\item[{\rm(b)}]  $\Pi_s N_s = N_s \Pi_s = 0$, and for each $k \in \mathbb{N}$,
there exist constants $C_{k}>0$ and $a \in (0,1)$ such that
\begin{align}\label{Opera-Nsn}
\sup_{s \in (-\eta, \eta)} 
\Big\lVert \frac{d^k}{d s^k} N_s^n \Big\rVert_{\mathcal{B}_{\gamma} \to \mathcal{B}_{\gamma}} \leq C_{k} a^n.
\end{align}
\end{itemize}
\end{proposition}


\subsection{Quasi-compactness of the operator \texorpdfstring{$Q_{s+it}$}{Rsz}}
\label{subsec-quasicomQst}
For any $s \in (-\eta, \eta)$ and $t \in \mathbb{R}$,
 define the operator $Q_{s+it}$ as follows: for any $\varphi \in \mathcal{B}_{\gamma}$,
\begin{align*}
Q_{s+it} \varphi(x) & =  \frac{1}{\kappa(s) r_s(x)} P_{s+it} (\varphi r_s)(x)  \notag\\
& =  \frac{1}{\kappa(s)r_{s}(x)} 
 \mathbb E \Big[
e^{(s + it) \sigma(g_1, x)} \varphi(g_1 \!\cdot\! x) r_s(g_1 \!\cdot\!x)  \Big],
\quad   x \in \mathcal{S}.
\end{align*}
The spectral gap properties of the operator $Q_{s+it}$ for $|t|$ small enough
can be deduced from Proposition \ref{transfer operator}.
However, this approach does not work for large $|t|$. 
In order to investigate the spectral gap properties of the operator $Q_{s+it}$ for $t\in \mathbb R$, 
we first prove the Doeblin-Fortet inequality and then we apply the theorem of
Ionescu-Tulcea and Marinescu \cite{IM50}
to establish the quasi-compactness of the operator $Q_{s+it}$.
Using this property, we shall apply the non-arithmeticicty condition \ref{CondiNonarith} to
prove that the spectral radius of $Q_{s+it}$ is strictly less than $1$ when $t$ is different from $0$.

The following is the Doeblin-Fortet inequality for the operator $Q_{s+it}$:

\begin{lemma}\label{LemDFinequ}
Assume the conditions of Proposition \ref{transfer operator}.
Then, there exist constants $0<a<1$, and $\eta > 0$ small enough, such that for any
$s \in (-\eta, \eta)$, $t \in \mathbb{R}$,
$n \geq 1$ and $\varphi \in \mathcal{B}_{\gamma}$, we have
\begin{align}\label{QuasiIne001}
 [Q_{s+it}^{n} \varphi ]_{ \gamma }
  \leq C_{s,t,n} \lVert \varphi \rVert_{\infty} + C_{s} a^n [ \varphi ]_{ \gamma}.
\end{align}
\end{lemma}

For positive-valued $s$, analogous results can be found in \cite{GL16} for invertible matrices
and in \cite{BM16}  for positive matrices.
The proofs in \cite{GL16, BM16} rely essentially on the
H\"{o}lder continuity of the mapping $x \mapsto q_n^s(x, g)$ defined in \eqref{Def-qns}.
However, this property does not hold any more in the case when $s$ is negative.
Our proof of Lemma \ref{LemDFinequ} is carried out using the H\"{o}lder inequality
and the spectral gap properties of the operator $P_s$ established in Proposition \ref{transfer operator}.

\begin{proof}[Proof of Lemma \ref{LemDFinequ}]
Using the definition of $Q_{s+it}$ and the cocycle property \eqref{CocycleEqua}, we get that for any $n \geq 1$,
 \begin{align*}
 Q_{s+it}^n \varphi(x) = \frac{1}{\kappa^n(s) r_{s}(x)} P_{s+it}^n(\varphi r_{s})(x),  \quad  x \in \mathcal{S}.
 \end{align*}
 It follows that 
 \begin{align}\label{DFinequDecom}
 \sup_{x, y \in \mathcal{S}, x \neq y} \frac{| Q_{s+it}^n \varphi(x) - Q_{s+it}^n \varphi(y) |}{ \mathbf{d}^{\gamma}(x,y) }
   \leq J_1(n) + J_2(n),
 \end{align}
   where
    \begin{align*}
   J_1(n)  & =    \sup_{x, y \in \mathcal{S}, x \neq y}
              \frac{1}{ \mathbf{d}^{\gamma}(x,y)  \kappa^n(s) }  \Big| \frac{1}{ r_{s}(x)}
             - \frac{1}{ r_{s}(y)}   \Big| \big| P_{s+it}^n(\varphi r_{s})(x) \big|,  \notag\\
   J_2(n)  & =   \sup_{x, y \in \mathcal{S}, x \neq y}   \frac{1}{ r_s(y) \mathbf{d}^{\gamma}(x,y) \kappa^n (s) }
                 \big|  P_{s+it}^n(\varphi r_{s})(x) - P_{s+it}^n(\varphi r_{s})(y)  \big|.
    \end{align*}
Note that by Proposition \ref{transfer operator}, for any $s \in (-\eta, \eta)$, we have
$\min_{x\in \mathcal{S}}  r_s (x) >0 $, $\max_{x\in \mathcal{S} } r_s (x) < \infty $ and $\kappa(s) >0$. 

\textit{Control of $J_1(n)$.}
Observe that uniformly in $x \in \mathcal{S}$,
\begin{align*}
 | P_{s+it}^n(\varphi r_{s})(x) | 
 \leq  P_{s}^n(|\varphi| r_{s})(x)  
 \leq  \lVert \varphi \rVert_{\infty} \kappa^n(s)  \lVert r_{s} \rVert_{\infty}
\leq C_s \lVert \varphi \rVert_{\infty} \kappa^n(s).
\end{align*}
Since $r_s \in \mathcal{B}_{\gamma}$, 
this implies that 
for any $s \in (-\eta, \eta)$ and $t \in \mathbb{R}$,
\begin{align}\label{DFinequJ1}
J_1(n) \leq C_s \lVert \varphi \rVert_{\infty}.
\end{align}

\textit{Control of $J_2(n)$.}
Using the definition of $P_{s+it}$ and taking into account that 
$r_s$ is strictly positive and bounded on $\mathcal{S}$, we have
\begin{align}\label{DFinequJ2n}
J_2(n) \leq C_s (J_{21}(n) + J_{22}(n) + J_{23}(n)),
\end{align}
where
\begin{align*}
J_{21}(n) & =  \sup_{x, y \in \mathcal{S}, x \neq y} \frac{1}{ \mathbf{d}^{\gamma}(x,y) \kappa^n(s) }
      \Big| \mathbb{E} \Big[ ( e^{(s+it) \sigma (G_n, x)} - e^{(s+it) \sigma (G_n, y)} ) 
         \varphi(X_n^x)   \Big] \Big|,  \notag\\
J_{22}(n) & =  \sup_{x, y \in \mathcal{S}, x \neq y} \frac{1}{ \mathbf{d}^{\gamma}(x,y) \kappa^n(s) }
    \Big| \mathbb{E}  \Big[ e^{(s+it) \sigma (G_n, y)}  ( \varphi(X_n^x) - \varphi(X_n^y) )  \Big] \Big|, \notag\\
J_{23}(n) & =  \sup_{x, y \in \mathcal{S}, x \neq y} \frac{1}{ \mathbf{d}^{\gamma}(x,y) \kappa^n(s) }
  \Big| \mathbb{E}  \Big\{  e^{(s+it) \sigma (G_n, y)}  \varphi(X_n^y)  [ r_s(X_n^x) - r_s(X_n^y) ]  \Big\} \Big|.
\end{align*}

\textit{Control of $J_{21}(n)$.}
Using \eqref{TO-Ine-ex} and the inequality $\log u \leq u^{\varepsilon}$, $u>1$, 
for $\varepsilon>0$ small enough, we obtain     
\begin{align}\label{DF-Ine-Gn-st}
 \big| e^{(s+it) \sigma (G_n, x)} - e^{(s+it) \sigma (G_n, y)} \big|  
\leq     2  ( N(G_n) )^{|s| + \varepsilon}
  \big| \sigma(G_n, x) - \sigma(G_n, y) \big|^{\gamma}.
\end{align}
From the inequality \eqref{Ineq-Distan},
by arguing as in the estimate of \eqref{InelogGx}, we get
\begin{align*}
\big| \sigma(G_n, x) - \sigma(G_n, y)\big|^{\gamma}
\leq C \lVert G_n \rVert^{\gamma} \iota( G_n )^{-\gamma} \mathbf{d}^{\gamma}(x,y).
\end{align*}
Using first \eqref{DF-Ine-Gn-st} and then the last bound,  we deduce that
\begin{align}\label{DFinequJ21}
J_{21}(n) 
\leq   \frac{C \lVert \varphi \rVert_{\infty} }{ \kappa^n(s) }
\Big\{ \mathbb{E} \Big[ ( N(g_1) )^{|s| + \varepsilon }  
  \lvert g_1 \rVert^{\gamma} \iota( g_1 )^{-\gamma} \Big]   \Big\}^n   
\leq     C_{s,t,n} \lVert \varphi \rVert_{\infty},
\end{align}
where the last inequality holds by condition \ref{CondiMoment}.

\textit{Control of $J_{22}(n)$.}
Since $\varphi \in \mathcal{B}_{\gamma}$, applying the H\"{o}lder inequality leads to
\begin{align}\label{DFinequJ22}
J_{22}(n) 
& \leq  \frac{C_s [ \varphi ]_{\gamma} }{ \kappa^n(s) } 
 \sup_{x, y \in \mathcal{S}, x \neq y}
  \mathbb{E} \Big[ e^{s \sigma(G_n, y)}  \frac{ \mathbf{d}^{\gamma}(X_n^x,X_n^y) }{ \mathbf{d}^{\gamma}(x,y) }  \Big] \notag\\
&  \leq  C_s [ \varphi ]_{\gamma}
   \sup_{x, y \in \mathcal{S}, x \neq y}
  \frac{ \big\{ \mathbb{E} \big[ e^{2s \sigma(G_n, y)}  \big]  \big\}^{ 1/2 } }{ \kappa^n(s) }
 \Big[ \mathbb{E} \frac{ \mathbf{d}^{2\gamma}(X_n^x,X_n^y) }{ \mathbf{d}^{2\gamma}(x,y) }  \Big]^{1/2}.
\end{align}
Since $\gamma>0$ is small enough, by \cite[Theorem 1]{LeP82}  for invertible matrices
and \cite[Lemma 3.2]{Hen97} for positive matrices,
there exists a constant $a_0 \in (0,1)$ such that for sufficiently large $n \geq 1$,
\begin{align} \label{ContracInverRes02}
\sup_{x, y \in \mathcal{S}, x \neq y}
     \Big[ \mathbb{E} \frac{ \mathbf{d}^{2\gamma}(X_n^x,X_n^y) }{ \mathbf{d}^{2\gamma}(x,y) }  \Big]^{1/2}
 \leq  a_0^{n}.
\end{align}
In view of Proposition \ref{transfer operator}, we have
\begin{align*}
\mathbb{E} \big[ e^{2s \sigma(G_n, y)} \big] 
= \kappa^n(2s) (M_{2s} \mathbf{1} )(y) + L_{2s}^n \mathbf{1}(y),
\quad  y \in \mathcal{S}.
\end{align*}
Since, by Proposition \ref{transfer operator}(d), $\lVert M_{2s} \mathbf{1} \rVert_{\infty}$ 
is bounded by some constant $C_s$, 
and $\lVert L_{2s}^n \mathbf{1} \rVert_{\infty}$
is bounded by $C_s \kappa^n(2s)$
uniformly in $n \geq 1$, 
it follows that
\begin{align}\label{DFinequJ21sNega02}
\sup_{n \geq 1} \sup_{y \in \mathcal{S}} \frac{ \mathbb{E} [ e^{2s \sigma(G_n, y)} ] }{ \kappa^n(2s) }  \leq  C_s.
\end{align}
As $\kappa$ is continuous in the neighborhood of $0$ and $\kappa(0) = 1$,
 one can choose $\eta>0$ small enough and a constant $\alpha \in (0, 1/a_0)$ 
 such that $\kappa^{n/2}(2s)/\kappa^n(s) \leq \alpha^n$, 
 uniformly in $s \in (-\eta, \eta)$.  
Substituting this inequality together with \eqref{ContracInverRes02} and \eqref{DFinequJ21sNega02} into \eqref{DFinequJ22},
we obtain that for any $s \in (-\eta, \eta)$ with $\eta >0$ small, 
there exists $0<a<1$ such that uniformly in $n \geq 1$,
\begin{align}\label{DFineqJ22}
J_{22}(n) \leq C_s a^{n} [ \varphi ]_{\gamma}. 
\end{align}

\textit{Control of $J_{23}(n)$.}
Using \eqref{DFinequJ21sNega02} and the fact that $r_s \in \mathcal{B}_{\gamma}$,
and applying similar techniques as in the  control of $J_{22}(n)$,
one can verify that there exists a constant $0<a<1$ such that uniformly in $n \geq 1$,
\begin{align}\label{DFinequJ23}
J_{23}(n) \leq C_s a^n \lVert \varphi \rVert_{\infty} [ r_s ]_{\gamma} \leq  C_s a^n \lVert \varphi \rVert_{\infty}.
\end{align}

Inserting  \eqref{DFinequJ21},  \eqref{DFineqJ22}  and \eqref{DFinequJ23}
into \eqref{DFinequJ2n}, we conclude that
\begin{align*}
 J_2(n) \leq C_{s,t,n} \lVert \varphi \rVert_{\infty} + C_s a^n [ \varphi ]_{\gamma}.
\end{align*}
Combining this with \eqref{DFinequJ1} and \eqref{DFinequDecom}, we obtain the inequality \eqref{QuasiIne001}.
\end{proof}

From Lemma \ref{LemDFinequ} and the fact that
$\lVert Q_{s + it} \varphi \rVert_{\infty} \leq C_s \lVert \varphi \rVert_{\infty}$ for any $s \in (-\eta, \eta)$ and $t \in \mathbb{R}$,
we can deduce that $Q_{s + it} \in \mathcal{L}(\mathcal{B}_\gamma,\mathcal{B}_\gamma)$.
We next prove that the operator $Q_{s+it}$ is quasi-compact.
Recall that an operator $Q\in \mathcal{L(B_{\gamma},B_{\gamma})}$ is called \emph{quasi-compact}
if $\mathcal{B_{\gamma}}$ can be decomposed into two $Q$ invariant closed subspaces
$\mathcal{B_{\gamma}}=E\oplus F$
such that  $\dim E< \infty$, each eigenvalue of $Q|_{E}$ has modulus $\varrho(Q)$,
 and   $\varrho(Q|_{F})<\varrho(Q)$ (see \cite{HH01} for more details).

\begin{proposition}\label{Prop-Qst-Quasi}
Assume the conditions of Proposition \ref{transfer operator}.
Then, there exists $\eta>0$ such that for any $s\in (-\eta, \eta)$ and $t \in \mathbb{R}$,
the operator $Q_{s+it}$ is quasi-compact.
\end{proposition}

\begin{proof}
The proof consists of verifying the conditions of the theorem of Ionescu-Tulcea and Marinescu \cite{IM50}. 
We follow the formulation in \cite[Theorem II.5]{HH01}.

Firstly, by the definition of $Q_{s+it}$,
there exists a constant $C_s > 0$ 
such that  $\lVert Q_{s+it} \varphi \rVert_{\infty} \leq C_s \lVert \varphi \rVert_{\infty}$
 for any $s \in (-\eta, \eta)$, $t \in \mathbb{R}$ and $\varphi \in \mathcal{B}_{\gamma}$.

Secondly, by Lemma \ref{LemDFinequ}, the Doeblin-Fortet inequality \eqref{QuasiIne001} holds for the operator $Q_{s+it}$.

Thirdly, denoting $K = \{ Q_{s+it} \varphi: \lVert \varphi \rVert_{ \gamma } \leq 1  \}$,
we claim that for any $s \in (-\eta, \eta)$ and $t \in \mathbb{R}$, the set $K$
is conditionally compact in $(\mathcal{B}_{\gamma}, \lVert \cdot \rVert_{\infty})$.
Since $\lVert Q_{s+it} \varphi \rVert_{\infty} \leq C_s \lVert \varphi \rVert_{\infty}$ for any $\varphi \in \mathcal{B}_{\gamma}$,
we conclude that $K$ is uniformly bounded
in $(\mathcal{B}_{\gamma}, \lVert \cdot \rVert_{\infty})$.
Moreover, by taking $n =1$ in \eqref{QuasiIne001}, we get that uniformly in $\varphi \in \mathcal{B}_{\gamma}$
with $\lVert \varphi \rVert_{ \gamma } \leq 1$,
\begin{align*}
 | Q_{s+it} \varphi(x) -  Q_{s+it} \varphi(y) |
  \leq C_{s,t}  \mathbf{d}^{\gamma}(x,y).
\end{align*}
This shows that $K$ is equicontinuous in $(\mathcal{B}_{\gamma}, \lVert \cdot \rVert_{\infty})$.
Therefore, we obtain the claim by the Arzel\`a-Ascoli theorem.

The assertion of the proposition now follows from  
the theorem of Ionescu-Tulcea and Marinescu.
\end{proof}

The proposition below shows that
the spectral radius of the operator $Q_{s + it}$ is strictly less than $1$
when $t$ is different from $0$.
The proof which relies on the non-arithmeticity condition \ref{CondiNonarith},
follows the standard pattern in \cite{GL16, BM16};
it is included for the commodity of the reader.

\begin{proposition}\label{PropNonArithQit}
Assume either conditions \ref{CondiMoment} and \ref{CondiIP} for invertible matrices,
or conditions \ref{CondiMoment}, \ref{CondiAP} and \ref{CondiNonarith} for positive matrices.
Then, there exists $\eta > 0$ such that 
for any $s \in (-\eta, \eta)$ and $t \in \mathbb{R} \!\setminus\! \{0\}$,
we have $\varrho(Q_{s + it}) < 1$.
\end{proposition}

\begin{proof}
By the definition of $Q_{s + it}$, we have $\varrho(Q_{s + it}) \leq \varrho(Q_s) = 1$.
Suppose that $\varrho(Q_{s + it}) = 1$ for some $t \neq 0$.
Then, applying Proposition \ref{Prop-Qst-Quasi},
there exist $\varphi \in \mathcal{B}_{\gamma}$ and $\beta \in \mathbb{R}$ 
such that $Q_{s + it} \varphi = e^{i \beta} \varphi$.
From this equation,  
we deduce that $|\varphi| \leq Q_s |\varphi|$.
Using Lemma \ref{Lem-QsIne}, this implies that
$|\varphi(x) | = \sup_{y \in \mathcal{S}} | \varphi(y)|$ for any $x \in V(\Gamma_{\mu} )$,
so that $\varphi(x) = c e^{i \vartheta(x)}$, where
$c \neq 0$ is a constant and  $\vartheta$ is a real-valued continuous function on $\mathcal{S}$.
Substituting this into the equation $Q_{s + it} \varphi = e^{i \beta} \varphi$
gives that for any  $x\in V(\Gamma_{\mu})$,
\begin{align*}
\mathbb{E}_{\mathbb{Q}_s^x} 
\exp \big[ it\sigma(g_1, x)- i\beta + i\vartheta(g_1 \!\cdot\! x)-i \vartheta(x) \big] = 1. 
\end{align*}
Since $\vartheta$ is real-valued, this implies
$\exp[it\sigma(g, x) - i\beta + i\vartheta(g \cdot x)-i \vartheta(x)]=1$
for any  $x\in V(\Gamma_{\mu})$ and $\mu$-a.e. $g\in \Gamma_{\mu}$, 
which contradicts to condition \ref{CondiNonarith}.
Therefore, $\varrho(Q_{s + it}) < 1$ for any $t \neq 0$. 
Recalling that condition \ref{CondiIP} implies condition  \ref{CondiNonarith} for invertible matrices,
the proof of Proposition \ref{PropNonArithQit} is complete.
\end{proof}

\subsection{Spectral gap properties of the perturbed operator \texorpdfstring{$R_{s,z}$}{Rsz}} 
\label{sec-spgappert}

For any $s\in (-\eta, \eta)$ and $z \in \mathbb{C}$
such that $s+ \Re z \in (-\eta_0, \eta_0)$,
define the perturbed operator $R_{s, z}$ as follows: for any $\varphi \in \mathcal{B}_{\gamma}$,
\begin{align}\label{Def-Rn-sw000}
R_{s, z}\varphi(x)
= \mathbb{E}_{\mathbb{Q}_{s}^{x}} \big[ e^{z( \sigma(g_1, x) - \Lambda'(s) )}\varphi(X_{1}^x) \big],
\quad    x \in \mathcal{S}.
\end{align}
With some calculations using \eqref{CocycleEqua}, it follows that for any $n \geq 1$,
\begin{align}\label{Def-Rn-sw}
R^{n}_{s, z}\varphi(x)
= \mathbb{E}_{\mathbb{Q}_{s}^{x}} \big[ e^{z( \sigma(G_n, x) - n\Lambda'(s) )}\varphi(X_{n}^x) \big],
\quad    x \in \mathcal{S}.
\end{align}
The following formula relates the operator $R^n_{s,z}$ to the operator $P^n_{s+z}$ and is of independent interest:
for any $\varphi \in \mathcal{B}_{\gamma}$, $n \geq 1$,
$s \in (-\eta, \eta)$ and $z \in B_\eta(0)$,
\begin{align}\label{PfRsw01}
R_{s,z}^n (\varphi)
=  e^{ -n z \Lambda'(s)} \frac{ P_{s+z}^n (\varphi r_s) }{ \kappa^n(s) r_s  }.
\end{align}
The identity \eqref{PfRsw01} is obtained by the definitions of  $R_{s,z}$ and $P_{z}$ using the change of measure \eqref{basic equ1}.  

There  are two ways to establish spectral gap properties of the operator  $R_{s,z}$: one is to use the perturbation theory of operators \cite[Theorem III.8]{HH01}, 
another is based on the Ionescu-Tulcea and Marinescu theorem \cite{IM50} about the quasi-compactness of operators. 
The representation \eqref{PfRsw01} allows us to deduce the spectral gap properties of $R_{s,z}$ 
directly from the properties of the operator $P_{z}$.
This has some advantages: 
it ensures  the uniformity in $s \in (-\eta,\eta)$, 
allows to deal with negative-vaued $s$ and 
provides an explicit formula for the projection operator $\Pi_{s, z}$ and the remainder operator $N^{n}_{s, z}$ defined below.

Recall that $\Lambda = \log \kappa$, where $\kappa$ is defined in \eqref{Def-kappa01}. 
 
\begin{proposition} \label{perturbation thm}
Assume the conditions of Proposition \ref{transfer operator}.
Then, there exist $\eta>0 $ and $\delta \in (0, \eta)$  
such that for any $s \in (-\eta, \eta)$ and $z \in  B_\delta(0)$,
\begin{align}
R^{n}_{s, z} &= \lambda^{n}_{s, z}\Pi_{s, z} + N^{n}_{s, z},  \quad  n \geq 1,  \label{perturb001}\\
\lambda_{s, z} &= e^{ \Lambda(s+z) - \Lambda(s) - \Lambda'(s) z} \label{relationlamkappa001}
\end{align}
and for $\varphi \in \mathcal{ B_{\gamma} }$,
\begin{align}
& \Pi_{s, z}(\varphi) = 
\frac{ \nu_{s+z} (\varphi r_s) }{ \nu_{s+z} (r_{s+z}) }  \frac{ r_{s+z} }{ r_s }, \label{exprePisw-Nsw-a}\\
& N^{n}_{s, z}(\varphi) =
e^{- n [\Lambda(s) + \Lambda'(s) z] } \frac{ L_{s+z}^n (\varphi r_s) }{r_s}, \label{exprePisw-Nsw-b}
\end{align}
where $r_z$, $\nu_z$ and $L_z$ are given in Proposition \ref{transfer operator}.
In addition, we have:
\begin{itemize}
\item[{\rm(a)}]  for fixed $s$,
the mappings $z \mapsto \Pi_{s, z}: B_\delta(0) \to \mathcal{L(B_{\gamma},B_{\gamma})}$,
$z \mapsto N_{s, z}: B_\delta(0) \to \mathcal{L(B_{\gamma},B_{\gamma})}$
and $z \mapsto \lambda_{s, z}: B_\delta(0) \to \mathbb{C}$
are analytic,

\item[{\rm(b)}]  for fixed $s$ and $z$, $\Pi_{s,z}$ is a rank-one projection with
$\Pi_{s, 0}(\varphi)(x)=\pi_{s}(\varphi)$ for any $\varphi \in \mathcal{B}_{\gamma}$ and $x\in \mathcal{S}$,
and $\Pi_{s,z} N_{s, z} = N_{s,z} \Pi_{s,z} = 0$,

\item[{\rm(c)}]  for any $k \in  \mathbb{N}$, there exist constants $C_k>0$ and $0<a<1$ such that
\begin{align}
&  \sup_{s \in (-\eta, \eta) } \sup_{ z \in  B_\delta(0) }
\Big\lVert \frac{d^{k}}{dz^{k}} \Pi_{s, z} \Big\rVert_{\mathcal{B}_{\gamma} \to \mathcal{B}_{\gamma}}
\leq C_k,  \label{SpGapContrPi}  \\
&  \sup_{s \in (-\eta, \eta) } \sup_{ z \in  B_\delta(0) }
\Big\lVert \frac{d^{k}}{dz^{k}}N^{n}_{s, z} \Big\rVert_{\mathcal{B}_{\gamma} \to \mathcal{B}_{\gamma}}
\leq C_k a^{n}.    \label{SpGapContrN}
\end{align}

\end{itemize}
\end{proposition}

Note that, for $s>0$, similar results have been obtained in \cite{BM16}.
The novelty here is that $s$ can account for negative values $s \in (-\eta, 0]$
and that the bounds \eqref{SpGapContrPi} and \eqref{SpGapContrN} hold 
uniformly in $s \in (-\eta, \eta)$.
This plays a crucial role in establishing Theorem \ref{MainThmNormTarget}.

\begin{proof}[Proof of Proposition \ref{perturbation thm}]
The proof is divided into three steps.

\textit{Step 1.} 
By Proposition \ref{transfer operator}, we have
\begin{align*}
P_{s+z}^n (\varphi r_s)
=  \kappa^n( s+ z ) \frac{ \nu_{s+z}(\varphi r_s) }{ \nu_{s+z}( r_{s +z} ) }  r_{s +z} + L_{s+z}^n(\varphi r_s).
\end{align*}
Substituting this into \eqref{PfRsw01} shows
\eqref{perturb001}, \eqref{relationlamkappa001}, \eqref{exprePisw-Nsw-a} and \eqref{exprePisw-Nsw-b}.

\textit{Step 2.} We prove parts (a) and (b).
The assertion in part (a)
follows from the expressions \eqref{relationlamkappa001}, 
\eqref{exprePisw-Nsw-a} and \eqref{exprePisw-Nsw-b},
and the analyticity of the mappings
$z \mapsto \kappa(z)$, $z \mapsto r_z$, $z \mapsto \nu_z$ and $z \mapsto L_z$
 defined in Proposition \ref{transfer operator}.
To show part (b), by \eqref{exprePisw-Nsw-a}, we have that $\Pi_{s,z}$ is a rank-one projection
on the subspace 
$\big\{ w\frac{ r_{s+z} }{r_s}: w\in \mathbb C  \big\}$.
The identity $\Pi_{s, 0}(\varphi)(x) = \pi_{s}(\varphi)$ follows from \eqref{exprePisw-Nsw-a}
and the fact that $\pi_{s}(\varphi) = \frac{\nu_s(\varphi r_s) }{ \nu_s(r_s) }$.
Using Proposition \ref{transfer operator}, we get 
that $L_z r_z =0$ and $\nu_z (L_z \varphi) =0$ for any $\varphi \in \mathcal{B}_{\gamma}$.
This, together with \eqref{exprePisw-Nsw-a} and \eqref{exprePisw-Nsw-b}, 
shows that $\Pi_{s,z} N_{s, z} = N_{s,z} \Pi_{s,z} = 0$.

\textit{Step 3.} We prove part (c).  By Proposition \ref{transfer operator},
there exists a constant $\eta >0$ such that the mappings $z \mapsto \kappa(z),$ $z \mapsto r_z,$
$z \mapsto \nu_z$ are analytic and uniformly bounded on $B_{2\eta}(0)$.
Combining this with \eqref{exprePisw-Nsw-a}, we obtain \eqref{SpGapContrPi}.
We now prove \eqref{SpGapContrN}. 
Since the function $r_s$ is strictly positive on the compact set $\mathcal{S}$,
by Proposition \ref{transfer operator}(d),
we deduce that there exists a constant $0< a_0 < 1$ 
such that uniformly in $\varphi \in \mathcal{B}_{\gamma}$,  
\begin{align}\label{Pf-Contr-Nsw01}
\sup_{s \in (-\eta, \eta)} \sup_{z \in B_{\eta}(0)}
\Big\lVert   \frac{ L_{s+z}^n (\varphi r_s) }{r_s} \Big\rVert_{\gamma} 
\leq C \lVert \varphi \rVert_{\gamma}  a_0^n. 
\end{align}
Using the fact that the function $\Lambda$ is continuous and $\Lambda(0)=0$,
 there exist a small $\eta>0$, $\delta \in (0,\eta)$ and a constant $a_1<\frac{1}{a_0}$ such that
\begin{align*}
\sup_{s \in (-\eta, \eta)} \sup_{z \in B_{\delta}(0)}
\big| e^{ - n [  \Lambda(s) + \Lambda'(s) z ] } \big| \leq C a_1^n.
\end{align*}
Combining this with \eqref{Pf-Contr-Nsw01} proves \eqref{SpGapContrN} with $k=0$.
The proof of \eqref{SpGapContrN} when $k \geq 1$ can be carried out in the same way
as in the case of $k=0$.  
\end{proof}

\begin{proof}[Proof of Proposition \ref{ProSpecGapNega}]
The assertion is obtained from Proposition \ref{perturbation thm} taking $z=0$.
\end{proof}

In order to establish the non-arithmeticity of the perturbed operator $R_{s,it}$, 
we shall need the following lemma from \cite[Lemma III.9]{HH01}:

\begin{lemma}\label{Lem_HH_MD_01}
Let $s \in \mathbb R$, $\delta >0$ and $I_{s, \delta} = (s-\delta, s+\delta)$. 
Assume that the mapping $t \in  I_{s, \delta} \mapsto P(t) \in \mathcal{L(B_{\gamma},B_{\gamma})}$ is continuous. 
Let $r> \varrho(P(s))$.  
Then, there exist constants $\varepsilon = \varepsilon(s)$ and $c = c(s) >0$ such that 
\begin{align*}
\sup_{ t \in (s - \varepsilon, s + \varepsilon) } 
\lVert P^n(t) \rVert_{\mathcal{B}_{\gamma}\rightarrow\mathcal{B}_{\gamma}} 
< c r^n. 
\end{align*} 
Moreover, it holds that 
\begin{align*}
\limsup_{t \to s} \varrho( P(t) )  \leq \varrho(P(s)). 
\end{align*} 
\end{lemma}

\begin{proposition}\label{Prop-UnifR}
Assume the conditions of Proposition \ref{PropNonArithQit}.
For any compact set $K\subseteq\mathbb{R}\backslash\{0\}$,
there exist constants $C_{K}>0$ and $\eta>0$ such that 
for any $n \geq 1$ and $\varphi\in \mathcal{B}_{\gamma}$,
\begin{align*} 
\sup_{s \in (-\eta, \eta)} \sup_{t\in K} \sup_{x\in \mathcal{S}}
|R^{n}_{s, it}\varphi(x)| \leq e^{-nC_{K}} \lVert \varphi \rVert_{\gamma}.
\end{align*}
\end{proposition}

\begin{proof} 
By Proposition \ref{PropNonArithQit}, for any fixed 
$s \in (-\eta, \eta)$ and $t \in \mathbb{R} \setminus \{0\}$, we have 
$\varrho(R_{s + it}) = \varrho( Q_{s + i t} ) < 1$. 
It follows that for any 
$s \in (-\eta, \eta)$ and $t \in \mathbb{R} \setminus \{0\}$, there exists a constant $C(s, t) > 0$ such that, 
for any $n \geq 1$ and $\varphi\in \mathcal{B}_{\gamma}$, 
\begin{align*}
\sup_{x\in \mathcal{S}} |R^{n}_{s,it}\varphi(x)|
\leq e^{ -  n C(s, t) }  \lVert \varphi \rVert_{\gamma}.
\end{align*}
From \eqref{PfRsw01}, we see that the operator $R_{s,it}$ is continuous in $s$ and $t$.  
By Lemma \ref{Lem_HH_MD_01}, there exist constants $\varepsilon(s) >0$ and $\delta(t) >0$ such that
\begin{align*}
\sup_{s' \in (s - \varepsilon(s), s + \varepsilon(s))} \sup_{t' \in (t- \delta(t), t + \delta(t))}
\sup_{x\in \mathcal{S}} |R^{n}_{s',it'}\varphi(x)|
\leq e^{ -  n C(s, t) }  \lVert \varphi \rVert_{\gamma}.
\end{align*}
Let $I \subset (-\eta, \eta)$ and $K \subseteq \mathbb{R} \backslash\{0\}$ be any compact sets. 
Since 
\begin{align*}
\bigcup_{(s,t) \in I \times K} 
\Big\{ \big( s - \varepsilon(s), s + \varepsilon(s) \big) \times \big( t - \delta(t), t + \delta(t) \big) \Big\}  
 \supset I \times K, 
\end{align*}
by Heine-Borel's theorem, 
there exist an integer $m_0 \geq 1$ and a sequence $\{s_m, t_m \}_{1 \leq m \leq m_0}$ such that 
\begin{align*}
\bigcup_{m = 1}^{m_0}  
\Big\{ (s_m - \varepsilon_m, s_m + \varepsilon_m) \times (t_m - \delta_m, t_m + \delta_m) \Big\}  
 \supset I \times K, 
\end{align*}
where $\varepsilon_m = \varepsilon(s_m)$ and $\delta_m = \delta(s_m)$.
This concludes the proof of Proposition \ref{Prop-UnifR}
by taking $C_K = \min_{1 \leq m \leq m_0} C(s_m, t_m)$. 
\end{proof}

We now give some properties of the function $b_{s,\varphi}$ defined as follows: 
for any $s \in (-\eta, \eta)$ and $\varphi \in \mathcal{B}_{\gamma}$,  
\begin{align*}
b_{s, \varphi}(x): = \lim_{n \to \infty}
   \mathbb{E}_{\mathbb{Q}_{s}^{x}} \big[ ( \sigma(G_n, x) - n \Lambda'(s) ) \varphi(X_n^x) \big],
\quad   x \in \mathcal{S}. 
\end{align*}
In particular, with $s =0$, we have $b_{0, \varphi} = b_{\varphi}$, which is defined in \eqref{func-phi-001}.

\begin{lemma}\label{Lem-Bs}
Assume the conditions of Proposition \ref{transfer operator}.
Then the function $b_{s,\varphi}$  
is well-defined, $b_{s,\varphi}  \in \mathcal{B}_{\gamma}$ and
\begin{align} \label{Def2-bs}
b_{s,\varphi}(x) = \frac{ d \Pi_{s,z} }{ dz } \Big|_{z=0} \varphi(x),  \quad  x \in \mathcal{S}. 
\end{align}
\end{lemma}

\begin{proof}
In view of Proposition \ref{perturbation thm}, 
we have that for any $\varphi \in \mathcal{B}_{\gamma}$, 
\begin{align*}
\mathbb{E}_{\mathbb{Q}_{s}^{x}} \big[ e^{z( \sigma(G_n, x) - n\Lambda'(s) )} \varphi(X_n^x) \big]
= \lambda^{n}_{s, z}\Pi_{s, z} \varphi(x) + N^{n}_{s, z} \varphi(x),  \quad  x \in \mathcal{S}.
\end{align*}
From \eqref{relationlamkappa001}, 
we have $\lambda_{s,0} = 1$ and $\frac{d \lambda_{s,z} }{ dz } |_{z=0} =0$. 
Differentiating both sides of the above equation with respect to $z$ 
at the point $0$ gives that for any $x \in \mathcal{S}$, 
\begin{align}\label{Bs01}
 \mathbb{E}_{\mathbb{Q}_s^x} \big[ ( \sigma(G_n, x) - n \Lambda'(s) ) \varphi(X_n^x) \big]
\! = \!  \frac{ d \Pi_{s,z} }{ dz } \Big|_{z=0} \varphi(x) + \frac{ d N^{n}_{s, z} }{ dz } \Big|_{z=0} \varphi(x).
\end{align}
Using the bounds \eqref{SpGapContrPi} and \eqref{SpGapContrN},
we find that the first term on the right-hand side of \eqref{Bs01}
belongs to $\mathcal{B}_{\gamma}$, and the second term 
converges to $0$ exponentially fast as $n \to \infty$.
Hence, letting $n \to \infty$ in \eqref{Bs01}, 
we obtain \eqref{Def2-bs}.
This shows that the function $b_{s,\varphi}$ is well-defined
and $b_{s,\varphi} \in \mathcal{B}_{\gamma}$.
\end{proof}

For any $s \in (-\eta, \eta)$ with $\eta>0$ small,
define $\mathbb{Q}_s = \int_{\mathcal{S}} \mathbb{Q}_s^x \, \pi_s(dx)$.
The following result will be used to prove the strong law of large numbers 
for $\sigma(G_n, x)$ under the changed measure $\mathbb{Q}_s$:

\begin{lemma}\label{Lem-properties000}
Assume the conditions of Proposition \ref{transfer operator}.
There exist $\eta >0$ and $c, C>0$ such that uniformly in $s \in (-\eta, \eta)$, 
$\varphi \in \mathcal B_{\gamma}$ and $n \geq 1$, 
\begin{align}\label{BsProp1}
\big| \mathbb{E}_{\mathbb{Q}_s} \big[ ( \sigma(G_n, x) - n \Lambda'(s) ) \varphi(X_n^x) \big]  \big|
\leq  C \lVert \varphi\rVert_{\gamma} e^{-cn}.
\end{align}
\end{lemma}

\begin{proof}
We follow the proof of the previous lemma.
Integrating both sides of the identity \eqref{Bs01} with respect to $\pi_s$, we get, for any $\varphi \in \mathcal B_{\gamma}$,
\begin{align}
\mathbb{E}_{\mathbb{Q}_s} \big[ ( \sigma(G_n, x) - n \Lambda'(s) ) \varphi(X_n^x) \big]  
 =  \pi_s  \Big( \frac{ d \Pi_{s,z} }{ dz } \Big|_{z=0} \varphi \Big)
  +  \pi_s  \Big( \frac{ d N^{n}_{s, z} }{ dz } \Big|_{z=0} \varphi \Big).  \label{Bs02}
\end{align}
Since $\Pi_{s,z}^2 \varphi = \Pi_{s,z} \varphi$, we have 
$ 2 \Pi_{s,0} ( \frac{ d \Pi_{s,z} }{ dz } |_{z=0} \varphi )
 = \frac{ d \Pi_{s,z} }{ dz } |_{z=0} \varphi$.
Integrating both sides of this equation with respect to $\pi_s$ and using the fact that $\Pi_{s,0} = \pi_s$,
we find  that
\begin{align}\label{Pf-Pis001}
\pi_s  \Big( \frac{ d \Pi_{s,z} }{ dz } \Big|_{z=0} \varphi  \Big) = 0.
\end{align}
It follows from \eqref{SpGapContrN} that uniformly in $\varphi \in \mathcal B_{\gamma}$ and $s \in (-\eta, \eta)$,
the second term on the right-hand side of \eqref{Bs02}
is bounded by $C \lVert \varphi \rVert_{\gamma} e^{-cn}$.
Therefore, from \eqref{Bs02} and \eqref{Pf-Pis001} 
we obtain \eqref{BsProp1}. 
\end{proof}

We now establish the strong laws of large numbers
for $\sigma(G_n, x)$ under the 
measures $\mathbb{Q}_s^x$ and $\mathbb{Q}_s$,
which are of independent interest. 

\begin{proposition}\label{Prop-LLN}
Assume the conditions of Proposition \ref{transfer operator}. Then,
there exists $\eta > 0$ such that for any $s \in (-\eta, \eta)$ and $x \in \mathcal{S}$,
\begin{align*}
\lim_{n \to \infty} \frac{\sigma(G_n, x)}{n} = \Lambda'(s), \quad \mathbb{Q}_s^x\mbox{-a.s..}
\end{align*}
\end{proposition}
\begin{proof}
By the Borel-Cantelli lemma, it suffices to show that for any $\varepsilon>0$, 
$s \in (-\eta, \eta)$ and $x \in \mathcal{S}$, we have
\begin{align}\label{LLN-series-01}
\sum_{ n = 1}^{\infty} \mathbb{Q}_s^x \Big(  \big| \sigma(G_n, x) - n \Lambda'(s) \big| \geq n \varepsilon  \Big) < \infty.
\end{align}
Now let us prove \eqref{LLN-series-01}. By Markov's inequality, 
we have for small $\delta>0$,
\begin{align*}
&  \mathbb{Q}_s^x \big(  \big| \sigma(G_n, x) - n \Lambda'(s) \big| \geq n \varepsilon  \big)  \notag\\
&  \leq e^{ -n\delta \varepsilon } \mathbb{E}_{ \mathbb{Q}_s^x } \big(  e^{ \delta ( \sigma(G_n, x) - n \Lambda'(s) ) } \big)
  +  e^{ -n\delta \varepsilon } \mathbb{E}_{ \mathbb{Q}_s^x } \big(  e^{ - \delta ( \sigma(G_n, x) - n \Lambda'(s) ) } \big). 
\end{align*}
From \eqref{Def-Rn-sw} and Proposition \ref{perturbation thm}, we deduce that 
there exist positive constants $c, C$ independent of $s, x, \delta$ such that
\begin{align*}
& \mathbb{E}_{ \mathbb{Q}_s^x } \big(  e^{ \delta ( \sigma(G_n, x) - n \Lambda'(s) ) } \big)
+  \mathbb{E}_{ \mathbb{Q}_s^x } \big(  e^{ - \delta ( \sigma(G_n, x) - n \Lambda'(s) ) } \big)  \notag\\
&  \leq   C e^{ n [ \Lambda(s+\delta) - \Lambda(s) - \Lambda'(s) \delta ] } 
         + C e^{ n [ \Lambda(s - \delta) - \Lambda(s) + \Lambda'(s) \delta ] }  + Ce^{-cn}. 
\end{align*}
Using Taylor's formula and taking $\delta>0$ small enough, we conclude that 
\begin{align*}
\mathbb{Q}_s^x \big(  \big| \sigma(G_n, x) - n \Lambda'(s) \big| \geq n \varepsilon  \big)
\leq  C e^{ -n \frac{\delta}{2} \varepsilon },  
\end{align*}
which implies the desired assertion \eqref{LLN-series-01}. 
\end{proof}

\begin{proposition}\label{Prop-LLN-b}
Assume the conditions of Proposition \ref{transfer operator}. Then,
there exists $\eta > 0$ such that for any $s \in (-\eta, \eta)$ and $x \in \mathcal{S}$,
\begin{align*}
\lim_{n \to \infty} \frac{\sigma(G_n, x)}{n}  = \Lambda'(s), 
 \quad  \mathbb{Q}_s\mbox{-a.s..}
\end{align*}
\end{proposition}
\begin{proof}
Taking $\varphi = \mathbf{1}$ in \eqref{BsProp1} leads to 
\begin{align}\label{Pf-LLN01}
\lim_{n \to \infty} \frac{1}{n}  \mathbb{E}_{\mathbb{Q}_s} \big( \sigma(G_n, x) \big) = \Lambda'(s).
\end{align}
Let $\Omega = M(d,\mathbb{R})^{\mathbb{N}^*}$ and $\widehat{\Omega} = \mathcal{S} \times \Omega$. 
Following \cite[Theorem 3.10]{GL16}, 
we define the shift operator $\widehat{\theta}$ on $\widehat{\Omega}$
by $\widehat{\theta}(x, \omega) = (g_1 \!\cdot\! x, \theta \omega)$,
where $\omega \in \Omega$ and $\theta$ is the shift operator on $\Omega$.
For any $x \in \mathcal{S}$ and $\omega \in \Omega$, set $h(x, \omega) = \sigma(g_1(\omega), x)$.
Then $h$ is $\mathbb{Q}_s$-integrable.
Since $\sigma(G_n, x) = \sum_{k=0}^{n-1} (h \circ \widehat{\theta}^k)(x, \omega)$ 
and $\mathbb{Q}_s$ is $\widehat{\theta}$-ergodic,
it follows from Birkhoff's ergodic theorem that $\frac{\sigma(G_n, x))}{n}$
converges $\mathbb{Q}_s$-a.s.\ to some constant $c_s$ as $n \to \infty$.
If we suppose that $c_s$ is different from $\Lambda'(s)$, then this contradicts to \eqref{Pf-LLN01}.
Thus $c_s = \Lambda'(s)$
and the assertion of the lemma follows. 
\end{proof}

Now we give the third-order Taylor expansion of $\lambda_{s,z}$ 
defined by \eqref{relationlamkappa001}, with respect to $z$ at the origin
in the complex plane $\mathbb C$.

\begin{proposition}\label{Prop-lambTaylor}
Assume the conditions of Proposition \ref{transfer operator}.
Then,  there exist  $\eta>0 $ and $\delta >0$ such that  
for any $s \in (-\eta, \eta)$ and $z \in  B_\delta(0)$,
\begin{align}\label{Expan-lambsw}
\lambda_{s, z}
= 1 + \frac{\sigma_s^{2}}{2}  z^{2} + \frac{ \Lambda^{'''}(s) }{6} z^3 + o(|z|^{3})  \quad  \mbox{as} \  |z| \to 0,
\end{align}
where
\begin{itemize}
\item[{\rm(a)}] $\sigma_s^{2} = \Lambda''(s) \geq 0$ and $\Lambda{'''}(s) \in \mathbb{R}$;
\item[{\rm(b)}] for invertible matrices, $\sigma_s >0$ under the stated conditions; for positive matrices, $\sigma_s >0$ 
if additionally $\sigma=\sigma_0>0$ or if the measure $\mu$ is non-arithmetic;
 \item[{\rm(c)}] uniformly in $s \in (-\eta, \eta)$ and $x \in \mathcal{S}$, 
\begin{align*}
\sigma_s^2 
=    \lim_{n \to \infty} \frac{1}{n} \mathbb{E}_{\mathbb{Q}_s^x} \big[ \sigma(G_n, x) - n \Lambda'(s) \big]^2 
= \lim_{n \to \infty} \frac{1}{n} \mathbb{E}_{\mathbb{Q}_s} \big[ \sigma(G_n, x) - n \Lambda'(s) \big]^2; 
\end{align*}
\item[{\rm(d)}] uniformly in $s \in (-\eta, \eta)$, 
\begin{align*}
\Lambda{'''}(s) =  
\lim_{n \to \infty} \frac{1}{n} \mathbb{E}_{\mathbb{Q}_s} \big[ \sigma(G_n, x) - n \Lambda'(s) \big]^3.
\end{align*}
\end{itemize}
\end{proposition}

The proof of Proposition \ref{Prop-lambTaylor} is based on the following lemma: 
\begin{lemma}\label{Lem-Lamb-Conv}
Assume the conditions of Proposition \ref{transfer operator}. Then 
the functions $\Lambda$ and $\kappa$ are convex on $(-\eta, \eta)$ for $\eta >0$ small enough.
Moreover, $\Lambda$ and $\kappa$ are strictly convex for invertible matrices under the given conditions,
and for positive matrices under the additional condition \ref{CondiNonarith}. 
\end{lemma}

\begin{proof}  
The proof follows \cite{GL16}.
Since $\Lambda = \log \kappa$, it suffices to prove Lemma \ref{Lem-Lamb-Conv} for the function $\Lambda$.
For any $t \in (0,1)$, $s_1, s_2 \in (-\eta, \eta)$, set $s' = t s_1 + (1-t) s_2$.  Using H\"{o}lder's inequality
and the fact that $P_s r_s = \kappa(s) r_s$, 
\begin{align}\label{LambConv-Ineq01}
P_{s'} ( r_{s_1}^{t} r_{s_2}^{1-t} )
\leq \kappa^t(s_1) \kappa^{1-t}(s_2) r_{s_1}^{t} r_{s_2}^{1-t}.
\end{align}
Since $\kappa( s' )$ is the dominant eigenvalue
of the operator $P_{s'}$,  we obtain
$\kappa( s' ) \leq \kappa^t(s_1) \kappa^{1-t}(s_2)$
and thus the function $\Lambda$ is convex. 

To show that the function $\Lambda$ is strictly convex, we suppose, by absurd,
that there exist $s_1 \neq s_2$ and some $t \in (0,1)$
such that $\kappa( s' ) = \kappa^t(s_1) \kappa^{1-t}(s_2)$.
Using this equality, the definition of the Markov operator $Q_s$ and  \eqref{LambConv-Ineq01},
we get $Q_{ s' } ( r_{s_1}^{t} r_{s_2}^{1-t} / r_{s'} ) \leq r_{s_1}^{t} r_{s_2}^{1-t} / r_{s'}.$
Applying Lemma \ref{Lem-QsIne} with $\varphi = - r_{s_1}^{t} r_{s_2}^{1-t} / r_{s'}$, this implies that 
$r_{s_1}^{t} r_{s_2}^{1-t} = c r_{s'}$ on $V(\Gamma_{\mu} )$ for some constant $c>0$.
Substituting this equality and the identity $\kappa( s' ) = \kappa^t(s_1) \kappa^{1-t}(s_2)$ 
into \eqref{LambConv-Ineq01}, we see that
the H\"{o}lder inequality in \eqref{LambConv-Ineq01} is actually an equality.
This yields that there exists a function $c(x)>0$ such that for any $g \in \Gamma_{\mu}$ and $x \in V(\Gamma_{\mu} )$,
\begin{align}\label{PrConv001}
e^{s_1 \sigma(g, x)} r_{s_1} (g \!\cdot\! x) = c(x) e^{s_2 \sigma(g, x)} r_{s_2} (g \!\cdot\! x).
\end{align}
Integrating both sides of the equation \eqref{PrConv001} with respect to $\mu$
gives $c(x) = \frac{ \kappa(s_1)  r_{s_1} ( x) }{ \kappa(s_2)  r_{s_2} ( x) }$.
Substituting this into \eqref{PrConv001} and noting that $s_1 \neq s_2$, we find that
there exist a constant $c_1 >0$ and a real-valued function $\varphi$ on $\mathcal{S}$ such that
$e^{\sigma(g, x)} = c_1 \frac{\varphi(g \cdot  x) }{ \varphi(x) }$
for any $g \in \Gamma_{\mu}$ and $x \in V(\Gamma_{\mu} )$.
This contradicts to the non-arithmetic condition \ref{CondiNonarith}.
Recall that condition \ref{CondiIP} implies condition \ref{CondiNonarith} for invertible matrices.
Hence $\Lambda$ is strictly convex for invertible matrices under stated conditions.
\end{proof}

\begin{proof}[Proof of Proposition \ref{Prop-lambTaylor}]
The expansion \eqref{Expan-lambsw}  follows from
 \eqref{relationlamkappa001} and Taylor's formula. 

For part (a), by Lemma \ref{Lem-Lamb-Conv}, we have $\Lambda''(s) \geq 0$ for any $s \in (-\eta, \eta)$. 
Since $\Lambda = \log \kappa$ and it is shown in Proposition \ref{transfer operator}
that the function $\kappa$ is real-valued and strictly positive on $(-\eta, \eta)$, we get $\Lambda'''(s) \in \mathbb{R}$. 

For part (b), recall that it was shown in \cite{BM16} that $\sigma_0>0$
for invertible matrices under the stated conditions, and for positive matrices under the additional condition
of non-arithmeticity.  
Hence, using the continuity of the function $\Lambda''$, we obtain that $\sigma_s>0$. 

For part (c),  by Proposition \ref{perturbation thm}, 
we get that for $|z|$ small, 
\begin{align}\label{Expan-SigLam002}
\mathbb{E}_{\mathbb{Q}_{s}^x} \big[ e^{z( \sigma(G_n, x) - n\Lambda'(s) )} \big]
 = \lambda^{n}_{s, z} ( \Pi_{s, z} \mathbf{1} )(x) +  ( N^{n}_{s, z} \mathbf{1} )(x).
\end{align}
It follows from \eqref{Expan-lambsw} that for 
$|z| = o(n^{-1/3})$,
\begin{align}\label{Expan-SigLam003}
\lambda^{n}_{s, z} = 1 + n \sigma_s^2 \frac{z^2}{2} + n \Lambda'''(s) \frac{z^3}{6} + o(n |z|^3).
\end{align}
Using Taylor's formula, the bound \eqref{SpGapContrPi} and the fact $\Pi_{s,0} \mathbf{1} = 1$, we obtain
\begin{align}\label{Expan-SigLam004}
 ( \Pi_{s, z} \mathbf{1} )(x) = 1 + c_{s,x,1} z + c_{s,x,2} z^2 + c_{s,x,3} z^3  + o(|z|^3), 
\end{align}
where the constants $c_{s,x,1}, c_{s,x,2}, c_{s,x,3}  \in \mathbb{C}$ are bounded as functions of $s$ and $x$. 
Similarly, using the fact $N_{s,0} \mathbf{1} = 0$ and the bound \eqref{SpGapContrN}, there exist constants
$C_{s,x,n,1},$ $C_{s,x,n,2},$ $C_{s,x,n,3} \in \mathbb{C}$ which are bounded as functions of  $s,x$ and $n$ such that
\begin{align}\label{Expan-SigLam005}
( N_{s,z}^n \mathbf{1} )(x) = C_{s,x,n,1} z + C_{s,x,n,2} z^2 + C_{s,x,n,3} z^3 + o(|z|^3).
\end{align}
Taking the second derivative on both sides of the equation \eqref{Expan-SigLam002}
with respect to $z$ at $0$,  
and using the expansions \eqref{Expan-SigLam003}-\eqref{Expan-SigLam005}, 
we deduce that
\begin{align}\label{Pf-Formu-sigmas}
 \mathbb{E}_{\mathbb{Q}_{s}^x} \big[ \sigma(G_n, x) - n\Lambda'(s)  \big]^2
=   n \sigma_s^2  + 2c_{s,x,2} + 2C_{s,x,n,2}.
\end{align}
This, together with the definition of $\mathbb{Q}_s$ and the fact that the constants $c_{s,x,2}$, $C_{s,x,n,2}$
are bounded as functions of $s,x,n$, concludes the proof of part (c).

For part (d), integrating both sides of the equations \eqref{Expan-SigLam002},
\eqref{Expan-SigLam004} and \eqref{Expan-SigLam005}
with respect to $\pi_s$, and using
the property \eqref{Pf-Pis001} with $\varphi = \mathbf{1}$ 
(thus the second term on the right-hand side of \eqref{Expan-SigLam004} vanishes),
in the same way as in the proof of \eqref{Pf-Formu-sigmas}, we get
\begin{align*}
 \mathbb{E}_{\mathbb{Q}_{s}} \big[ \sigma(G_n, x) - n\Lambda'(s)  \big]^3
= & \   n \Lambda'''(s)   + 6c_{s,3} + 6 C_{s,n,3}.
\end{align*}
This implies the desired assertion in part (d). 
\end{proof}

\begin{remark}
Inspecting the proof of Proposition \ref{Prop-lambTaylor}, 
it is easy to see that the results in parts (c) and (d) can be reinforced to the following bounds: 
\begin{align*}
& \sup_{ s \in (-\eta, \eta) }  \sup_{ x \in \mathcal{S} }
\Big| \frac{1}{n} \mathbb{E}_{\mathbb{Q}_s^x} \big[ \sigma(G_n, x) - n \Lambda'(s) \big]^2  - \sigma_s^2  \Big| \leq  \frac{C}{n}, \\
& \sup_{ s \in (-\eta, \eta) }  
\Big| \frac{1}{n} \mathbb{E}_{\mathbb{Q}_s} \big[ \sigma(G_n, x) - n \Lambda'(s) \big]^3 - \Lambda{'''}(s) \Big|
\leq  \frac{C}{n}. 
 \end{align*}
The first bound above also holds with the measure $\mathbb{Q}_s^x$ replaced by $\mathbb{Q}_s$.  
\end{remark}

\section{Smoothing inequality on the complex plane}\label{sec-BerryEss}
In this section we aim to establish a new smoothing inequality,
which plays a crucial role in proving the Berry--Esseen bound and
Edgeworth expansion 
with a target function $\varphi$ on $X_n^x$; 
see Theorems \ref{Thm-BerryEsseen-Norm:s=0}, \ref{Thm-Edge-Expan:s=0}, 
   \ref{Thm-BerryEsseen-Norm} and \ref{Thm-Edge-Expan}.

From now on, for any integrable function $h: \mathbb{R} \to \mathbb{C}$,
denote its Fourier transform by
$\widehat{h}(t) = \int_{\mathbb{R}} e^{-ity} h(y) \,dy,$  $t \in \mathbb{R}$. 
If $\widehat{h}$ is integrable on $\mathbb{R}$, then using the inverse Fourier transform gives
$h(y) = \frac{1}{2 \pi} \int_{\mathbb{R}} e^{ity} \widehat{h}(t) \,dt,$ for almost all $y \in \mathbb{R}$
with respect to the Lebesgue measure on $\mathbb R$. 
Denote by $h_1 * h_2$ the convolution of the functions $h_1$, $h_2$ on the real line.

For any $r >0$, denote
\begin{align*}
& D_{r} = \{ z \in \mathbb{C}: |z| < r \},   \notag\\
& D_{r}^{+} = \{ z \in \mathbb{C}: |z|<r,  \Im z > 0 \}     \quad  \mbox{and}  \quad    
D_{r}^{-} =  \{ z \in \mathbb{C}: |z|<r,  \Im z < 0 \}. 
\end{align*}
We construct a density function $\rho_{T}$ which plays an important role 
in establishing a new smoothing inequality. 
As in \cite{Pet75}, we define the density function $\rho$ on the real line $\mathbb R$ by setting $\rho(0) = 1 / 2\pi$ and  
\begin{align*}
\rho(y) = \frac{1}{2 \pi} \Big( \frac{\sin \frac{y}{2} }{ \frac{y}{2} } \Big)^2, \quad  y \in \mathbb R \setminus \{0\}. 
\end{align*}
Then $\rho$ is a non-negative function bounded by $\frac{1}{2\pi}$
and $\int_{\mathbb{R}} \rho(y) \, dy=1$.
Its Fourier transform $\widehat{\rho}$ is given by 
\begin{align*}
\widehat{\rho}(t) = 1 - |t|, \quad  t \in [-1, 1],
\end{align*}
and $\widehat{\rho}(t) = 0$ otherwise.

For any $T>0$ and the fixed constant $b>0$ satisfying $\int_{-b}^{b} \rho(y) \, dy = 3/4$, define the density function
\begin{align*}
\rho_{T}(y) = T \rho ( T y - b ), \quad   y \in \mathbb{R},
\end{align*} 
whose Fourier transform $\widehat{\rho}_{T}$
is given  by  
\begin{align}\label{Ch4Def_rhohatkk}
\widehat{\rho}_{T}(t) = e^{-ib \frac{t}{T}} \Big( 1 - \frac{|t|}{T} \Big),
\quad   t \in [-T, T],
\end{align}
and $\widehat{\rho}_{T}(t) = 0$ otherwise. 
Note that the function $\widehat{\rho}_{T}$ is not smooth at the point $0$, so that it
can not have an analytic extension in a small neighborhood of $0$ in the complex plane $\mathbb C$. 

Now we are ready to establish our new smoothing inequality. Its proof is based on the properties of 
the density function $\rho_{T}$, Cauchy's integral theorem and some techniques from
\cite{Ess45, Pet75}.  

\begin{proposition}\label{Prop-BerryEsseen}
Assume that $F$ is non-decreasing on $\mathbb{R}$,
and that $H$ is differentiable of bounded variation on $\mathbb{R}$ such that $ \sup_{y \in \mathbb{R}} |H'(y)|< \infty$.
Suppose that $F(-\infty) = H(-\infty) $ and $F(\infty) = H(\infty)$.
Let
\begin{align*}
f(t) = \int_{\mathbb{R}} e^{-ity} dF(y)  \quad  \mbox{and}  \quad
h(t) = \int_{\mathbb{R}} e^{-ity} dH(y),  \quad  t \in \mathbb{R}.
\end{align*}
Suppose that $r >0$ and that $f$ and $h$
have analytic extensions on $D_{r}$.  
Then, for any $ T \geq r$,
\begin{align*}
 \sup_{y \in \mathbb{R}} | F(y) - H(y) |
& \leq
    \frac{1}{\pi } \sup_{y \leq 0 }  \Big|  \int_{ \mathcal{C}_{r}^{-} }
\frac{ f(z) - h(z) }{ z } e^{izy}  e^{-i b \frac{z}{T}} \,dz \Big|  \notag\\
& \quad  +  \frac{1}{\pi } \sup_{y > 0 }   \Big|  \int_{ \mathcal{C}_{r}^{+} }
\frac{ f(z) - h(z) }{ z } e^{izy}  e^{-i b \frac{z}{T}}  \,dz \Big|  \notag\\
& \quad  +  \frac{1}{\pi  } \sup_{y \leq 0 }  \Big|  \int_{ \mathcal{C}_{r}^{-} }
\frac{ f(z) - h(z) }{ z } e^{izy} e^{i b \frac{z}{T}} \,dz  \Big|  \notag\\
& \quad  + \frac{1}{\pi  } \sup_{y > 0 }   \Big|  \int_{ \mathcal{C}_{r}^{+} }
\frac{ f(z) - h(z) }{ z } e^{izy} e^{i b \frac{z}{T}} \,dz  \Big|   \notag\\
& \quad  +  \frac{1}{\pi }  \int_{ r \leq |t| \leq T }  \Big|  \frac{ f(t) - h(t) }{ t }  \Big|  \,dt  \notag\\
& \quad  + \frac{2}{\pi T} \int_{-T}^T  | f(t) - h(t) | \,dt  +  \frac{3 b}{ T} \sup_{y \in \mathbb{R}} |H'(y)|,
\end{align*}
where $b>0$ is a fixed constant  satisfying $\int_{-b}^{b} \rho(y) \, dy = 3/4$,  
and $\mathcal{C}_{r}^{-}$ and $\mathcal{C}_{r}^{+}$ are semicircles given by 
\begin{align}\label{Def-Cr-+}
\mathcal{C}_{r}^{-} = \{ z \in \mathbb{C}: |z| = r, \Im z <0 \},   \quad
\mathcal{C}_{r}^{+} = \{ z \in \mathbb{C}: |z| = r, \Im z >0 \}. 
\end{align}
\end{proposition}

\begin{proof}
Let $T\geq r$. From the definition of $\rho_{T}$ and the choice of the constant $b$, we have
$\int_{0}^{2b/T} \rho_{T}(y) \, dy = 3/4$.
Since $\rho \leq  \frac{1}{2 \pi}$, the function $\rho_{T}$ is bounded by $T/ 2\pi$.
The proof of Proposition \ref{Prop-BerryEsseen} consists 
in establishing first an upper bound and then a lower bound.

\textit{Upper bound.}
Since the function $F$ is non-decreasing on $\mathbb{R}$ and $\rho_{T}$ is a density function on $\mathbb{R}$,
we find that for any $y \in \mathbb{R}$,  
\begin{align}\label{BE-Ineq-FG01}
   F(y)  & \leq   \frac{4}{3} \int_{y}^{y + \frac{2b}{T} } F(u) \rho_{T}(u-y) \, du   \notag\\
&  =     H(y) +  \frac{4}{3} \int_{y}^{y + \frac{2b}{T} } 
  \Big[ \big( F(u)- H(u) \big) \rho_{T}(u-y) +  \big( H(u) - H(y) \big) \rho_{T}(u-y) \Big]  \, du    \notag\\
&  \leq  H(y) +  \frac{4}{3} \int_{y}^{y + \frac{2b}{T} }  \big( F(u) - H(u) \big) \rho_{T}(u-y) \, du
+ \frac{ 2b }{ T} \sup_{y \in \mathbb{R}} |H'(y)|.
\end{align}
Let
$
F_1(y) = \int_{\mathbb{R}} F(u) \rho_{T}(u-y) \, du$,  
and 
$H_1(y) = \int_{\mathbb{R}} H(u) \rho_{T}(u-y) \, du$,  
$y \in \mathbb{R}$.
Elementary calculations lead to 
\begin{align*}
\int_{\mathbb{R}} e^{-ity} \, dF_1(y) = f(t) \widehat{\rho}_{T}(-t),  \quad
\int_{\mathbb{R}} e^{-ity} \, dH_1(y) = h(t) \widehat{\rho}_{T}(-t),  \quad t \in \mathbb{R}.
\end{align*}
Restricted on the real line, the function $\widehat{\rho}_{T}$ is supported on $[-T,T]$.
By the Fourier inversion formula 
we get
\begin{align*}
F_1(y) - F_1(v)
& =   \frac{1}{2\pi} \int_{-T}^T \frac{e^{ity} - e^{itv} }{ it } f(t) \widehat{\rho}_{T}(-t) \,dt,
\qquad y, v \in \mathbb{R},  \notag\\
H_1(y) - H_1(v)
& =  \frac{1}{2\pi} \int_{-T}^T \frac{e^{ity} - e^{itv} }{ it } h(t) \widehat{\rho}_{T}(-t) \,dt, 
\qquad  y, v \in \mathbb{R}. 
\end{align*}
By the definition of $\widehat{\rho}_T$ (cf. \eqref{Ch4Def_rhohatkk}), we get
\begin{align*}
& F_1(y) - H_1(y) - ( F_1(v) - H_1(v) )   \notag\\
& =  \frac{1}{2\pi} \int_{-T}^T \frac{f(t) - h(t)}{ it } e^{ity} e^{i b \frac{t}{T}} \,dt 
   -  \frac{1}{2\pi} \int_{-T}^T \frac{f(t) - h(t)}{ it } e^{itv} e^{i b \frac{t}{T}} \,dt  \notag\\
& \quad -  \frac{1}{2\pi} \int_{-T}^T \frac{f(t) - h(t)}{ it } e^{ity} e^{i b \frac{t}{T}} \frac{|t|}{T} \,dt 
        +  \frac{1}{2\pi} \int_{-T}^T \frac{f(t) - h(t)}{ it } e^{itv} e^{i b \frac{t}{T}} \frac{|t|}{T} \,dt.
\end{align*}
It follows that for any $y, v \in \mathbb R$, 
\begin{align}\label{Ch4_BESmoo22}
& \big| F_1(y) - H_1(y) - ( F_1(v) - H_1(v) )  \big|  \notag\\
& \leq  \frac{1}{2\pi}  \Big|  \int_{-T}^T \frac{f(t) - h(t)}{ it } e^{ity} e^{i b \frac{t}{T}} \,dt 
    -  \int_{-T}^T \frac{f(t) - h(t)}{ it } e^{itv} e^{i b \frac{t}{T}} \,dt \Big| \notag\\
& \quad +  \frac{1}{\pi T} \int_{-T}^T  | f(t) - h(t) | \,dt.
\end{align}
We shall use Cauchy's integral theorem to change the integration path
$[-T,T]$ to a contour in the complex plane.
In order to estimate the difference $| F_1(y) - H_1(y) |$, we are led to consider two cases: 
$y \leq 0$ and $y > 0$.

\textit{Control of $| F_1(y) - H_1(y) |$ when $y \leq 0$.}
Let $\mathcal{C}_{-} =  \mathcal{C}_{r,T}  \cup \mathcal{C}_{r}^{-}$,
where $\mathcal{C}_{r,T} = [-T, -r] \cup [r, T]$
and $\mathcal{C}_{r}^{-}$ is the lower semicircle given in \eqref{Def-Cr-+}. 
Since $F(-\infty) = H(-\infty)$ and $F(\infty) = H(\infty)$, 
by the definition of $f$ and $h$, we see that $f(0) = h(0)$. 
This, together with the condition that $f$ and $h$ have analytic extensions on $D_r$,
implies that $z = 0$ is a removable singular point of the function $z \in D_r \mapsto \frac{f(z) - h(z)}{z} \in \mathbb C$. 
Hence, using the fact that  
the function $z \mapsto e^{izy} e^{i b \frac{z}{T}}$ is analytic on the domain $D_{r}$,
applying Cauchy's integral theorem, we obtain that for any $y, v \in \mathbb R$, 
\begin{align}\label{BE-F1G1-equ01}
&   \int_{-T}^T \frac{f(t) - h(t)}{ it } e^{ity} e^{i b \frac{t}{T}} \,dt 
   -  \int_{-T}^T \frac{f(t) - h(t)}{ it } e^{itv} e^{i b \frac{t}{T}} \,dt     \notag\\
&  =  \int_{\mathcal{C}_{-} } \frac{ f(z) - h(z) }{ iz } e^{izy} e^{i b \frac{z}{T}} \,dz
    - \int_{\mathcal{C}_{-} } \frac{ f(z) - h(z) }{ iz } e^{izv} e^{i b \frac{z}{T}} \,dz,
\end{align}
where the integration is over the complex curve $\mathcal{C}_{-}$ oriented from $-T$ to $T$.
The second integral in \eqref{BE-F1G1-equ01} converges to $0$ as $v \to -\infty$,
by using the Riemann-Lebesgue lemma on the real segment $\mathcal{C}_{r,T}$ 
and by applying the Lebesgue convergence theorem on the semicircle $\mathcal{C}_{r}^{-}$.
Note that  $F_1(-\infty) = H_1(-\infty)$ since $F(-\infty) = H(-\infty)$.
Consequently, letting $v \to -\infty$ in \eqref{BE-F1G1-equ01}
and substituting it into \eqref{Ch4_BESmoo22}, we get
\begin{align*} 
\big| F_1(y) - H_1(y) \big| \leq  \frac{1}{2\pi} 
\Big| \int_{\mathcal{C}_{-}} \frac{ f(z) - h(z) }{ iz } e^{izy} e^{i b \frac{z}{T}} \, dz \Big|
 + \frac{1}{\pi T} \int_{-T}^T  | f(t) - h(t) | \,dt. 
\end{align*}
Therefore, recalling that $\mathcal{C}_{-} =  \mathcal{C}_{r,T}  \cup \mathcal{C}_{r}^{-}$, it follows that
\begin{align}\label{BE-Ineq-FG02}
& \sup_{y \leq 0 } |F_1(y) - H_1(y)|
 \leq  \frac{1}{2\pi} \int_{ \mathcal{C}_{r,T} } \Big| \frac{ f(t) - h(t) }{ t } \Big|  \,dt  \notag\\
& \quad  +  \frac{1}{2\pi} \sup_{y \leq 0 }  
  \Big|  \int_{ \mathcal{C}_{r}^{-} } \frac{ f(z) - h(z) }{ z } e^{izy} e^{i b \frac{z}{T}} \, dz \Big|  
  +  \frac{1}{\pi T} \int_{-T}^T  | f(t) - h(t) | \,dt. 
\end{align}

\textit{Control of $| F_1(y) - H_1(y) |$ when $y > 0$.}
Let $\mathcal{C}_{+} =  \mathcal{C}_{r,T}  \cup \mathcal{C}_{r}^{+}$,
where $\mathcal{C}_{r,T} = [-T, -r] \cup [r, T]$
and $\mathcal{C}_{r}^{+}$ is the upper semicircle  given in \eqref{Def-Cr-+}.
In an analogous way as in \eqref{BE-F1G1-equ01}, applying Cauchy's integral theorem we have
\begin{align}\label{BE-F1G1-equ02}
&   \int_{-T}^T \frac{f(t) - h(t)}{ it } e^{ity} e^{i b \frac{t}{T}} \,dt 
   -  \int_{-T}^T \frac{f(t) - h(t)}{ it } e^{itv} e^{i b \frac{t}{T}} \,dt     \notag\\
&  =  \int_{\mathcal{C}_{+} } \frac{ f(z) - h(z) }{ iz } e^{izy} e^{i b \frac{z}{T}} \, dz
    - \int_{\mathcal{C}_{+} } \frac{ f(z) - h(z) }{ iz } e^{izv} e^{i b \frac{z}{T}} \, dz,
\end{align}
where the integration is over the complex curve $\mathcal{C}_{+}$ also oriented from $-T$ to $T$.
The second integral in \eqref{BE-F1G1-equ02} converges to $0$ as $v \to +\infty$,
by using again the Riemann-Lebesgue lemma on the real segment $\mathcal{C}_{r,T}$ 
and by applying the Lebesgue convergence theorem on the upper semicircle $\mathcal{C}_{r}^{+}$.
Note that  $F_1(\infty) = H_1(\infty)$ since $F(\infty) = H(\infty)$.
Hence, letting $v \to +\infty$ in \eqref{BE-F1G1-equ02},
similarly to \eqref{BE-Ineq-FG02}, we obtain
\begin{align}\label{BE-Ineq-FG-UB0}
& \sup_{y > 0 } |F_1(y) - H_1(y)|   \leq
\frac{1}{2\pi} \int_{ \mathcal{C}_{r,T} } \Big| \frac{ f(t) - h(t) }{ t } \Big|  \,dt \notag\\
 &  \qquad  + \frac{1}{2\pi} \sup_{y > 0 }   
  \Big|  \int_{ \mathcal{C}_{r}^{+} } \frac{ f(z) - h(z) }{ z } e^{izy} e^{i b \frac{z}{T}} \, dz  \Big|  
  +  \frac{1}{\pi T} \int_{-T}^T  | f(t) - h(t) | \,dt. 
\end{align}

Putting together \eqref{BE-Ineq-FG02} and \eqref{BE-Ineq-FG-UB0} leads to 
\begin{align}\label{Be-Ineq-UR}
\sup_{y \in \mathbb{R} } |F_1(y) - H_1(y)|   
& \leq  
\frac{1}{2\pi} \int_{ \mathcal{C}_{r,T} }  \Big| \frac{ f(t) - h(t) }{ t } \Big|  \,dt 
 +  \frac{1}{2\pi} \sup_{y \leq 0 }  \Big|  \int_{ \mathcal{C}_{r}^{-} }
\frac{ f(z) - h(z) }{ z } e^{izy} e^{i b \frac{z}{T}} \, dz  \Big|  \notag\\
 & \quad + \frac{1}{2\pi} \sup_{y > 0 }   \Big|  \int_{ \mathcal{C}_{r}^{+} }
\frac{ f(z) - h(z) }{ z } e^{izy} e^{i b \frac{z}{T}} \, dz \Big|   
 +  \frac{1}{\pi T} \int_{-T}^T  | f(t) - h(t) | \,dt. 
\end{align}
Denote $\Delta = \sup_{y \in \mathbb{R} } |F(y) - H(y)|$. 
Then, taking into account that $\rho_{T}$ is a density function on $\mathbb{R}$, 
using \eqref{Be-Ineq-UR} and the fact that $\int_{0}^{2b/T} \rho_{T}(y) \, dy = 3/4$, 
 we get that for any $y \in \mathbb{R}$, 
\begin{align*}
&  \Big| \int_{y}^{y + \frac{2b}{T} } \big( F(u)- H(u) \big) \rho_{T}(u-y) \, du \Big|   \notag\\
&  \leq  |F_1(y) - H_1(y)|  +  \Delta  \bigg(   1 - \int_{y}^{y + \frac{2b}{T}} \rho_{T}(u-y) \, du \bigg)  \notag\\
&  \leq  \frac{1}{2\pi} \int_{ \mathcal{C}_{r,T} } \Big| \frac{ f(t) - h(t) }{ t } \Big|  \,dt
+  \frac{1}{2\pi} \sup_{y \leq 0 }  \Big|  \int_{ \mathcal{C}_{r}^{-} }
\frac{ f(z) - h(z) }{ z } e^{izy} e^{i b \frac{z}{T}} \, dz \Big|  \notag\\
 & \quad  + \frac{1}{2\pi} \sup_{y > 0 }   \Big|  \int_{ \mathcal{C}_{r}^{+} }
\frac{ f(z) - h(z) }{ z } e^{izy} e^{i b \frac{z}{T}} \, dz  \Big|
 + \frac{1}{\pi T} \int_{-T}^T  | f(t) - h(t) | \,dt + \frac{ \Delta }{4}.
\end{align*}
Substituting this inequality into \eqref{BE-Ineq-FG01}, we obtain the following desired upper bound: 
for any $y \in \mathbb{R}$, 
\begin{align}\label{BE-Inequ-Upper}
 F(y) - H(y)  & \leq   \frac{2}{3\pi}
\int_{ \mathcal{C}_{r,T} }  \Big| \frac{ f(t) - h(t) }{ t }  \Big|  \,dt  \notag\\
&  \quad   +   \frac{2}{3\pi} \sup_{y \leq 0 }  \Big|  \int_{ \mathcal{C}_{r}^{-} }
\frac{ f(z) - h(z) }{ z } e^{izy} e^{i b \frac{z}{T}} \, dz \Big|  \notag\\
&  \quad   +  \frac{2}{3\pi} \sup_{y > 0 }   \Big|  \int_{ \mathcal{C}_{r}^{+} }
    \frac{ f(z) - h(z) }{ z } e^{izy} e^{i b \frac{z}{T}} \, dz  \Big|   \notag\\
&  \quad  + \frac{4}{3\pi T} \int_{-T}^T  | f(t) - h(t) | \,dt
 +  \frac{ \Delta }{3}  +  \frac{ 2 b}{ T } \sup_{y \in \mathbb{R}} |H'(y)|.
\end{align}

\textit{Lower bound.}
Similarly to the upper bound \eqref{BE-Ineq-FG01}, 
using the fact that $F$ is non-decreasing and $\rho_{T}$ is a density function on $\mathbb{R}$,
we have for any $y \in \mathbb{R}$, 
\begin{align*}
F(y) & \geq  \frac{4}{3} \int_{ y - \frac{2b}{T} }^{y } F(u) \rho_{T}(y - u) \, du   \notag\\
&  \geq   H(y) +  \frac{4}{3} \int_{y - \frac{2b}{T} }^{y} \big( F(u)-H(u) \big) \rho_{T}(y-u) \, du
  -  \frac{ 2 b }{ T} \sup_{y \in \mathbb{R}} |H'(y)|.
\end{align*}
Let $F_2(y) = (F * \rho_{T})(y)$ and $H_2(y) = (H * \rho_{T})(y)$,
$y \in \mathbb{R}$. 
Then, 
\begin{align*}
\int_{\mathbb{R}} e^{-ity} \, dF_2(y) = f(t) \widehat{\rho}_{T}(t),  \quad
\int_{\mathbb{R}} e^{-ity} \, dH_2(y) = h(t) \widehat{\rho}_{T}(t), \quad  t \in \mathbb{R}.
\end{align*}
Proceeding in the same way as in the proof of \eqref{Be-Ineq-UR}, one has
\begin{align*} 
& \sup_{y \in \mathbb{R} } |F_2(y) - H_2(y)|   \notag\\
& \leq   
\frac{1}{2\pi} \int_{ \mathcal{C}_{r,T} } \Big| \frac{ f(t) - h(t) }{ t } \Big|  \,dt 
  +  \frac{1}{2\pi} \sup_{y \leq 0 }  \Big|  \int_{ \mathcal{C}_{r}^{-} }
\frac{ f(z) - h(z) }{ z } e^{izy} e^{-i b \frac{z}{T}} \, dz \Big|  \notag\\
 & \quad  + \frac{1}{2\pi} \sup_{y > 0 }   \Big|  \int_{ \mathcal{C}_{r}^{+} }
\frac{ f(z) - h(z) }{ z } e^{izy} e^{-i b \frac{z}{T}} \, dz \Big|  
 +  \frac{1}{\pi T} \int_{-T}^T  | f(t) - h(t) | \,dt. 
\end{align*}
Following the proof of \eqref{BE-Inequ-Upper}, we obtain the lower bound:
for any $y \in \mathbb{R}$, 
\begin{align}\label{BE-Inequ-Lower}
F(y) - H(y)   & \geq   - \frac{2}{3\pi} \int_{ \mathcal{C}_{r,T} } \Big| \frac{ f(t) - h(t) }{ t } \Big|  \,dt   \notag\\
& \quad  -  \frac{2}{3\pi}  \sup_{y \leq 0 }  \Big|  \int_{ \mathcal{C}_{r}^{-} }
\frac{ f(z) - h(z) }{ z } e^{izy} e^{-i b \frac{z}{T}} \,dz \Big|  \notag\\
 & \quad -  \frac{2}{3\pi}  \sup_{y > 0 }   \Big|  \int_{ \mathcal{C}_{r}^{+} }
 \frac{ f(z) - h(z) }{ z } e^{izy} e^{-i b \frac{z}{T}} \,dz  \Big|   \notag\\
&\quad  -\frac{4}{3\pi T} \int_{-T}^T  | f(t) - h(t) | \,dt 
    -  \frac{ \Delta }{3}  -  \frac{ 2 b }{ T} \sup_{y \in \mathbb{R}} |H'(y)|.
\end{align}
Combining \eqref{BE-Inequ-Upper} and \eqref{BE-Inequ-Lower}, we conclude
the proof of Proposition \ref{Prop-BerryEsseen}.
\end{proof}

\section{Proofs of Berry--Esseen bound and Edgeworth expansion} \label{sec:proof berry-esseen-edgeworth}

\subsection{Berry--Esseen bound and Edgeworth expansion under the changed measure}

We first present a Berry--Esseen bound under the changed measure $\mathbb Q_s^x$. 

\begin{theorem}\label{Thm-BerryEsseen-Norm}
Assume either conditions \ref{CondiMoment} and \ref{CondiIP} for invertible matrices,
or conditions \ref{CondiMoment}, \ref{CondiAP} and \ref{Condi-Variance} for positive matrices.
Then there exist constants $\eta>0$ and $C>0$ such that
for all $n \geq 1$, $s \in (-\eta, \eta)$, $x \in \mathcal{S}$, $y \in \mathbb{R}$ and   $\varphi \in \mathcal{B}_{\gamma}$, 
\begin{align}\label{Th-BerEssMatr001}
\Big| \mathbb{E}_{\mathbb{Q}_s^x}
 \Big[  \varphi(X_n^x) \mathds{1}_{ \big\{ \frac{\sigma(G_n, x) - n \Lambda'(s) }{\sigma_s \sqrt{n}} \leq y \big\} } \Big]
    -  \pi_s(\varphi)  \Phi(y) \Big|   
 \leq  \frac{C}{\sqrt{n}}  \lVert \varphi \rVert_{\gamma}. 
\end{align}
\end{theorem}

The next result gives an Edgeworth expansion for $(X_n^x, \sigma(G_n, x))$
with a target function $\varphi$ on $X_n^x$ under $\mathbb Q_s^x$.
The function $b_{s,\varphi}(x),$ $x \in \mathcal{S}$, which will be used in the formulation of this result, is defined in Lemma \ref{Lem-Bs} and
has an equivalent expression \eqref{Def2-bs}  
in terms of derivative of 
the projection operator $\Pi_{s,z}$, see Proposition \ref{perturbation thm}.

\begin{theorem}\label{Thm-Edge-Expan-bb}
Assume either conditions \ref{CondiMoment} and \ref{CondiIP} for invertible matrices,
or conditions \ref{CondiMoment}, \ref{CondiAP} and \ref{CondiNonarith} for positive matrices.
Then there exists $\eta>0$ such that
as $n \to \infty$, uniformly in   $s \in (-\eta, \eta)$, $x \in \mathcal{S}$, $y \in \mathbb{R} $ and $\varphi \in \mathcal{B}_{\gamma}$,
\begin{align*}
& \bigg| \mathbb{E}_{\mathbb{Q}_s^x} \Big[  \varphi(X_n^x) 
   \mathds{1}_{ \big\{ \frac{\sigma(G_n, x) - n \Lambda'(s) }{\sigma_s \sqrt{n}} \leq y \big\} } \Big] \\
& \qquad -  \mathbb{E}_{\mathbb{Q}_s^x} \big[ \varphi(X_n^x) \big] 
 \Big[  \Phi(y) + \frac{\Lambda'''(s)}{ 6 \sigma_s^3 \sqrt{n}} (1-y^2) \phi(y) \Big] 
 + \frac{ b_{s,\varphi}(x) }{ \sigma_s \sqrt{n} } \phi(y)   \bigg|
=  \lVert \varphi \rVert_{\gamma}  o \Big( \frac{ 1 }{\sqrt{n}} \Big). 
\end{align*}
\end{theorem}

The following asymptotic expansion is slightly different from that in Theorem \ref{Thm-Edge-Expan-bb}, 
with the term $\mathbb{E}_{\mathbb{Q}_s^x} [ \varphi(X_n^x) ]$
replaced by $\pi_s(\varphi)$: 

\begin{theorem}\label{Thm-Edge-Expan}
Under the conditions of Theorem \ref{Thm-Edge-Expan-bb}, 
there exists  $\eta>0$ such that, 
as $n \to \infty$, uniformly in   $s \in (-\eta, \eta)$, $x \in \mathcal{S}$, $y \in \mathbb{R} $ 
and $\varphi \in \mathcal{B}_{\gamma}$, 
\begin{align}
&  \bigg| \mathbb{E}_{\mathbb{Q}_s^x}
   \Big[  \varphi(X_n^x) \mathds{1}_{ \big\{ \frac{\sigma(G_n, x) - n \Lambda'(s) }{\sigma_s \sqrt{n}} \leq y \big\} } \Big]
   \notag \\
 &  \qquad  - \pi_s(\varphi) \Big[  \Phi(y) + \frac{\Lambda'''(s)}{ 6 \sigma_s^3 \sqrt{n}} (1-y^2) \phi(y) \Big]
  + \frac{ b_{s,\varphi}(x) }{ \sigma_s \sqrt{n} } \phi(y)   \bigg|
  =  \lVert \varphi \rVert_{\gamma} o \Big( \frac{ 1 }{\sqrt{n}} \Big).  \label{MyEdgExp001:s=0}   
\end{align}
\end{theorem}
With fixed $s > 0$ and $\varphi = \mathbf{1}$, the expansion \eqref{MyEdgExp001:s=0} has been established earlier in \cite{BM16}.

The assertion of Theorem \ref{Thm-Edge-Expan} follows from Theorem  \ref{Thm-Edge-Expan-bb}, 
since the bound \eqref{Opera-Nsn} implies that there exist constants $c, C>0$ 
such that uniformly in $\varphi \in \mathcal{B}_{\gamma}$,
\begin{align}\label{BEExpBound}
\sup_{ s \in (-\eta, \eta) } \sup_{ x \in \mathcal{S} }
\big| \mathbb{E}_{\mathbb{Q}_s^x} [ \varphi(X_n^x) ] - \pi_s(\varphi) \big| 
\leq C e^{- cn} \lVert \varphi \rVert_{\gamma}.
\end{align}

Theorems \ref{Thm-BerryEsseen-Norm:s=0} and \ref{Thm-Edge-Expan:s=0} follow from the above theorems taking $s=0$ and recalling the fact that $\Lambda'(0)=\lambda$, 
$\sigma_0=\sigma$ and $b_{0,\varphi}=b_{\varphi}$.

\subsection{Proof of Theorem \ref{Thm-Edge-Expan-bb}}

Without loss of generality, we assume that the target function $\varphi$ is non-negative on $\mathcal{S}$. 
For any $x \in \mathcal{S}$, denote
\begin{align*}
F(y) & = \mathbb{E}_{\mathbb{Q}_s^x}
\Big[  \varphi(X_n^x) \mathds{1}_{ \big\{ \frac{\sigma(G_n, x) - n \Lambda'(s) }{\sigma_s \sqrt{n}} \leq y \big\} } \Big],  \quad y\in \mathbb{R},   \notag\\
H(y) & =   \mathbb{E}_{\mathbb{Q}_s^x} [ \varphi(X_n^x) ]
 \Big[ \Phi(y) + \frac{\Lambda'''(s)}{ 6 \sigma_s^3 \sqrt{n}} (1-y^2) \phi(y) \Big]
  - \frac{ b_{s,\varphi}(x) }{ \sigma_s \sqrt{n} } \phi(y),  \quad  y\in \mathbb{R}.
\end{align*}
 Define
\begin{align*}
f(t) = \int_{\mathbb{R}} e^{-ity} dF(y),   \qquad
h(t) = \int_{\mathbb{R}} e^{-ity} dH(y),  \quad   t \in \mathbb{R}.
\end{align*}
By straightforward calculations we have that for any $x \in \mathcal{S}$, 
\begin{align}
f(t) & =  \mathbb{E}_{\mathbb{Q}_s^x}
   \Big[  \varphi(X_n^x) e^{-it \frac{\sigma(G_n, x) - n \Lambda'(s) }{\sigma_s \sqrt{n}} } \Big]
  =  R^n_{s, \frac{-it}{ \sigma_s \sqrt{n} } } \varphi(x),  \quad   t \in \mathbb{R},  \label{Def-hat-Fnsx}  \\
h(t) & =  e^{-\frac{t^2}{2}}   \Big\{
\Big[  1 - (it)^3 \frac{\Lambda'''(s)}{ 6 \sigma_s^3 \sqrt{n}} \Big] 
 R_{s,0}^n \varphi(x)  -  it  \frac{ b_{s,\varphi}(x) }{ \sigma_s \sqrt{n} } 
  \Big\},   \quad   t \in \mathbb{R}.  \label{Def-hat-Hnsx}
\end{align}
It is clear that $F(-\infty) = H(-\infty) = 0$ and $F(\infty) = H(\infty)$. 
Moreover, one can verify that the functions $F, H$
and their corresponding  Fourier-Stieltjes transforms $f, h$ satisfy the conditions
of Proposition \ref{Prop-BerryEsseen}
 for $r=\delta_1\sqrt{n} $, with some $\delta_1>0$ sufficiently small.
Hence, by Proposition \ref{Prop-BerryEsseen} we get that for any real $T\geq r$,
\begin{align}\label{BerryEsseen001}
  \sup_{y \in \mathbb{R}}  \big| F(y) - H(y)  \big|
\leq  \frac{1}{\pi }  ( I_1 + I_2 + I_3 + I_4),
\end{align}
where
\begin{align}
I_1  & =  \frac{ 3 \pi b }{ T} \sup_{y \in \mathbb{R}} |H'(y)|,  
\quad  \   I_2 =  \int_{ r \leq |t| \leq T }  \Big| \frac{ f(t) - h(t) }{ t } \Big|  \,dt,  \notag\\
I_3  & =  \sup_{y \leq 0 }  
 \Big|  \int_{ \mathcal{C}_{r}^{-} }  \frac{ f(z) - h(z) }{ z } e^{izy}  e^{-i b \frac{z}{T}}  \,dz  \Big|  
   +  \sup_{y > 0 }   \Big|  \int_{ \mathcal{C}_{r}^{+} }
\frac{ f(z) - h(z) }{ z } e^{izy}  e^{-i b \frac{z}{T}}  \,dz  \Big|  \notag\\
&  \quad  +   \sup_{y \leq 0 }  \Big|  \int_{ \mathcal{C}_{r}^{-} } \frac{ f(z) - h(z) }{ z } e^{izy}  e^{i b \frac{z}{T}} \, dz  \Big|  
  +  \sup_{y > 0 }   \Big|  \int_{ \mathcal{C}_{r}^{+} } \frac{ f(z) - h(z) }{ z } e^{izy} e^{i b \frac{z}{T}} \,  dz  \Big| \notag\\
&  =:  I_{31} +  I_{32} + I_{33} + I_{34},    \notag\\   
I_4  &  =  \frac{2}{T} \int_{-T}^T  | f(t) - h(t) | \,dt,  \label{Pf-Thm2-DecomI3}
\end{align}
with the constant $b>0$ and the complex contours $\mathcal{C}_{r}^{-}, \mathcal{C}_{r}^{+}$  
defined in \eqref{Def-Cr-+}.

By virtue of \eqref{BerryEsseen001}, in order to establish Theorem \ref{Thm-Edge-Expan-bb} 
it suffices to prove that, as $n \to \infty$, uniformly in $s \in (-\eta, \eta)$, $x \in \mathcal{S}$
and $\varphi \in \mathcal{B}_{\gamma}$,  
\begin{align}\label{EW-Expan-Toprove}
I_1 + I_2 + I_3 + I_4 = \lVert \varphi \rVert_{\gamma} o\Big(\frac{ 1 }{ \sqrt{n} } \Big).
\end{align}

\textit{Control of $I_1$.} 
From \eqref{BEExpBound} we deduce that uniformly in  $\varphi \in \mathcal{B}_{\gamma}$,
\begin{align}\label{BoundR0varphi}
\sup_{ s \in (-\eta, \eta) } \sup_{ x \in \mathcal{S} }
\big| \mathbb{E}_{\mathbb{Q}_s^x} [ \varphi(X_n^x) ]  \big|  \leq C \lVert \varphi \rVert_{\gamma}. 
\end{align}
By the formula \eqref{Def2-bs} and the bound \eqref{SpGapContrPi}, 
we get that uniformly in $\varphi \in \mathcal{B}_{\gamma}$,
\begin{align}\label{Boundbvarphi}
\sup_{ s \in (-\eta, \eta) } \sup_{ x \in \mathcal{S} }
| b_{s,\varphi}(x) | \leq C \lVert \varphi \rVert_{\gamma}.
\end{align}
Using the bounds \eqref{BoundR0varphi} and \eqref{Boundbvarphi}, 
and taking into account that $\sigma_s^2>0$ and $\Lambda'''(s) \in \mathbb{R}$ 
are bounded by a constant independent of $s \in (-\eta, \eta)$,
we obtain that $|H'(y)|$ is bounded by $c_1 \lVert \varphi \rVert_{\gamma}$, uniformly in  
$s \in (-\eta, \eta)$, $x \in \mathcal{S}$, $y \in \mathbb{R}$
and $\varphi \in \mathcal{B}_{\gamma}$.  
Hence, for any $\varepsilon >0$, we can choose $a>0$ large enough
such that for $T = a \sqrt{n}$, uniformly in $\varphi \in \mathcal{B}_{\gamma}$, 
\begin{align}\label{Edge-Contr-I3}
\sup_{ s \in (-\eta, \eta) } \sup_{ x \in \mathcal{S} } 
I_1 \leq  \frac{ 3 \pi b c_1}{ T} \lVert \varphi \rVert_{\gamma}  < \frac{\varepsilon}{ \sqrt{n}} \lVert \varphi \rVert_{\gamma}.
\end{align}

\textit{Control of $I_2$.}
Since $\sigma_m: = \inf_{s \in (-\eta, \eta)} \sigma_s>0$,
we can pick $\delta_1$ small enough such that 
$0< \delta_1 < \min\{a, \delta \sigma_m /2 \}$, 
where $\delta>0$ is the constant given in Proposition \ref{perturbation thm}.  
Then, with $r = \delta_1 \sqrt{n}$ we bound $I_2$ as follows:
\begin{align} \label{two integrals001}
I_2  \leq    \int_{\delta_1 \sqrt{n} < |t| \leq a \sqrt{n}} \Big| \frac{f(t)}{t}   \Big| \,dt 
  + \int_{\delta_1 \sqrt{n} < |t| \leq a \sqrt{n}} \Big| \frac{h(t)}{t}   \Big| \,dt.
\end{align}
Let $\sigma_M: = \sup_{s \in (-\eta, \eta)} \sigma_s$. 
It holds that $0< \sigma_M < \infty$. 
On the right-hand side of \eqref{two integrals001},
using Proposition \ref{Prop-UnifR} with
$K = \{ t \in \mathbb{R}: \delta_1 /\sigma_M \leq |t| \leq  a/\sigma_m  \}$,
the first integral is bounded by $C e^{- c n } \lVert \varphi \rVert_{\gamma}$, uniformly in 
$s \in (-\eta, \eta)$, $x \in \mathcal{S}$ and $\varphi \in \mathcal{B}_{\gamma}$; 
the second integral, by the bounds \eqref{BoundR0varphi} and \eqref{Boundbvarphi} and direct calculations,  
is bounded by $C e^{- c n } \lVert \varphi \rVert_{\gamma}$,
also uniformly in  $s \in (-\eta, \eta)$, $x \in \mathcal{S}$ and $\varphi \in \mathcal{B}_{\gamma}$. 
Consequently, we conclude that uniformly in $\varphi \in \mathcal{B}_{\gamma}$, 
\begin{align}\label{Edge-Contr-I1}
\sup_{ s \in (-\eta, \eta) } \sup_{ x \in \mathcal{S} } 
\  I_2  \leq C e^{-cn} \lVert \varphi \rVert_{\gamma}.
\end{align}

\textit{Control of $I_3$.}
Recall that the term $I_3$ is decomposed into four terms in  \eqref{Pf-Thm2-DecomI3}.    
We will only deal with $I_{31}$, since $I_{32}, I_{33}, I_{34}$ can be treated in a similar way. 
In view of \eqref{Def-hat-Fnsx} and \eqref{Def-hat-Hnsx}, 
by the spectral gap decomposition \eqref{perturb001}, we get 
\begin{align} \label{Decom-f-h}
f(z)  - h(z) =  J_{1}(z)  +  J_{2}(z) + J_{3}(z)  +  J_{4}(z),
\end{align}
where
\begin{align}
J_{1}(z) & =  \pi_s(\varphi)  \Big\{  
 \lambda^{n}_{s, \frac{ -iz }{\sigma_s \sqrt{n}} }
     - e^{- \frac{ z^2 }{2}} \Big[ 1 - (iz)^3 \frac{\Lambda'''(s)}{ 6 \sigma_s^3 \sqrt{n} } \Big]
     \Big\},  \label{Def-J1} \\
J_{2}(z) & =  \lambda^{n}_{s, \frac{-iz}{\sigma_s \sqrt{n}} } 
   \Big[ \Pi_{s, \frac{-iz}{\sigma_s \sqrt{n}}} \varphi(x)
    - \pi_s(\varphi) + iz \frac{b_{s, \varphi}(x)}{ \sigma_s \sqrt{n} } \Big],  \label{Def-J2} \\
J_{3}(z) & =   iz \frac{b_{s, \varphi}(x)}{ \sigma_s \sqrt{n} }
 \Big( e^{-\frac{z^2}{2}} - \lambda^{n}_{s, \frac{-iz}{\sigma_s \sqrt{n}} }  \Big),  \label{Def-J3} \\
J_{4}(z) & =   N_{s, \frac{-iz}{\sigma_s \sqrt{n}}}^n \varphi(x)
         -  N_{s, 0}^n \varphi(x) e^{- \frac{ z^2 }{2}} 
           \Big[ 1 - (iz)^3 \frac{\Lambda'''(s)}{ 6 \sigma_s^3 \sqrt{n} }  \Big].   \label{Def-J4}
\end{align}
With the above notation, 
we use the decomposition \eqref{Decom-f-h} to bound  $I_{31}$ in \eqref{Pf-Thm2-DecomI3} as follows:
\begin{align}\label{Def-A1-4}
I_{31} \leq \sum_{k=1}^4 A_k,   \quad \mbox{where} \quad
A_k : =  \sup_{y \leq 0 }  \Big|  \int_{ \mathcal{C}_{r}^{-} }
\frac{ J_k(z) }{ z } e^{izy}  e^{-i b \frac{z}{T}}  \,dz \Big|.
\end{align}
We now  give bounds of $A_k$, $1 \leq k \leq 4$, in a series of lemmata.
Let us start by showing an elementary inequality, which will be used repeatedly in the sequel.
Let $[z_1,z_2] = \{ z_1 + \theta (z_2 - z_1)): 0\leq \theta \leq 1\} $ 
be the complex segment with the endpoints $z_1$ and $z_2.$  

\begin{lemma}\label{Lem-MeanIne}
Let $f$ be an analytic function on the open convex domain $D \subseteq \mathbb{C}$. 
Then for any $z_1, z_2 \in  D$, and  $n \geq 1$,
\begin{align*}
\Big| f(z_2) - \sum_{k = 0}^{n-1} \frac{f^{(k)}(z_1)}{k!} (z_2 - z_1)^k  \Big|
\leq \frac{ \sup_{ z\in [z_1,z_2]}  |f^{(n)}(z)|}{n!} |z_2 - z_1|^n.
\end{align*}
\end{lemma}

\begin{proof}
The proof of this inequality can be carried out by induction. The inequality clearly holds
for $n = 1$ since for any $z_1, z_2 \in D$,
\begin{align} \label{BBBIN001}
| f(z_2) - f(z_1) | = \Big| \int_{[z_1,z_2]} f'(z) \,dz \Big|
\leq \sup_{z \in [z_1,z_2]} |f'(z)| |z_2 - z_1|.
\end{align}
For $n \geq 2$, applying \eqref{BBBIN001} to
$F(z) = f(z) - \sum_{k = 1}^{n-1} \frac{f^{(k)}(z_1)}{k!} (z - z_1)^k$, 
$z \in D$,
leads to the desired assertion.
\end{proof}

Now we are ready to establish a bound for each term $A_k$. 
The proof is based on the saddle point method. 
To be more precise, we deform the integration path, which passes through
a suitable point related to the saddle point, to minimise the integral in $A_k$ 
(see \eqref{Def-A1-4}).

\begin{lemma}\label{EW-Contr-I22}
Let $\mathcal{C}_{r}^{-}$ be defined by \eqref{Def-Cr-+}
with $r = \delta_1 \sqrt{n}$ and $\delta_1>0$ small enough. 
Then, for $T = a \sqrt{n}$ with $a>0$ large enough, uniformly in $x \in \mathcal{S}$, $s \in (-\eta, \eta)$ and $\varphi \in \mathcal{B}_{\gamma}$, 
\begin{align*}
A_1 =  \sup_{y \leq 0 }  \Big|  \int_{ \mathcal{C}_{r}^{-} } \frac{ J_1(z) }{ z } e^{izy}  e^{-i b \frac{z}{T}}  \,dz \Big|
\leq \frac{c}{n} \lVert \varphi \rVert_{\gamma}.
\end{align*}
\end{lemma}

\begin{proof}
In view of \eqref{relationlamkappa001}, using $\Lambda = \log \kappa$ and Taylor's formula, we have
\begin{align}\label{EW-lambda-expan}
\lambda^{n}_{s, \frac{ -iz }{\sigma_s \sqrt{n}} }
=  e^{-\frac{z^2}{2}}
 e^{ n \sum_{k=3}^{\infty} \frac{\Lambda^{(k)} (s) }{ k! } ( -\frac{iz}{ \sigma_s \sqrt{n} } )^k }.
\end{align}
For brevity, for any $z \in \mathcal{C}_{r}^{-}$, denote
\begin{align}\label{EW-Def-hnz}
h_1(z) = \frac{ 1 }{ z } 
\Big[ e^{ n \sum_{k=3}^{\infty} \frac{\Lambda^{(k)} (s) }{ k! } ( -\frac{iz}{ \sigma_s \sqrt{n} } )^k }
  - 1 - (-iz)^3 \frac{\Lambda'''(s)}{ 6 \sigma_s^3 \sqrt{n} } \Big]  e^{-i b \frac{z}{T}}. 
\end{align}
Then, in view of \eqref{Def-J1},  the term  $A_{1}$ can be rewritten as
\begin{align} \label{termA_1-001}
A_{1} =  \pi_s(\varphi) \sup_{y \leq 0 }  
\Big|  \int_{ \mathcal{C}_{r}^{-} }  e^{ - \frac{z^2}{2} + izy} h_1(z) \,dz  \Big|.
\end{align}
The main contribution to the integral in \eqref{termA_1-001}  is given by the saddle point $z = iy$ which is the solution of  
the equation $\frac{d}{dz} ( - \frac{z^2}{2} + izy ) = 0.$
Denote by $D_{2r}^{-} = \{ z \in \mathbb{C}: |z|< 2r,  \Im z < 0 \}$ 
the domain on analyticity of $h_1$, where $r = \delta_1 \sqrt{n}$ with $\delta_1 > 0$ small enough. 
Set
\begin{align}\label{EW-Def-yn}
y_n = \min \big\{-y, \delta_1 \sqrt{n} \big\}.
\end{align}
When $- \delta_1 \sqrt{n} \leq y \leq 0$, the saddle point $iy$ belongs to $D_{2r}^{-}$.
By Cauchy's integral theorem, we change the integration in \eqref{termA_1-001} to a rectangular path inside the domain on analyticity 
$D_{2r}^{-} $ which passes through the saddle point. 
When $y< - \delta_1 \sqrt{n}$ is large,
the saddle point $iy$ is outside the domain $D_{2r}^{-}$. 
In this case we choose a rectangular path inside $D_{2r}^{-}$ which 
passes through the point $-iy_n = -i\delta_1 \sqrt{n}$.
Note that $\pi_s(\varphi)$ is bounded by $c_1 \lVert \varphi \rVert_{\gamma}$
uniformly in $s \in ( -\eta, \eta)$ and $\varphi \in \mathcal{B}_{\gamma}$. 
Since the function $h_1$ has an analytic extension on the domain $D_{2r}^{-}$ with $r = \delta_1 \sqrt{n}$,
applying Cauchy's integral theorem, we deduce that
\begin{align}  \label{Def-A1}
  A_{1} & \leq  c_1 \lVert \varphi \rVert_{\gamma}  \   \sup_{y \leq 0 }  
  \bigg|  \bigg\{  \int_{ -\delta_1 \sqrt{n} }^{ -\delta_1 \sqrt{n} - i y_n }
  + \int_{ \delta_1 \sqrt{n} - i y_n }^{ \delta_1 \sqrt{n} }  \bigg\}
e^{ - \frac{z^2}{2} + izy} h_1(z)  \, dz  \bigg|   \notag\\
 &  \quad  +  c_1  \lVert \varphi \rVert_{\gamma}  \   \sup_{y \leq 0 }   
 \bigg| \int_{ -\delta_1 \sqrt{n} - i y_n }^{ \delta_1 \sqrt{n} - i y_n }
e^{ - \frac{z^2}{2} + izy} h_1(z) \, dz   \bigg|    \notag\\
&  =:  c_1  \lVert \varphi \rVert_{\gamma}  ( A_{11} + A_{12} ).
\end{align}

\textit{Control of $A_{11}$.}
Using a change of variable, we get
\begin{align}\label{EW-I22-1}
  A_{11}  &  =     e^{ - \frac{ \delta_1^2 }{2} n  } \sup_{y \leq 0 }
\bigg|  \int_{ 0 }^{ y_n }
e^{ \frac{t^2}{2} + ty - i \delta_1 \sqrt{n} (t+y) } h_1( -\delta_1 \sqrt{n} - i t ) \,dt   \notag\\
&  \qquad   \qquad \qquad \qquad -  \int_{ 0 }^{ y_n }
e^{ \frac{t^2}{2} + ty + i \delta_1 \sqrt{n} (t+y) } h_1( \delta_1 \sqrt{n} - i t ) \,dt \bigg|  \notag\\
& \leq   e^{ - \frac{ \delta_1^2 }{2} n  } \sup_{y \leq 0 }  
 \bigg|  \int_{ 0 }^{ y_n } e^{ \frac{t^2}{2} + ty } 
 \big\{ | h_1( -\delta_1 \sqrt{n} - i t ) | + |h_1( \delta_1 \sqrt{n} - i t )|  \big\} \,dt  \bigg|.
\end{align}
We first give a bound for $| h_1( \pm \delta_1 \sqrt{n} - i t ) |$. 
Since $t \in [0, y_n]$ and $y_n \leq \delta_1 \sqrt{n}$, direct calculations give
\begin{align*}
\Re  \big[ (-i)^3 ( \pm \delta_1 \sqrt{n} - i t)^3  \big] 
=  3 \delta_1^2 nt - t^3
  \leq 2 \delta_1^3 n^{3/2},
\end{align*}
which implies that for $\delta_1>0$ sufficiently small,
\begin{align}\label{EW-Exp-Real01}
 \Re \bigg\{ n \sum_{k=3}^{\infty} \frac{\Lambda^{(k)} (s) }{ k! }
   \frac{ (-i)^k ( \pm \delta_1 \sqrt{n} - i t)^k }{ (\sigma_s \sqrt{n})^k }   \bigg\}
   \leq   \frac{1}{4} \delta_1^2 n. 
\end{align}
Observe that there exists a constant $c>0$ such that uniformly in $t \in [0,y_n]$ and $s \in (-\eta, \eta)$, 
\begin{align} \label{EW-Exp-Real03}
\Big|\frac{ 1 }{ z } \Big| = \Big|  \frac{1}{ \pm \delta_1 \sqrt{n} - it }  
\Big| \leq \frac{c}{ \delta_1 \sqrt{n} },  \quad
\Big|  i^3 (\pm \delta_1 \sqrt{n} -it)^3  \frac{\Lambda'''(s)}{ 6 \sigma_s^3 \sqrt{n} }  \Big| \leq  cn. 
\end{align}
Since $|\exp \{ -\frac{ib}{T} (\pm \delta_1 \sqrt{n} - it) \}|$ is bounded by some constant $c>0$,
uniformly in $t \in [0,y_n]$ and $n\geq 1$, 
from the bounds \eqref{EW-Exp-Real01} and \eqref{EW-Exp-Real03}, 
it follows that uniformly in $s \in (-\eta, \eta)$,  
\begin{align*}
| h_1( -\delta_1 \sqrt{n} - i t ) |  +   | h_1( \delta_1 \sqrt{n} - i t ) |
\leq   \frac{c}{ \delta_1 \sqrt{n} }
\Big(  e^{ \frac{\delta_1^2}{4} n }  +  cn  \Big)
\leq \frac{c_{\delta_1}}{\sqrt{n}} e^{ \frac{\delta_1^2}{4} n }.
\end{align*}
In view of \eqref{EW-Def-yn}, we have $t \leq y_n \leq -y$ and thus
$e^{ \frac{t^2}{2} + ty } \leq 1$ for any $t \in [0, y_n]$.
Note that $y_n \leq \delta_1 \sqrt{n}$ by \eqref{EW-Def-yn}. 
Consequently, we obtain the desired upper bound for $A_{11}$: 
\begin{align}\label{EW-Contr-I22-1}
\sup_{ s \in (-\eta, \eta) }
A_{11}
 \leq  c_{\delta_1}  \frac{y_n}{ \sqrt{n} }  e^{ - \frac{ \delta_1^2 }{2} n  } 
 e^{ \frac{\delta_1^2}{4} n }
 \leq c_{\delta_1} e^{ - \frac{ \delta_1^2 }{4} n  }.
\end{align}

\textit{Control of $A_{12}$.}
Using a change of variable $z = t - i y_n$ leads to
\begin{align}\label{EW-I22-2}
A_{12} & =   \sup_{y \leq 0 }  
 \bigg| e^{ \frac{1}{2} y_n^2 + y_n y}  \int_{ -\delta_1 \sqrt{n} }^{ \delta_1 \sqrt{n} }
 e^{ -\frac{t^2}{2} +  i  t (y_n +y) } h_1( t - i y_n ) \,dt  \bigg|  \notag\\
 & \leq   \sup_{y \leq 0 }  \bigg|
e^{ \frac{1}{2} y_n^2 + y_n y}  \int_{ -\delta_1 \sqrt{n} }^{ \delta_1 \sqrt{n} }
e^{ -\frac{t^2}{2} } | h_1( t - i y_n ) | \,dt  \bigg|,
\end{align}
where the function $h_1$ is defined by \eqref{EW-Def-hnz}.
To estimate the term $A_{12}$,
the main task is to give a control of $| h_1( t - i y_n ) |$.
It follows from Lemma \ref{Lem-MeanIne} that $| e^{z_1} - e^{z_2} | \leq e^{\max \{ \Re z_1, \Re z_2 \} } |z_1 - z_2 | $
and $|e^{z_2} - 1 - z_2 | \leq \frac{1}{2}|z_2|^2 e^{ |z_2| }$ for any $z_1, z_2 \in \mathbb{C}$, and hence
\begin{align}\label{EW-Ine-exp02}
| e^{z_1} - 1 - z_2 | \leq e^{\max \{ \Re z_1, \Re z_2 \} } |z_1 - z_2 | + \frac{1}{2}|z_2|^2 e^{ |z_2| }.
\end{align}
We shall make use of the inequality \eqref{EW-Ine-exp02} to derive a bound of $| h_1( t - i y_n ) |$.
Since $\frac{y_n}{\sqrt{n}} \leq \delta_1$ where $\delta_1>0$ can be sufficiently small, 
we get that, for $|t| \leq \delta_1 \sqrt{n}$ and large enough $n$, uniformly in $s \in (-\eta, \eta)$,
\begin{align}
& \Re \Big\{  \Big[ -i (t-iy_n) \Big]^3 \frac{\Lambda^{(3)}(s)}{ 6 \sigma_s^3 \sqrt{n} } \Big\}
= \frac{y_n}{\sqrt{n}}  \frac{( 3t^2 -y_n^2  ) \Lambda^{(3)}(s) }{ 6 \sigma_s^3 }
\leq \frac{1}{4} t^2,   \label{EW-Ine-Re-01} \\
& \Re \bigg\{ n \sum_{k=3}^{\infty} \frac{\Lambda^{(k)} (s) }{ k! }
   \Big[ -\frac{i (t - i y_n)}{ \sigma_s \sqrt{n} }  \Big]^k  \bigg\}
   \leq  \frac{y_n}{ \sqrt{n} }  \frac{( 6t^2 - \frac{1}{2} y_n^2 ) \Lambda^{(3)}(s) }{ 6 \sigma_s^3 }
    \leq \frac{1}{4} t^2. \label{EW-Ine-Re-02}
\end{align}
Moreover,  elementary calculations yield that there exists a constant $c>0$ such that,
for sufficiently large $n$, uniformly in $s \in (-\eta, \eta)$,  
\begin{align}\label{EW-Ine-Dif}
 \bigg|  n \sum_{k=3}^{\infty} \frac{\Lambda^{(k)} (s) }{ k! }
   \Big[ -\frac{i (t - i y_n)}{ \sigma_s \sqrt{n} }  \Big]^k
   -  [-i (t-iy_n) ]^3 \frac{\Lambda^{(3)}(s)}{ 6 \sigma_s^3 \sqrt{n} }  \bigg|  
& =  \bigg|  n \sum_{k=4}^{\infty} \frac{\Lambda^{(k)} (s) }{ k! }
   \Big[ -\frac{i (t - i y_n)}{ \sigma_s \sqrt{n} }  \Big]^k  \bigg|   \notag\\ 
& \leq  c \frac{t^4 + y_n^4}{n}.
\end{align}
It is clear that
\begin{align}\label{EW-Ine-bd01}
\sup_{s \in (-\eta, \eta)}
\Big| \Big[ -i (t-iy_n) \Big]^3 \frac{\Lambda^{(3)}(s)}{ 6 \sigma_s^3 \sqrt{n} } \Big|^2
\leq c \frac{ t^6 + y_n^6 }{n}.
\end{align}
Taking into account that both $|t|$ and $y_n$ are less than $\delta_1 \sqrt{n}$,
and the fact $\delta_1>0$ can be small enough,  it follows that 
\begin{align*}
\sup_{s \in (-\eta, \eta)}
\exp \bigg\{  \bigg| \Big[ -i (t-iy_n) \Big]^3 \frac{\Lambda^{(3)}(s)}{ 6 \sigma_s^3 \sqrt{n} }  \bigg| \bigg\}
\leq  e^{ \frac{1}{4} ( t^2 + y_n^2 ) }.
\end{align*}
Combining this with the bounds \eqref{EW-Ine-Re-01}, \eqref{EW-Ine-Re-02},
\eqref{EW-Ine-Dif} and \eqref{EW-Ine-bd01}, and using the inequality \eqref{EW-Ine-exp02}, we conclude that
\begin{align} \label{Bound-A1-Ine}
\sup_{s \in (-\eta, \eta)}
  \Big| e^{ n \sum_{k=3}^{\infty} \frac{\Lambda^{(k)} (s) }{ k! } ( -\frac{iz}{ \sigma_s \sqrt{n} } )^k }
  - 1 - (-iz)^3   &  \frac{\Lambda^{(3)}(s)}{ 6 \sigma_s^3 \sqrt{n} } \Big|  
 \leq  \  c \frac{t^4 + y_n^4}{n} e^{\frac{1}{4} t^2} + c \frac{ t^6 + y_n^6}{n} e^{ \frac{1}{4} ( t^2 + y_n^2 ) }  
   \notag\\
 & \leq c \frac{ t^4 + y_n^4 + t^6  + y_n^6}{n} e^{ \frac{1}{4} ( t^2 + y_n^2 ) }.
\end{align}
Since $|\exp\{ -\frac{ib}{T} ( t - iy_n) \}|$ is bounded by some constant, 
uniformly in $|t| \leq \delta_1 \sqrt{n}$ and $n\geq 1$, 
by \eqref{Bound-A1-Ine} and the fact $| \frac{1}{ t - iy_n } | = 1/\sqrt{t^2 + y_n^2}$, we find that
\begin{align*}
\sup_{ s \in (-\eta, \eta) }
 | h_1(t-iy_n) |
\leq c \frac{ |t|^3 + y_n^3 + |t|^5  + y_n^5 }{n} e^{ \frac{1}{4} ( t^2 + y_n^2 ) }.
\end{align*}
Therefore, noting that $y \leq -y_n$ and $0 \leq y_n \leq \delta_1 \sqrt{n}$, we obtain
\begin{align*} 
\sup_{ s \in (-\eta, \eta) } A_{12}  
& \leq  \frac{c}{n}   \sup_{y \leq 0 }  
 \bigg|  e^{ \frac{3}{4} y_n^2 + y_n y}  \int_{ -\delta_1 \sqrt{n} }^{ \delta_1 \sqrt{n} }
  e^{ -\frac{t^2}{4} } ( |t|^3 + y_n^3 + |t|^5 + y_n^5 ) \,dt  \bigg|   \notag\\
&  \leq   \frac{c}{n}
\sup_{y_n \in [0, \delta_1 \sqrt{n}]}  e^{ -\frac{1}{4} y_n^2} ( 1 + y_n^3 + y_n^5 )  \leq \frac{c}{n}.
\end{align*}
Substituting this and \eqref{EW-Contr-I22-1} into \eqref{Def-A1}, 
we conclude the proof. 
\end{proof}

\begin{lemma}\label{EW-Contr-I23}
Let $J_2(z)$ be defined by \eqref{Def-J2}, 
and $\mathcal{C}_{r}^{-}$ be defined by \eqref{Def-Cr-+}
with $r = \delta_1 \sqrt{n}$  and $\delta_1>0$ small enough. 
Then, for $T = a \sqrt{n}$ with $a>0$ large enough,  
uniformly in $x \in \mathcal{S}$, $s \in (-\eta, \eta)$
and  $\varphi \in \mathcal{B}_{\gamma}$,  
\begin{align*}
A_2 =  \sup_{y \leq 0 }  \bigg|  \int_{ \mathcal{C}_{r}^{-} }  \frac{ J_2(z) }{ z } e^{izy} e^{-i b \frac{z}{T}} \, dz \bigg|
\leq \frac{c}{n} \lVert \varphi \rVert_{\gamma}.
\end{align*}
\end{lemma}

\begin{proof}
Denote 
\begin{align*}
h_2(z) =  e^{ n \sum_{k=3}^{\infty} \frac{\Lambda^{(k)} (s) }{ k! } ( -\frac{iz}{ \sigma_s \sqrt{n} } )^k }
\bigg[  \Pi_{s, \frac{-iz}{\sigma_s \sqrt{n}}} \varphi(x)
    - \pi_s(\varphi) + iz \frac{b_{s, \varphi}(x)}{ \sigma_s \sqrt{n} } \bigg]
  \frac{ e^{-i b \frac{z}{T}} }{ z }.
\end{align*}
Using \eqref{EW-lambda-expan}, we rewrite $A_2$ as 
\begin{align*}
A_2 =  \sup_{y \leq 0 }  \bigg|  \int_{ \mathcal{C}_{r}^{-} } e^{ - \frac{z^2}{2} + izy} h_2(z) \,dz \bigg|.
\end{align*}
As in the estimation of Lemma \ref{EW-Contr-I22}, the solution of
the saddle point equation $\frac{d}{dz} ( - \frac{z^2}{2} + izy ) = 0$ is $z = iy$.
Set $y_n = \min \{-y, \delta_1 \sqrt{n}\}$. 
Since $y_n \in D_{2r}^{-}$, where $r = \delta_1 \sqrt{n}$, 
and the function $h_2$ is analytic on the domain $D_{2r}^{-}$, by Cauchy's integral theorem we obtain
\begin{align*}
A_{2} &  \leq     \sup_{y \leq 0 }  
 \bigg|  \bigg\{ \int_{ -\delta_1 \sqrt{n} }^{ -\delta_1 \sqrt{n} - i y_n }
   +  \int_{ \delta_1 \sqrt{n} - i y_n }^{ \delta_1 \sqrt{n} }  \bigg\}
     e^{ - \frac{z^2}{2} + izy} h_2(z) \,dz  \bigg|   \notag\\
 &  \quad  +   \sup_{y \leq 0 }   
 \bigg|  \int_{ -\delta_1 \sqrt{n} - i y_n }^{ \delta_1 \sqrt{n} - i y_n }
      e^{ - \frac{z^2}{2} + izy} h_2(z) \,dz   \bigg|   
 =:  A_{21} + A_{22}.
\end{align*}

\textit{Control of $A_{21}$.}
Similarly to \eqref{EW-I22-1},  we use a change of variable to get
\begin{align*}
A_{21} 
\leq  e^{ - \frac{ \delta_1^2 }{2} n  } \sup_{y \leq 0 }  
  \bigg| \int_{ 0 }^{ y_n }  e^{ \frac{t^2}{2} + ty } 
\Big[ | h_2( -\delta_1 \sqrt{n} - i t ) | + |h_2( \delta_1 \sqrt{n} - i t )| \Big] \,dt  \bigg|.
\end{align*}
Using Lemma \ref{Lem-MeanIne}, the formula \eqref{Def2-bs} and the bound \eqref{SpGapContrPi},
for any $z = \pm \delta_1 \sqrt{n} - it$ with $t \in [0, y_n]$,
we get that uniformly in $s \in (-\eta, \eta)$, $x \in \mathcal{S}$ and $\varphi \in \mathcal{B}_{\gamma}$,
\begin{align}\label{EW-Contr-I33-1}
\Big| \frac{1}{ z } \Big| 
 \Big| \Pi_{s, \frac{-iz}{\sigma_s \sqrt{n}}} \varphi(x)
    - \pi_s(\varphi) + iz \frac{b_{s, \varphi}(x)}{ \sigma_s \sqrt{n} } \Big|
\leq c \frac{|z|}{n}  \lVert \varphi \rVert_{\gamma}  
\leq \frac{c}{ \sqrt{n} } \lVert \varphi \rVert_{\gamma}.
\end{align}
Note that $| e^{-i b \frac{z}{T}} |$ is  bounded uniformly in
$z = \pm \delta_1 \sqrt{n} - it$, where $t \in [0, y_n]$.
Therefore, taking into account the bounds \eqref{EW-Exp-Real01} and \eqref{EW-Contr-I33-1}, we obtain
that uniformly in $s \in (-\eta, \eta)$ , $x \in \mathcal{S}$ and $\varphi \in \mathcal{B}_{\gamma}$,
\begin{align*}
| h_2( -\delta_1 \sqrt{n} - i t ) | + | h_2( \delta_1 \sqrt{n} - i t ) |
\leq   \frac{c}{ \sqrt{n} }  e^{ \frac{\delta_1^2}{4} n }  \lVert \varphi \rVert_{\gamma}.
\end{align*}
Since $y \leq 0$,  for any $t \in [0, y_n]$, it follows that $\frac{t^2}{2} + ty \leq 0$ 
and thus  $e^{ \frac{t^2}{2} + ty } \leq 1$.
Combining this with the above inequality yields that 
uniformly in $\varphi \in \mathcal{B}_{\gamma}$,
\begin{align}\label{EW-Contr-I23-1}
\sup_{ s \in (-\eta, \eta) } \sup_{ x \in \mathcal{S} }  
A_{21}
 \leq   c  e^{ - \frac{ \delta_1^2 }{2} n  }
 \frac{ y_n }{ \sqrt{n} }  e^{ \frac{\delta_1^2}{4} n }  \lVert \varphi \rVert_{\gamma}
 \leq c e^{ - \frac{ \delta_1^2 }{4} n  }  \lVert \varphi \rVert_{\gamma}.
\end{align}

\textit{Control of $A_{22}$.}
Similarly to \eqref{EW-I22-2}, we use  a change of variable  to get
\begin{align*}
A_{22} \leq  \sup_{y \leq 0 }  
\bigg|
 e^{ \frac{1}{2} y_n^2 + y_n y}  \int_{ -\delta_1 \sqrt{n} }^{ \delta_1 \sqrt{n} }
 e^{ -\frac{t^2}{2} } | h_2( t - i y_n ) | \,dt  \bigg|.
\end{align*}
We first estimate $| h_2( t - i y_n ) |$.
In the same way as in \eqref{EW-Contr-I33-1}, with  $z = t - i y_n$,  we obtain
that uniformly in $s \in (-\eta, \eta)$, $x \in \mathcal{S}$ and $\varphi \in \mathcal{B}_{\gamma}$,
\begin{align*}
\Big| \frac{1}{ z } \Big| \Big| \Pi_{s, \frac{-iz}{\sigma_s \sqrt{n}}} \varphi(x)
    - \pi_s(\varphi) + iz \frac{b_{s, \varphi}(x)}{ \sigma_s \sqrt{n} } \Big|
\leq c \frac{|z|}{n} \lVert \varphi \rVert_{\gamma} 
\leq c \frac{ |t| + y_n }{ n }  \lVert \varphi \rVert_{\gamma}.
\end{align*}
Combining this with the bound \eqref{EW-Ine-Re-02}, we get that uniformly in $\varphi \in \mathcal{B}_{\gamma}$,
\begin{align}\label{EW-Contr-I23-2}
\sup_{ s \in (-\eta, \eta) } \sup_{ x \in \mathcal{S} } 
A_{22}  & \leq  \frac{c}{ n }  \lVert \varphi \rVert_{\gamma}   
\sup_{y \leq 0 }  \bigg|  e^{ \frac{1}{2} y_n^2 + y_n y}  \int_{ -\delta_1 \sqrt{n} }^{ \delta_1 \sqrt{n} }
e^{ -\frac{t^2}{4} } ( |t| + y_n ) \,dt  \bigg|  \notag\\
&  \leq  \frac{c}{n}  \lVert \varphi \rVert_{\gamma} 
\sup_{y_n \in [0, \delta_1 \sqrt{n}]}  e^{ -\frac{1}{2} y_n^2} ( 1 + y_n )  
\leq \frac{c}{n}  \lVert \varphi \rVert_{\gamma}.
\end{align}
Putting together \eqref{EW-Contr-I23-1} and \eqref{EW-Contr-I23-2} completes the proof.
\end{proof}

\begin{lemma}\label{EW-Contr-I24}
Let $J_3(z)$ be defined by \eqref{Def-J3}, 
and $\mathcal{C}_{r}^{-}$ be defined by \eqref{Def-Cr-+}
with $r = \delta_1 \sqrt{n}$  and $\delta_1>0$ small enough. 
Then, for $T = a \sqrt{n}$ with $a>0$ large enough,  
uniformly in $x \in \mathcal{S}$, $s \in (-\eta, \eta)$
and  $\varphi \in \mathcal{B}_{\gamma}$, 
\begin{align*}
A_3 =  \sup_{y \leq 0 }  \bigg|  \int_{ \mathcal{C}_{r}^{-} }
\frac{ J_3(z) }{ z } e^{izy}  e^{-i b \frac{z}{T}}  \, dz \bigg|
\leq \frac{c}{n}  \lVert \varphi \rVert_{\gamma}. 
\end{align*}
\end{lemma}

\begin{proof}
We denote 
\begin{align}\label{EW-Def-hnz-03}
h_3(z) =  \frac{1}{ \sigma_s \sqrt{n} }
  \Big[ e^{ n \sum_{k=3}^{\infty} \frac{\Lambda^{(k)} (s) }{ k! } ( -\frac{iz}{ \sigma_s \sqrt{n} } )^k }
  - 1  \Big]  e^{-i b \frac{z}{T}}.
\end{align}
Using the expansion \eqref{EW-lambda-expan} and the bound \eqref{Boundbvarphi}, we have
that uniformly in $s \in (-\eta, \eta)$, $x \in \mathcal{S}$ and $\varphi \in \mathcal{B}_{\gamma}$,
\begin{align*}
A_3 \leq  c \lVert \varphi \rVert_{\gamma}  \   \sup_{y \leq 0 }  
\bigg| \int_{ \mathcal{C}_{r}^{-} }  e^{ - \frac{z^2}{2} + izy} h_3(z) \, dz  \bigg|.
\end{align*}
As in  Lemma \ref{EW-Contr-I22}, the saddle point equation $\frac{d}{dz} ( - \frac{z^2}{2} + izy ) = 0$
has the solution $z = iy$.
Set $y_n = \min \{-y, \delta_1 \sqrt{n}\}$.  It follows from Cauchy's integral theorem that
\begin{align*}
A_3  & \leq    c \lVert \varphi \rVert_{\gamma}  \  \sup_{y \leq 0 }  
\bigg|  \bigg\{ \int_{ -\delta_1 \sqrt{n} }^{ -\delta_1 \sqrt{n} - i y_n }
  +  \int_{ \delta_1 \sqrt{n} - i y_n }^{ \delta_1 \sqrt{n} }  \bigg\}
e^{ - \frac{z^2}{2} + izy} h_3(z) \, dz  \bigg|   \notag\\
 &  \quad  +  c  \lVert \varphi \rVert_{\gamma} \  \sup_{y \leq 0 }   
 \bigg|  \int_{ -\delta_1 \sqrt{n} - i y_n }^{ \delta_1 \sqrt{n} - i y_n }
e^{ - \frac{z^2}{2} + izy} h_3(z) \, dz   \bigg|  
 =:    A_{31} +  A_{32}.
\end{align*}

\textit{Control of $A_{31}$.}
Similarly to \eqref{EW-I22-1},  we use a change of variable to get
\begin{align*}
A_{31} \leq  c \lVert \varphi \rVert_{\gamma}   e^{ - \frac{ \delta_1^2 }{2} n  } \sup_{y \leq 0 }  
\bigg| \int_{ 0 }^{ y_n } e^{ \frac{t^2}{2} + ty } 
\Big[ \Big| h_3( -\delta_1 \sqrt{n} - i t ) \Big| + \Big| h_3( \delta_1 \sqrt{n} - i t ) \Big| \Big] \,dt  \bigg|.
\end{align*}
Using \eqref{EW-Exp-Real01}, we deduce that 
uniformly in  $s \in (-\eta, \eta)$ and $x \in \mathcal{S}$,
\begin{align*}
| h_3( -\delta_1 \sqrt{n} - i t ) | + | h_3( \delta_1 \sqrt{n} - i t ) |
\leq   \frac{c}{\sqrt{n}}  \Big(  e^{\frac{\delta_1^2}{4} n} + 1 \Big)
\leq  \frac{c}{\sqrt{n}} e^{\frac{\delta_1^2}{4} n}.
\end{align*}
Since $\frac{t^2}{2} + ty \leq 0$ for any $t \in [0, y_n]$ and $y \leq 0$,
we have $e^{ \frac{t^2}{2} + ty } \leq 1$. This, together with the above inequality, implies
that uniformly in $\varphi \in \mathcal{B}_{\gamma}$, 
\begin{align}\label{EW-Contr-I24-1}
\sup_{ s \in (-\eta, \eta) } \sup_{ x \in \mathcal{S} }  
A_{31}
 \leq   c \frac{y_n}{ \sqrt{n} } e^{ - \frac{ \delta_1^2 }{4} n  }  \lVert \varphi \rVert_{\gamma}
 \leq c e^{ - \frac{ \delta_1^2 }{4} n}  \lVert \varphi \rVert_{\gamma}.
\end{align}

\textit{Control of $A_{32}$.}
Similarly to \eqref{EW-I22-2}, one has
\begin{align*}
A_{32} \leq  c  \lVert \varphi \rVert_{\gamma}  \   \sup_{y \leq 0 }  
\bigg|  e^{ \frac{1}{2} y_n^2 + y_n y}  \int_{ -\delta_1 \sqrt{n} }^{ \delta_1 \sqrt{n} }
e^{ -\frac{t^2}{2} } | h_3( t - i y_n ) | \,dt  \bigg|.
\end{align*}
We first give a control of $| h_3( t - i y_n ) |$.
By Lemma \ref{Lem-MeanIne}, it holds that
$| e^{z} - 1 | \leq e^{\max \{\Re z, 0\} } |z| $
for any $z \in \mathbb{C}$. Using this inequality and taking into account the bound \eqref{EW-Ine-Re-02}, we obtain
\begin{align*}
\sup_{ s \in (-\eta, \eta) } 
\Big| e^{ n \sum_{k=3}^{\infty} \frac{\Lambda^{(k)} (s) }{ k! } ( -\frac{iz}{ \sigma_s \sqrt{n} } )^k }  - 1  \Big| 
  \leq  c e^{\frac{1}{4} t^2} \frac{ |t|^3 + y_n^3 }{ \sqrt{n} },
\end{align*}
and hence
\begin{align*}
\sup_{ s \in (-\eta, \eta) } \sup_{ x \in \mathcal{S} } 
| h_3( t - i y_n ) |  \leq  c e^{\frac{1}{4} t^2} \frac{ |t|^3 + y_n^3 }{ n }.
\end{align*}
It follows that uniformly in  $s \in (-\eta, \eta)$, $x \in \mathcal{S}$ and $\varphi \in \mathcal{B}_{\gamma}$,
\begin{align}\label{EW-Contr-I24-2}
A_{32}
 \leq   \frac{c}{n}  \lVert \varphi \rVert_{\gamma} \   \sup_{y \leq 0 }  
 \bigg|  e^{ -\frac{1}{2} y_n^2 }  \int_{ -\delta_1 \sqrt{n} }^{ \delta_1 \sqrt{n} }
     e^{ -\frac{t^2}{4} } ( |t|^3 + y_n^3 ) \,dt  \bigg|  
     \leq \frac{c}{n}  \lVert \varphi \rVert_{\gamma}.
\end{align}
Putting together \eqref{EW-Contr-I24-1} and \eqref{EW-Contr-I24-2}, we conclude the proof.
\end{proof}

\begin{lemma}\label{EW-Contr-I25}
Let $J_4(z)$ be defined by \eqref{Def-J4}, 
and $\mathcal{C}_{r}^{-}$ be defined by \eqref{Def-Cr-+}
with $r = \delta_1 \sqrt{n}$  and $\delta_1>0$ small enough. 
Then, for $T = a \sqrt{n}$ with $a>0$ large enough,  
uniformly in $x \in \mathcal{S}$,  $s \in (-\eta, \eta)$ 
and  $\varphi \in \mathcal{B}_{\gamma}$, 
\begin{align*}
A_4 =  \sup_{y \leq 0 }  \bigg|  \int_{ \mathcal{C}_{r}^{-} }
\frac{ J_4(z) }{ z } e^{izy} e^{-i b \frac{z}{T}} \, dz \bigg|
\leq  c e^{-cn}  \lVert \varphi \rVert_{\gamma}. 
\end{align*}
\end{lemma}

\begin{proof}
Since $\Im z \leq 0$ on $\mathcal{C}_{r}^{-}$ and $y \leq 0$,
we have $|e^{izy}| \leq 1$. 
Using again the fact that $\Im z \leq 0$, 
we get that $|e^{-i b \frac{z}{T}}|$ is uniformly bounded on $\mathcal{C}_{r}^{-}$.
From the bound \eqref{SpGapContrN} and the fact that $\delta_1>0$ can be sufficiently small, 
we deduce that  $|J_4(z)| \leq  c e^{-cn} \lVert \varphi \rVert_{\gamma}$,
uniformly in  $s \in (-\eta, \eta)$, $x \in \mathcal{S}$ and $\varphi \in \mathcal{B}_{\gamma}$.  
Therefore, noting that $|\frac{1}{z}| = (\delta_1 \sqrt{n})^{-1}$
and that the length of $\mathcal{C}_{r}^{-}$ is $\pi \delta_1 \sqrt{n}$, 
the desired result follows.
\end{proof}

\begin{proof}[\textit{End of the proof of Theorem \ref{Thm-Edge-Expan-bb}}]
Combining 
Lemmata \ref{EW-Contr-I22}-\ref{EW-Contr-I25}, we obtain that
$I_{31} \leq  \frac{c}{n}  \lVert \varphi \rVert_{\gamma},$
uniformly in  $s \in (-\eta, \eta)$,  $x \in \mathcal{S}$ and $\varphi \in \mathcal{B}_{\gamma}.$

Now we give a control of the term $I_{32}$ defined in \eqref{Pf-Thm2-DecomI3}.  
Note that $y>0$ in $I_{32}$ and the integral in $I_{32}$ is taken over the semicircle $\mathcal{C}_{r}^{+}$,
which lies in the upper part of the complex plane. 
In this case we have the saddle point equation $\frac{d}{dz} ( - \frac{z^2}{2} + izy ) = 0$  whose solution 
$z = iy$ also lies in the upper part of the complex plane. 
Similarly to \eqref{EW-Def-yn}, we choose a suitable point $y_n = \min \{y, \delta_1 \sqrt{n}\}$. 
Proceeding in the same way as for bounding $I_{31}$
we obtain that $I_{32} \leq \frac{c}{n} \lVert \varphi \rVert_{\gamma}$, 
uniformly in  $s \in (-\eta, \eta)$,  $x \in \mathcal{S}$ and $\varphi \in \mathcal{B}_{\gamma}$.

Let us now bound the terms $I_{33}$ and $I_{34}$ defined in \eqref{Pf-Thm2-DecomI3}.
Since the function $z \mapsto e^{i b \frac{z}{T}}$
is analytic on $\mathcal{C}_{r}^{-}$ and $\mathcal{C}_{r}^{+}$,
the estimates of  $I_{33}$ and $I_{34}$ are similar to those of $I_{31}$ and $I_{32}$, respectively.
From these bounds, we conclude that there exists a constant $c>0$ such that
uniformly in  $s \in (-\eta, \eta)$,  $x \in \mathcal{S}$ and $\varphi \in \mathcal{B}_{\gamma}$, 
\begin{align}\label{Ch4_I3vv}
I_3 \leq  \frac{c}{n} \lVert \varphi \rVert_{\gamma}. 
\end{align}

It remains to estimate $I_4$ defined in \eqref{Pf-Thm2-DecomI3}. 
We can decompose the difference $|f(t) - h(t)|$ in the same way as we did in \eqref{Decom-f-h}
(with real-valued $t = z$). 
Then proceeding in a similar way as in the estimation of 
$I_{31}$, $I_{32}$,  $I_{33}$ and $I_{34}$, 
one can verify that there exists a constant $c>0$ such that
uniformly in  $s \in (-\eta, \eta)$,  $x \in \mathcal{S}$ and $\varphi \in \mathcal{B}_{\gamma}$, 
\begin{align}\label{Ch4_I4vv}
I_4 \leq  \frac{c}{n}  \lVert \varphi \rVert_{\gamma}. 
\end{align}

Combining \eqref{Ch4_I3vv}, \eqref{Ch4_I3vv} and the bounds for $I_1$ and $I_2$ 
in \eqref{Edge-Contr-I3} and \eqref{Edge-Contr-I1}, 
and using the fact that $\varepsilon >0$ can be arbitrary small, we obtain \eqref{EW-Expan-Toprove}, 
which finishes the proof of Theorem \ref{Thm-Edge-Expan-bb}.
\end{proof}

\subsection{Proof of Theorem \ref{Thm-BerryEsseen-Norm}}

Since the proof of Theorem \ref{Thm-BerryEsseen-Norm} 
is quite similar to that of Theorem \ref{Thm-Edge-Expan-bb},
we only sketch the main differences. 
Denote
\begin{align*}
F(y) =  
\mathbb{E}_{\mathbb{Q}_s^x}
   \Big[  \varphi(X_n^x) \mathds{1}_{ \big\{ \frac{\sigma(G_n, x) - n \Lambda'(s) }{\sigma_s \sqrt{n}} \leq y \big\} } \Big],  
   \quad   H(y) =    \mathbb{E}_{\mathbb{Q}_s^x} [ \varphi(X_n^x) ] \   \Phi(y),  \quad  y\in \mathbb{R}.   
\end{align*}
By the definition of the operator $R_{s,z}$ in \eqref{Def-Rn-sw000}, 
direct calculations lead to 
\begin{align*}
f(t)   =  \int_{\mathbb{R}} e^{-ity} \, dF(y) = R^n_{s, \frac{-it}{ \sigma_s \sqrt{n} } } \varphi(x), 
\quad   
h(t)   =  \int_{\mathbb{R}} e^{-ity} \, dH(y) = e^{-\frac{t^2}{2}}  R_{s,0}^n \varphi(x),
\quad  t \in \mathbb{R}.
\end{align*}
One can verify that the functions $F, H$ and their corresponding Fourier-Stieljes transforms 
$f, h$ satisfy all the conditions
stated in Proposition \ref{Prop-BerryEsseen}.
Instead of using Proposition \ref{Prop-BerryEsseen} with
$r < T$ in the proof of Theorem \ref{Thm-Edge-Expan-bb}, 
we apply Proposition \ref{Prop-BerryEsseen}
with $r = T = \delta_1 \sqrt{n}$, where $\delta_1>0$ is a sufficiently small constant.  
Then we obtain a similar inequality as \eqref{BerryEsseen001} but with the term $I_2 =0$.
Since the non-arithmeticity condition \ref{CondiNonarith} is only used in the bound of the term $I_2$,
following the  proof of Theorem \ref{Thm-Edge-Expan-bb} 
we show that 
under the conditions of Theorem \ref{Thm-BerryEsseen-Norm}, 
the terms $I_1$ and $I_3$ defined in \eqref{Pf-Thm2-DecomI3} are bounded by 
$c \lVert \varphi \rVert_{\gamma} / \sqrt{n}$, 
uniformly in  $s \in (-\eta, \eta)$,  $x \in \mathcal{S}$
and $\varphi \in \mathcal{B}_{\gamma}$. 
We omit the details of the rest of the proof.

\section{Proof of moderate deviation expansions}

In this section we prove Theorem \ref{MainThmNormTarget}.
The proof 
is based on the Berry--Esseen bound in Theorem \ref{Thm-BerryEsseen-Norm}
and follows the standard techniques in Petrov \cite{Pet75}, and therefore some details will be left to the reader.  
 
We start with the following lemma whose proof uses the analyticity 
of the eigenfunction $r_s$ and the linear functional $\nu_s$, see Proposition \ref{transfer operator}:
\begin{lemma}\label{lem:expconv}
Assume either conditions \ref{CondiMoment} and \ref{CondiIP} for invertible matrices,
or conditions \ref{CondiMoment} and \ref{CondiAP}  for positive matrices.
Then, there exists $\eta>0$ such that uniformly in $s \in (-\eta, \eta)$ and $\varphi \in \mathcal{B}_{\gamma}$,
\begin{align*}
\lVert r_s - \mathbf{1} \rVert_{\infty} \leq C | s | \quad \mbox{and} 
\quad | \nu_s( \varphi ) - \nu(\varphi)| \leq C | s |  \lVert \varphi \rVert_{\gamma}.
\end{align*}
\end{lemma}
\begin{proof}

According to Proposition \ref{transfer operator}, we have $r_0 = \mathbf{1}$, $\nu_0 = \nu$. In addition,
the mappings
$s \mapsto r_s$ and $s \mapsto \nu_s$ are analytic on $(-\eta, \eta)$.
The assertions follow using Taylor's formula.  
\end{proof}

\begin{proof}[Proof of  Theorem \ref{MainThmNormTarget}]
When $y\in [0,1]$,
Theorem \ref{MainThmNormTarget} is a direct consequence of Theorem \ref{Thm-BerryEsseen-Norm},
so it remains to prove Theorem \ref{MainThmNormTarget} in the case when $y>1$ with $y=o(\sqrt{n})$.
We proceed to prove the first assertion in Theorem \ref{MainThmNormTarget}.
Applying the change of measure formula \eqref{basic equ1}, we have  
\begin{align}
I : &  =    \mathbb{E} \Big[
    \varphi(X_n^x)  \mathds{1}_{ \{ \sigma(G_n, x) \geq n\Lambda'(0) + \sqrt{n}\sigma_0y \}  }
     \Big]  \label{def_of_I001} \\
&  =    r_s(x) \kappa^{n}(s) \mathbb{E}_{ \mathbb{Q}^{x}_{s} }
  \Big[ (\varphi r_s^{-1})(X_n^x) e^{-s \sigma(G_n, x) } 
  \mathds{1}_{\{ \sigma(G_n, x) \geq n\Lambda'(0) + \sqrt{n}\sigma_0y\}} \Big]. \notag 
\end{align}
Under the assumptions of Theorem \ref{MainThmNormTarget}, by Proposition \ref{Prop-lambTaylor},  
$\sigma_s^2 = \Lambda''(s)>0$ for any $s \in (-\eta, \eta)$ with $\eta>0$ small enough.
We denote $W_n^x = \frac{ \sigma(G_n, x) - n \Lambda'(s) }{ \sigma_s\sqrt{n} }$.
Recalling that $\Lambda = \log \kappa$, we rewrite  \eqref{def_of_I001} as follows:
\begin{align} \label{change of measure Wn1}
I  =     r_s(x) e^{-n [ s \Lambda'(s)-\Lambda(s) ] }  
\mathbb{E}_{\mathbb{Q}_{s}^{x}}
    \Big[  (\varphi r_s^{-1}) (X_n^x)
    e^{ -s \sigma_s \sqrt{n} W_{n}^x }
        \mathds{1}_{ \big\{ W_{n}^x \geq \frac{\sqrt{n}[\Lambda'(0)-\Lambda'(s)]}{\sigma_s} + \frac{ \sigma_0y }{ \sigma_s}
        \big\} }
      \Big].
\end{align}
By Proposition \ref{transfer operator},
the function $\Lambda$ is analytic and hence 
for $s\in (-\eta, \eta)$,
$\Lambda(s)=\sum_{k=1}^{\infty}\frac{\gamma_k}{k!}s^k,$ where $\gamma_k = \Lambda^{(k)}(0).$
For any $y>1$ with $y = o(\sqrt{n})$,  consider the equation
\begin{align} \label{equation s and y 1}
\sqrt{n} [ \Lambda'(s) - \Lambda'(0) ] = \sigma_0y.
\end{align}
Choosing the unique real root $s$ of \eqref{equation s and y 1}, it follows from Petrov \cite{Pet75} that
\begin{align}\label{zetaLamb}
s\Lambda'(s)-\Lambda(s)
= \frac{y^2}{2n} - \frac{y^3}{ n^{3/2} } \zeta( \frac{y}{\sqrt{n} } ),
\end{align}
where $\zeta$ is the Cram\'{e}r series defined by \eqref{Def-CramSeri}.
Substituting \eqref{equation s and y 1} into \eqref{change of measure Wn1},
and using \eqref{zetaLamb}, we get
\begin{align}\label{RewriI001}
I =   r_s(x) e^{ -\frac{y^2}{2} + \frac{y^3}{ \sqrt{n} } \zeta(\frac{y}{\sqrt{n}}) }  \mathbb{E}_{\mathbb{Q}_{s}^{x}}
    \Big[  (\varphi r_s^{-1}) (X_n^x)
    e^{ -s \sigma_s \sqrt{n} W_{n}^x }
        \mathds{1}_{ \{ W_{n}^x \geq 0 \} }
      \Big].
\end{align}
For brevity, denote  
$F(u) = \mathbb{E}_{\mathbb{Q}_s^x}
   \big[ (\varphi r_s^{-1}) (X_n^x) \mathds{1}_{\{ W_n^x \leq u \} } \big]$, $u \in \mathbb{R}.$ 
In view of \eqref{RewriI001}, using Fubini's theorem and integration by parts,
we deduce that
\begin{align} \label{change measure 1}
I & =   r_s(x) e^{ -\frac{y^2}{2} + \frac{y^3}{ \sqrt{n} } \zeta(\frac{y}{\sqrt{n}}) }  \mathbb{E}_{\mathbb{Q}_{s}^{x}}
    \bigg[  (\varphi r_s^{-1}) (X_n^x)
  \int_0^{\infty} \mathds{1}_{ \{ 0\leq W_n^x \leq u  \} } s \sigma_s \sqrt{n}  \, e^{ -s \sigma_s \sqrt{n} u}  \, du
      \bigg]   \notag\\
& =  r_s(x) e^{ -\frac{y^2}{2} + \frac{y^3}{ \sqrt{n} } \zeta(\frac{y}{\sqrt{n}}) }
           \int_{0}^{\infty} e^{ -s\sqrt{n} \sigma_su } \, dF(u).
\end{align}
Let $l(u) = F(u) - \pi_s(\varphi r_s^{-1} ) \Phi(u)$, $u \in \mathbb{R}$. It follows that
\begin{align}\label{compare I1 and normal 1}
& \int_{0}^{\infty} e^{-s\sqrt{n} \sigma_s u } \, d F(u)  
=  I_1  + \frac{ \pi_s( \varphi r_s^{-1} ) }{\sqrt{2\pi}} I_2,\\
& I_1 =  \int_{0}^{\infty} e^{-s\sqrt{n} \sigma_su} \, dl(u),  \quad 
  I_2 = \int_{0}^{\infty} e^{-s\sqrt{n} \sigma_su - \frac{u^2}{2}} \, du. 
\end{align}

\textit{Estimate of $I_1$.}
Integrating by parts, using the fact that $r_s \in \mathcal{B}_{\gamma}$
and the Berry--Esseen bound in Theorem \ref{Thm-BerryEsseen-Norm} 
implies that uniformly in $s \in [0, \eta)$, $x \in \mathcal{S}$ and $\varphi \in \mathcal{B}_{\gamma}$,  
\begin{align} \label{error I1 and normal}
|I_1|  \leq |l(0)| + s\sqrt{n} \sigma_s \int_{0}^{\infty} e^{-s\sqrt{n}\sigma_s u} |l(u)| \, du
 \leq   \frac{C}{\sqrt{n}} \lVert \varphi \rVert_{\gamma}.
\end{align}

\textit{Estimate of $I_2$.}
Since the function $\Lambda$ is analytic on $(-\eta, \eta)$ and $\sigma_s^2 = \Lambda''(s) >0$,  
by Taylor's formula,  
we have $\Lambda'(s)-\Lambda'(0)=s\sigma_0^2\big[1+O(s)\big]$
and $\sigma_s^2 = \sigma_0^2\big[1+O(s)\big]$.
Thus, using standard techniques from Petrov \cite{Pet75}, one has 
\begin{align} \label{defini I1I2}
I_2 = I_3 +  O \Big( \frac{ 1 }{ \sqrt{n} }  \Big),
\quad  \mbox{where}  \   I_3 =  
  \int_{0}^{\infty} e^{-\frac{ \sqrt{n} [\Lambda'(s)-\Lambda'(0)]  }{\sigma_0} u - \frac{u^2}{2}} \, du.
\end{align}
Since $\sigma_s$ is strictly positive and bounded uniformly in $s \in (0, \eta)$, 
using \eqref{equation s and y 1}  and the fact that $y>1$, for sufficiently large $n$, we get that
$s \sqrt{n} \, \sigma_s  
 \geq \frac{y}{ 2\sigma_0 } \sigma_s \geq  c_1 >0$. 
This implies that $C_1  \leq s \sqrt{n}I_2 \leq  C_2$ for large enough $n$, where 
$C_1<C_2$ are two positive constants independent of $n$ and $s$. 
Combining this two-sided bound with \eqref{compare I1 and normal 1}, 
\eqref{error I1 and normal} and \eqref{defini I1I2},  we obtain 
\begin{align} \label{estimaite I1 1}
\int_{0}^{\infty} e^{-s\sqrt{n} \sigma_s u } \, d F(u)
=   I_3  \bigg[ \frac{ \pi_s( \varphi r_s^{-1} ) }{\sqrt{2\pi}} +  \lVert \varphi \rVert_{\gamma} O(s) \bigg].
\end{align}
Substituting \eqref{equation s and y 1} into \eqref{defini I1I2}, we get
\begin{align*} 
  \int_{0}^{\infty} e^{-s\sqrt{n} \sigma_s u } \, d F(u)
=  e^{\frac{y^2}{2} }  \int_{y}^{\infty} e^{-\frac{1}{2}u^2} \, du 
   \bigg[ \frac{ \pi_s( \varphi r_s^{-1} ) }{\sqrt{2\pi}} + \lVert \varphi \rVert_{\gamma}  O(s) \bigg].
\end{align*}
Together with \eqref{change measure 1}, this implies
\begin{align}\label{PrfCram001}
I = r_s(x) e^{\frac{y^3}{\sqrt{n}} \zeta( \frac{y}{\sqrt{n}} )} \big[ 1 - \Phi(y) \big]
   \Big[ \pi_s( \varphi r_s^{-1} ) + \lVert \varphi \rVert_{\gamma}  O(s) \Big],
\end{align}
where $\pi_s( \varphi r_s^{-1} ) = \frac{ \nu_s(\varphi) }{ \nu_s(r_s) }$.
By Lemma \ref{lem:expconv}, we have
$\lVert r_s - \mathbf{1} \rVert_{\infty} \leq Cs$ and
$| \pi_s( \varphi r_s^{-1} ) - \nu(\varphi)| \leq C s \lVert \varphi \rVert_{\gamma}$,
uniformly in $s \in [0, \eta)$ and $\varphi \in \mathcal{B}_{\gamma}$. 
Since $s=O(\frac{y}{\sqrt{n}})$,
this concludes the proof of 
the first assertion of Theorem \ref{MainThmNormTarget}.

The proof of  the second assertion of Theorem \ref{MainThmNormTarget} can be carried out in a similar way.
Specifically, instead of using \eqref{equation s and y 1}, we consider the equation
$\sqrt{n}[\Lambda'(s)-\Lambda'(0)] = -\sigma_0y, $ 
where $y>1$ and $s \in (-\eta,0]$.
We then apply the spectral gap properties of operators $P_s, Q_s, R_{s,z}$ 
(see Section \ref{sec:spec gap norm})
for negative valued $s$ to deduce the second assertion 
by following the proof of the first one.  We omit the details.
\end{proof}


\section{Proof of the local limit theorems} 

The goal of this section is to establish the local limit theorems with moderate deviations, 
namely Theorems \ref{ThmLocal01} and \ref{Thm_LLT_Norm_0a}. 

\subsection{Proof of Theorem \ref{ThmLocal01}} 

We first establish an asymptotic expansion which will be used to prove Theorem \ref{ThmLocal01}. 
Assume that $\psi: \mathbb R \mapsto \mathbb C$
is a continuous function with compact support in $\mathbb{R}$, 
which is differentiable in a small neighborhood of $0$ on the real line.

\begin{proposition} \label{Prop_R_st_limit}
Assume either conditions \ref{CondiMoment} and \ref{CondiIP} for invertible matrices,
or conditions \ref{CondiMoment}, \ref{CondiAP} and \ref{Condi-Variance} for positive matrices.
Then, there exist constants $\eta, \delta, c, C >0$ such that for all $s \in (-\eta, \eta)$, 
$x \in \mathcal{S}$, $|l|\leq \frac{1}{\sqrt{n}}$, $\varphi \in \mathcal{B}_{\gamma}$ and $n \geq 1$, 
\begin{align} \label{R_st_limit}
&  \bigg| \sigma_s \sqrt{n}  \,  e^{ \frac{n l^2}{2 \sigma_s^2} }
\int_{\mathbb R} e^{-it l n} R^{n}_{s, it}(\varphi)(x) \psi (t) \,dt
  - \sqrt{2\pi} \pi_s(\varphi) \psi(0) \bigg|   \notag\\
& \leq  \frac{ C }{ \sqrt{n} }  \lVert \varphi \rVert_{\gamma}
  + \frac{C}{n}  \lVert \varphi \rVert_{\gamma}  \sup_{|t| \leq \delta} \big( |\psi(t)| + |\psi'(t)| \big)
  + Ce^{-cn}  \lVert \varphi \rVert_{\gamma}  \int_{\mathbb R} |\psi(t)| \,dt. 
\end{align}
\end{proposition}

\begin{proof}
For brevity, denote $c_s(\psi)= \frac{\sqrt{2\pi}}{\sigma_s} \pi_{s}(\varphi) \psi(0).$ 
Taking a small constant $\delta >0$
and using the spectral gap decomposition \eqref{perturb001} with $z = it$, 
we have
\begin{align} \label{Thm1 integral1 J}
&  \Big| \sqrt{n}  \   e^{ \frac{n l^2}{2 \sigma_s^2} }
  \int_{\mathbb R} e^{-it l n} R^{n}_{s,it}(\varphi)(x) \psi (t) \,dt  - c_s(\psi)  \Big|  \notag\\ 
& \leq  \Big| \sqrt{n} \   e^{ \frac{n l^2}{2 \sigma_s^2} }
\int_{|t|\geq\delta} e^{-itln}R^{n}_{s,it}(\varphi)(x) \psi(t) \,dt  \Big|  
  + \Big| \sqrt{n}  \   e^{ \frac{n l^2}{2 \sigma_s^2} } \int_{|t|< \delta }  e^{-i t l n}
N^{n}_{s,it}(\varphi)(x) \psi(t) \,dt \Big|    \notag\\
& \quad + \Big| \sqrt{n}  \  e^{ \frac{n l^2}{2 \sigma_s^2} }  \int_{|t|<\delta} e^{-i t l n}
\lambda^{n}_{s,it} \Pi_{s,it}(\varphi)(x) \psi(t) \,dt - c_s(\psi) \Big|   \notag\\
& =:  J_1 + J_2 + J_3.  
\end{align}
For $J_1$, since the function $\psi$ is bounded and compactly supported on $\mathbb R$, 
taking into account Proposition \ref{Prop-UnifR} and the fact $|e^{-it l n}| = 1$, 
we get
\begin{align} \label{Thm1 integral1 J1}
\sup_{s \in (-\eta, \eta)} \sup_{x\in \mathcal{S}} \sup_{|l|\leq \frac{1}{\sqrt{n}} }
J_1 \leq C_{\delta}  e^{- c_{\delta} n}  \lVert \varphi \rVert_{\gamma}
 \int_{|t|\geq\delta} |\psi(t)| \,dt. 
\end{align}
For $J_2$, by \eqref{SpGapContrN} there exist constants $c_{\delta} >0$ and $a \in (0,1)$ such that 
\begin{align*}
\sup_{s \in (-\eta, \eta)}
\sup_{x\in \mathcal{S}} \sup_{|t| < \delta}
|N^{n}_{s,it}(\varphi)(x)|
\leq  \sup_{s \in (-\eta, \eta)} \sup_{|t| < \delta}
\lVert N^{n}_{s,it} \rVert_{\mathcal B_{\gamma} \to \mathcal B_{\gamma}}  \lVert \varphi \rVert_{\gamma}
\leq  c_{\delta} a^n \lVert \varphi \rVert_{\gamma}.  
\end{align*}
This implies that
uniformly in $s \in (-\eta, \eta)$, $|l| \leq \frac{1}{\sqrt{n}}$, 
$x \in \mathcal{S}$ and $\varphi \in \mathcal{B}_{\gamma}$, 
\begin{align}  \label{SaddleIntegral 2}
J_2 \leq C_{\delta}  e^{- c_{\delta} n}  \lVert \varphi \rVert_{\gamma}  \int_{|t| < \delta} |\psi(t)| \,dt. 
\end{align}
For $J_3$, we make a change of variable $t = t'/\sqrt{n}$ to get
\begin{align}\label{Ch2PropRn1aya}
J_3  & =  \Big|  e^{ \frac{n l^2}{2 \sigma_s^2} }  \int_{-\delta \sqrt{n}}^{\delta \sqrt{n}}  e^{-i t l \sqrt{n}}
\lambda^{n}_{s, \frac{ it }{ \sqrt{n} } } 
 \Pi_{s, \frac{it}{\sqrt{n}} }(\varphi)(x) \psi \big( \frac{t}{\sqrt{n}} \big) \,dt - c_s(\psi) \Big|  \notag\\
& \leq \Big|  e^{ \frac{n l^2}{2 \sigma_s^2} }  \int_{-\delta \sqrt{n}}^{\delta \sqrt{n}}  e^{-i t l \sqrt{n}}
\lambda^{n}_{s, \frac{ it }{ \sqrt{n} } }
\Big[ \Pi_{s, \frac{it}{\sqrt{n}} }(\varphi)(x) \psi \big( \frac{t}{\sqrt{n}} \big) 
     - \pi_s(\varphi) \psi(0) \Big]\,dt
 \Big|   \notag\\
& \quad  +  \Big| \pi_s(\varphi) \psi(0)  e^{ \frac{n l^2}{2 \sigma_s^2} }  
   \int_{-\delta \sqrt{n}}^{\delta \sqrt{n}}  e^{-i t l \sqrt{n}}
  \lambda^{n}_{s,\frac{it}{\sqrt{n}}}\,dt - c_s(\psi) \Big|   
 =:  J_{31} + J_{32}. 
\end{align}
Using the formula \eqref{relationlamkappa001} and the fact that the function $\Lambda$
is analytic in a small neighborhood of $0$ of the complex plane,  
we can check that there exists a constant $C >0$ such that for all $s \in (-\eta, \eta)$,
$t \in [-\delta \sqrt{n}, \delta \sqrt{n}]$ and $n \geq 1$,  
\begin{align}\label{LLT_Esti_Add_1}
\Big| \lambda^{n}_{ s, \frac{it}{\sqrt{n}} } - e^{ - \frac{ \sigma_s^2 t^2 }{2} } \Big|
\leq  \frac{C}{ \sqrt{n} } e^{ -\frac{ \sigma_s^2 t^2 }{4} }.  
\end{align}
By \eqref{SpGapContrPi} and the fact that $\Pi_{s,0} (\varphi) (x) = \pi_s(\varphi)$, 
it follows that uniformly in $s \in (-\eta, \eta)$, 
$t \in [-\delta \sqrt{n}, \delta \sqrt{n}]$ and $x\in \mathcal{S}$, 
\begin{align*} 
\Big| \Pi_{s,\frac{it}{\sqrt{n}}} (\varphi)(x) - \pi_s (\varphi) \Big|
 \leq  \Big\lVert \Pi_{s, \frac{it}{\sqrt{n}}} - \Pi_{s,0} \Big\rVert_{\mathcal B_{\gamma} \to \mathcal B_{\gamma}}
  \lVert \varphi \rVert_{\gamma} 
 \leq  c \frac{|t|}{\sqrt{n}} \lVert \varphi \rVert_{\gamma}. 
\end{align*}
Since the function $\psi$ is differentiable in a small neighborhood of $0$, 
we obtain that there exists a constant $C >0$ such that 
for all $s \in (-\eta, \eta)$, $x \in \mathcal{S}$
and $t \in [-\delta \sqrt{n}, \delta \sqrt{n}]$,
\begin{align*}
&  \Big| \Pi_{s, \frac{it}{\sqrt{n}} }(\varphi)(x) \psi \big( \frac{t}{\sqrt{n}} \big) 
     - \pi_s (\varphi) \psi(0)  \Big|   \notag\\
& \leq  \Big| \Pi_{s, \frac{it}{\sqrt{n}} }(\varphi)(x) \psi \big( \frac{t}{\sqrt{n}} \big) 
     - \pi_s (\varphi) \psi \big( \frac{t}{\sqrt{n}} \big)  \Big| 
   +  \Big| \Pi_{s, 0}(\varphi)(x) \psi \big( \frac{t}{\sqrt{n}} \big) 
     - \pi_s (\varphi) \psi(0)  \Big|   \notag\\
& \leq C \frac{|t|}{\sqrt{n}}  \lVert \varphi \rVert_{\gamma}  \sup_{|t| \leq \delta} |\psi(t)|
   + C  \frac{|t|}{\sqrt{n}}  \lVert \varphi \rVert_{\gamma}  \sup_{|t| \leq \delta} |\psi'(t)|. 
\end{align*}
Combining this with \eqref{LLT_Esti_Add_1}, 
we get the desired bound for $J_{31}$: 
there exists a constant $C >0$ such that, for all $n \geq 1$, 
$|l| \leq \frac{1}{\sqrt{n}}$, $s \in (-\eta, \eta)$, $x \in \mathcal{S}$ and $\varphi \in \mathcal{B}_{\gamma}$, 
\begin{align}  \label{Ch2LDfghj}
J_{31}   \leq  \frac{C}{\sqrt{n}}  \lVert \varphi \rVert_{\gamma}  
      + \frac{C}{n}  \lVert \varphi \rVert_{\gamma}  \sup_{|t| \leq \delta} \big( |\psi(t)| + |\psi'(t)| \big).  
\end{align} 
To estimate $J_{32}$ in \eqref{Ch2PropRn1aya}, we first notice that
\begin{align} \label{CH2LD_ius}
J_{32} 
& \leq   \bigg| \pi_s(\varphi) \psi(0) e^{\frac{n l^2}{2\sigma_s^2}} 
\int_{-\delta \sqrt{n}}^{\delta \sqrt{n}} e^{-i t l \sqrt{n}} 
\Big( \lambda^{n}_{s,\frac{it}{\sqrt{n}}} - e^{-\frac{\sigma_s^{2}t^{2}}{2}} \Big) \,dt  \bigg| \notag\\
& \quad  +   \bigg| \pi_s(\varphi) \psi(0) e^{\frac{n l^2}{2\sigma_s^2}}  \int_{|t|\geq \delta\sqrt{n}} 
 e^{-i t l \sqrt{n}} e^{-\frac{\sigma_s^{2}t^{2}}{2}} \,dt \bigg|  
 =:  J_{321} + J_{322}. 
\end{align}
For $J_{321}$, from \eqref{LLT_Esti_Add_1} it follows that $J_{321} \leq \frac{C}{\sqrt{n}} \lVert \varphi \rVert_{\gamma}$. 
For $J_{322}$, using the basic inequality
$\int_y^\infty e^{-\frac{t^2}{2}} \,dt \leq \frac{1}{y} e^{-\frac{y^2}{2}}$ for $y>0$, we get that 
$J_{322} \leq  e^{-c n} \lVert \varphi \rVert_{\gamma}$.  
Hence, there exists a constant $C >0$ 
such that for all $|l| \leq \frac{1}{\sqrt{n}}$, $s \in (-\eta, \eta)$ and $\varphi \in \mathcal{B}_{\gamma}$, 
it holds that $J_{32} \leq \frac{C}{\sqrt{n}}  \lVert \varphi \rVert_{\gamma}$. 
This, together with \eqref{Ch2LDfghj} and \eqref{Ch2PropRn1aya},  implies 
the desired bound for $J_3$: there exists a constant $C >0$ such that
for all $n \geq 1$, $|l| \leq \frac{1}{\sqrt{n}}$, $s \in (\eta, \eta)$, 
$x \in \mathcal{S}$ and $\varphi \in \mathcal{B}_{\gamma}$, 
\begin{align*}
J_3 \leq  \frac{C}{\sqrt{n}}  \lVert \varphi \rVert_{\gamma} 
 + \frac{C}{n} \lVert \varphi \rVert_{\gamma} \sup_{|t| \leq \delta} \big( |\psi(t)| + |\psi'(t)| \big). 
\end{align*}
Combining this with \eqref{Thm1 integral1 J1} and \eqref{SaddleIntegral 2}, 
we conclude the proof of Proposition \ref{Prop_R_st_limit}. 
\end{proof}

Now we are equipped to establish Theorem \ref{ThmLocal01}. 

\begin{proof}[Proof of Theorem \ref{ThmLocal01}]
We only need to establish the first assertion of the theorem 
since the second and the third ones are its particular cases. 
By the change of measure formula \eqref{basic equ1}, 
we get that for any $s \in (-\eta, \eta)$ with sufficiently small $\eta>0$,  
\begin{align}\label{LLT_Def_An_d}
J_n : & =  \mathbb{E} \Big[ \varphi(X_n^x) \psi \big( \sigma(G_n, x) - n\lambda - \sqrt{n}\sigma y \big) \Big]   \\
      & =  r_s(x) \kappa^{n}(s) \mathbb{E}_{ \mathbb{Q}^{x}_{s} }
\Big[ (\varphi r_s^{-1})(X_n^x) e^{-s \sigma(G_n, x) } 
   \psi \big( \sigma(G_n, x) - n\lambda - \sqrt{n}\sigma y \big) \Big]. \notag
\end{align}
For brevity, denote 
\begin{align*}
T_n^x = \sigma(G_n, x) - n \Lambda'(s).
\end{align*}
By considering equation \eqref{equation s and y 1} for any $|y| = o(\sqrt{n})$ (not necessarily $|y| > 1$), 
we get the identity \eqref{zetaLamb} for $|y| = o(\sqrt{n})$. 
Hence, we have
\begin{align*}
J_n & = r_s(x) e^{-n [s\Lambda'(s) - \Lambda(s)]} \mathbb{E}_{ \mathbb{Q}^{x}_{s} }
\Big[ (\varphi r_s^{-1})(X_n^x) e^{-s T_n^x } \psi(T_n^x) \Big] \notag\\
& = r_s(x) e^{- \frac{y^2}{2} + \frac{y^3}{ \sqrt{n} } \zeta( \frac{y}{\sqrt{n} } )} \mathbb{E}_{ \mathbb{Q}^{x}_{s} }
\Big[ (\varphi r_s^{-1})(X_n^x) e^{-s T_n^x } \psi(T_n^x) \Big]. 
\end{align*}
We denote
\begin{align}\label{LLT_psi_jjj}
\psi_s(u) = e^{-su} \psi(u),  \quad u \in \mathbb R. 
\end{align} 
Taking into account Lemma \ref{lem:expconv}, 
in order to establish Theorem \ref{ThmLocal01}, it is sufficient to prove 
the following asymptotic: as $n \to \infty$, 
\begin{align}\label{LLT_Pf_Object}
A_n : = \sigma \sqrt{2 \pi n} \,  \mathbb{E}_{ \mathbb{Q}^{x}_{s} }
\Big[ (\varphi r_s^{-1})(X_n^x)  \psi_s(T_n^x) \Big]
\to \nu(\varphi) \int_{\mathbb R} \psi(u) \, du. 
\end{align}
To prove \eqref{LLT_Pf_Object}, we need to use some smoothing techniques.
For sufficiently small $\varepsilon >0$, we denote for any $s \in (-\eta, \eta)$ and $u \in \mathbb R$, 
\begin{align}\label{LLT_psi_lll}
\psi_{s, \varepsilon}^+(u) = \sup_{u' \in \mathbb R: |u' - u| \leq \varepsilon}  \psi_s(u'), 
\qquad
\psi_{s, \varepsilon}^-(u) = \inf_{u' \in \mathbb R: |u' - u| \leq \varepsilon}  \psi_s(u').  
\end{align}
Denote respectively by $\widehat{\psi}^+_{s, \varepsilon}$ and $\widehat{\psi}^-_{s, \varepsilon}$
the Fourier transform of $\psi_{s, \varepsilon}^+$ and $\psi_{s, \varepsilon}^-$. 
For the moment we suppose that
\begin{align} \label{eq-cond-coincidence001}
\lim_{\varepsilon \to 0} \widehat{\psi}^+_{0, \varepsilon}(0) = \lim_{\varepsilon \to 0} \widehat{\psi}^-_{0, \varepsilon}(0) 
= \int_{\mathbb R} \psi(u) \, du. 
\end{align}
Note that the Fourier transform of the function $\psi_s$ may not be integrable on $\mathbb R$. 
In the sequel we shall use a smoothing inequality from \cite[Lemma 5.2]{GLL17}, 
which gives two-sided bounds for $\psi_s$. 
Let $\rho$ be a non-negative density function on $\mathbb{R}$  with $\int_{\mathbb{R}} \rho(u) \, du = 1$
and $\rho(u) \leq \frac{C}{1 + u^4}$ for all $u \in \mathbb R$, 
so that its Fourier transform $\widehat{\rho}$ is supported on $[-1,1]$. 
For any $0< \varepsilon < 1$, define the rescaled density function $\rho_{\varepsilon}$ by
$\rho_{\varepsilon}(u) = \frac{1}{\varepsilon} \rho(\frac{u}{\varepsilon})$, $u\in \mathbb R,$ 
whose Fourier transform has a compact support on $[-\varepsilon^{-1},\varepsilon^{-1}]$. 
Then, there exists a positive constant $C_{\rho}(\varepsilon)$ with $C_{\rho}(\varepsilon) \to 0$ as $\varepsilon \to 0$,
such that for any $u \in \mathbb{R}$, 
\begin{align}\label{LLT_Smooth_nnn}
\psi^-_{s, \varepsilon}\!\ast\!\rho_{\varepsilon^2}(u) - 
\int_{|v|\geq \varepsilon} {\psi}^-_{s,\varepsilon}(u - v) \rho_{\varepsilon^2}(v) \,dv
\leq \psi_s(u) \leq (1+ C_{\rho}(\varepsilon))
\psi^+_{s, \varepsilon}\!\ast\!\rho_{\varepsilon^2}(u). 
\end{align}
Now we are going to prove \eqref{LLT_Pf_Object}. 
The proof will be done by establishing  upper and lower bounds for $A_n$. 
Without loss of generality, we assume that the target functions $\varphi$ and $\psi$ are non-negative. 

\textit{Upper bound.} 
Applying the smoothing inequality \eqref{LLT_Smooth_nnn}
and the Fourier inversion formula to the function $\psi^+_{s, \varepsilon}\!\ast\!\rho_{\varepsilon^2}$, we get 
\begin{align}\label{LLT_An_UPP_i}
A_n & \leq  (1+ C_{\rho}(\varepsilon)) \sigma \sqrt{2 \pi n} \,  \mathbb{E}_{ \mathbb{Q}^{x}_{s} }
         \Big[ (\varphi r_s^{-1})(X_n^x) (\psi^+_{s, \varepsilon} * \rho_{\varepsilon^2})(T_n^x) \Big]  \notag\\
    & = (1+ C_{\rho}(\varepsilon)) \sigma \sqrt{\frac{n}{2 \pi}} 
         \int_{\mathbb R} R_{s,it}^n (\varphi r_s^{-1}) (x) \widehat{\psi}^+_{s, \varepsilon} (t) \widehat{\rho}_{\varepsilon^2} (t) \,dt,
\end{align}
where $R_{s,it}$ is the perturbed operator defined by \eqref{Def-Rn-sw000} with $z = it$. 
Applying Proposition \ref{Prop_R_st_limit} with $\varphi = \varphi r_s^{-1}$
and $\psi = \widehat{\psi}^+_{s, \varepsilon} \widehat{\rho}_{\varepsilon^2}$
(one can verify that the remainder term in \eqref{R_st_limit} vanishes as $n \to \infty$, 
uniformly in $s \in (-\eta, \eta)$),
we obtain, uniformly in $s \in (-\eta, \eta)$, $|t| \geq \delta$ and $x\in \mathcal{S}$, 
\begin{align*}
\limsup_{n \to \infty} A_n \leq  (1+ C_{\rho}(\varepsilon)) \nu(\varphi) \widehat{\psi}^+_{0, \varepsilon}(0).
\end{align*}
Letting $\varepsilon \to 0$,  we get the desired upper bound for $A_n$:
uniformly in $s \in (-\eta, \eta)$ and $x\in \mathcal{S}$, 
\begin{align}\label{LLT_An_upp_final}
\limsup_{n \to \infty} A_n \leq \nu(\varphi) \lim_{\varepsilon \to 0} \widehat{\psi}^+_{0, \varepsilon}(0).
\end{align}

\textit{Lower bound.} 
Similarly to \eqref{LLT_An_UPP_i}, 
using the smoothing inequality \eqref{LLT_Smooth_nnn}, the fact that 
$\psi^-_{s,\varepsilon} \leq \psi_s  \leq (1+ C_{\rho}(\varepsilon)) \psi_{s,\varepsilon}^+ \! \ast \! \rho_{\varepsilon^2}$, 
and the Fourier inversion formula to the functions $\psi^-_{s, \varepsilon}\!\ast\!\rho_{\varepsilon^2}$ and
$\psi_{s,\varepsilon}^+ \! \ast \! \rho_{\varepsilon^2}$, we obtain 
\begin{align}\label{LLT_An_Low_j}
 A_n  & \geq  \sigma \sqrt{2 \pi n} \,  \mathbb{E}_{ \mathbb{Q}^{x}_{s} }
         \Big[ (\varphi r_s^{-1})(X_n^x) (\psi^-_{s, \varepsilon} * \rho_{\varepsilon^2}) (T_n^x) \Big]  \notag\\
      & \quad  - \sigma \sqrt{2 \pi n} \int_{ |v| \geq \varepsilon}  \mathbb{E}_{ \mathbb{Q}^{x}_{s} }  
        \Big[ (\varphi r_s^{-1})(X_n^x) \psi^-_{s, \varepsilon} (T_n^x - v) \Big] \rho_{\varepsilon^2}(v) \, dv \notag\\
    & \geq  \sigma \sqrt{\frac{n}{2 \pi}} 
 \int_{\mathbb R}  R_{s,it}^n (\varphi r_s^{-1}) (x) \widehat{\psi}^-_{s, \varepsilon} (t) \widehat{\rho}_{\varepsilon^2} (t) \,dt  \notag\\
    & \quad  - (1+ C_{\rho}(\varepsilon)) \sigma \sqrt{\frac{n}{2 \pi}} 
         \int_{ |v| \geq \varepsilon} 
\bigg[ \int_{\mathbb R} e^{-itv} R_{s,it}^n (\varphi r_s^{-1}) (x) 
       \widehat{\psi}^+_{s, \varepsilon} (t) \widehat{\rho}_{\varepsilon^2} (t) \,dt  \bigg] 
        \rho_{\varepsilon^2}(v) \, dv   \notag\\
  & = : B_n(\varepsilon) - D_n(\varepsilon). 
\end{align}
For $B_n(\varepsilon)$, in the same way as in the proof of \eqref{LLT_An_upp_final}, 
by considering the function $\psi^-_{s, \varepsilon}$ instead of $\psi^+_{s, \varepsilon}$ and
using Proposition \ref{Prop_R_st_limit}, we have that uniformly in $s \in (-\eta, \eta)$ and $x\in \mathcal{S}$, 
\begin{align}\label{LLT_Bn_Low_final}
\liminf_{\varepsilon \to 0} \liminf_{n \to \infty} B_n(\varepsilon) \geq \nu(\varphi) \lim_{\varepsilon \to 0} \widehat{\psi}^-_{0, \varepsilon}(0).
\end{align}
For $D_n(\varepsilon)$, 
we first note that we can follow the proof of the upper bound for $A_n$ 
to check the following asymptotic: for sufficiently small $\varepsilon >0$, 
uniformly in $s \in (-\eta, \eta)$, $x\in \mathcal{S}$ and $v \in [-\sqrt{n}, \sqrt{n}]$, 
\begin{align}\label{LLT_Ext_ff}
\lim_{n \to \infty} 
\sigma \sqrt{\frac{n}{2 \pi}}  e^{\frac{v^2}{2 n \sigma_s^2}}
 \int_{\mathbb R} e^{-itv} R_{s,it}^n (\varphi r_s^{-1}) (x) \widehat{\psi}^+_{s, \varepsilon} (t) \widehat{\rho}_{\varepsilon^2} (t) \,dt  
= \nu(\varphi) \widehat{\psi}^+_{0, \varepsilon}(0). 
\end{align}
To obtain an upper bound for the term $D_n(\varepsilon)$,  
we shall apply the Lebesgue dominated convergence theorem 
to pass to the limit as $n \to \infty$ through the integral $\int_{|v| \geq \varepsilon}$. 
The applicability of this theorem is justified below. 
We split the integral $\int_{|v| \geq \varepsilon}$ in the term $D_n(\varepsilon)$
into two parts: $\int_{ |v| >  \sqrt{n}}$ and $\int_{ \varepsilon \leq |v| \leq \sqrt{n}}$.
For the first part $\int_{ |v| >  \sqrt{n}}$, since the density function $\rho_{\varepsilon^2}$ has 
polynomial decay, i.e. $\rho_{\varepsilon^2}(v) \leq \frac{C_\varepsilon}{1 + v^4}$, $|v| > \sqrt{n}$,
we get that $\sqrt{n} \rho_{\varepsilon^2}(v) \leq \frac{C_\varepsilon}{1 + |v|^3}$, 
which is integrable on $\mathbb{R}$.  
For the second part, using \eqref{LLT_Ext_ff} we see that, the function
under the integral $\int_{\varepsilon \leq  |v| \leq \sqrt{n}}$ 
is dominated by  $C \rho_{\varepsilon^2}$ which is integrable on $\mathbb{R}$. 
Therefore, we can interchange 
the limit as $n\to \infty$ 
and the integral $\int_{|v|\geq \varepsilon}$, 
and then use \eqref{LLT_Ext_ff} again
to obtain that uniformly in $s \in (-\eta, \eta)$ and $x\in \mathcal{S}$, 
\begin{align*}
 \limsup_{n\to\infty} D_n(\varepsilon) 
& \leq  (1+ C_{\rho}(\varepsilon))  \nu(\varphi) \widehat{\psi}^+_{0, \varepsilon}(0)
\int_{|v| \geq \varepsilon} \rho_{\varepsilon^2}(v) \, dv.   
\end{align*}
The integral on right-hand side converges to $0$ as $\varepsilon \to 0$, 
since $\rho_{\varepsilon^2}(v) = \frac{1}{\varepsilon^2} \rho(\frac{v}{\varepsilon^2})$ 
and the function $\rho$ is integrable on $\mathbb{R}$.
Together with \eqref{LLT_An_Low_j} and \eqref{LLT_Bn_Low_final}, this implies the desired lower bound for $A_n$: 
uniformly in $s \in (-\eta, \eta)$ and $x\in \mathcal{S}$, 
\begin{align} \label{lower bound An 001}
\liminf_{n \to \infty} A_n \geq \nu(\varphi) \lim_{\varepsilon \to 0} \widehat{\psi}^-_{0, \varepsilon}(0). 
\end{align}

Combining \eqref{LLT_An_upp_final} and \eqref{lower bound An 001}, 
we obtain the assertion of Theorem \ref{ThmLocal01}, 
provided that \eqref{eq-cond-coincidence001} holds. 
Condition \eqref{eq-cond-coincidence001} can be relaxed to the direct Riemann integrability condition of the target function $\psi$,
by applying the approximation techniques developed in \cite{XGL18}. 
So the proof of Theorem \ref{ThmLocal01} is complete. 
\end{proof}


\subsection{Proof of Theorem \ref{Thm_LLT_Norm_0a}} 
In this subsection we prove Theorem \ref{Thm_LLT_Norm_0a}  
concerning the local limit theorem with moderate deviations 
for the operator norm $\lVert G_n \rVert$ in the case of invertible matrices. 
In this proof Theorem \ref{ThmLocal01}  plays the key role. 
Another important ingredient is the following Lemma \ref{Lem_Com_Gn_Gnx}
established recently by Benoist and Quint \cite{BQ16b},  
which provides a precise and interesting comparison between $\log \lVert G_n \rVert$ and $\sigma(G_n, x)$:

\begin{lemma}\label{Lem_Com_Gn_Gnx}
Assume conditions \ref{CondiMoment} and \ref{CondiIP} for invertible matrices. 
Then, for any $a>0$, 
there exist $c >0$ and $k_0 \in \mathbb{N}$, such that for all $n \geq k \geq k_0$
and $x = \mathbb R v \in \mathbb P^{d-1}$,  
\begin{align*}
\mathbb{P} \Big(   
\Big| \log \frac{\lVert G_n \rVert}{\lVert G_k \rVert} 
 - \log \frac{|G_n v|}{|G_k v|}   \Big| \leq e^{-a k}  
   \Big)
> 1- e^{ - ck }. 
\end{align*}
\end{lemma}

\begin{proof}[Proof of Theorem \ref{Thm_LLT_Norm_0a}]
Without loss of generality, we assume that the target function $\varphi$ is non-negative.

We first give the upper bound. 
By Lemma \ref{Lem_Com_Gn_Gnx}, we get that for any $a>0$, 
there exist $c >0$ and $k_0 \in \mathbb{N}$, such that for all $n \geq k \geq k_0$ and $x = \mathbb R v \in \mathbb P^{d-1}$,  
\begin{align*}
 J_n: & = \mathbb{E} \Big[ \varphi(X_n^x)
    \mathds{1}_{\{ \log \lVert G_n \rVert - n\lambda \in [a_1, a_2] + \sqrt{n}\sigma y \} }
     \Big] \notag \\
& \leq  \mathbb{E} \Big[ \varphi(X_n^x)
    \mathds{1}_{ \big\{ \log \frac{|G_n v|}{|G_k v|} + \log \lVert G_k\rVert - n\lambda 
     \in [a_1 - e^{-ak}, a_2 + e^{-ak}] + \sqrt{n}\sigma y \big\} }
     \Big] + e^{-ck} \lVert \varphi \rVert_{\infty}.  
\end{align*}
With the notation $G_{n,k} = g_n \ldots g_{k+1}$ for any $n \geq k \geq 1$, 
we have $X_n^x = G_{n,k} \cdot X_k^x$ and $\sigma(G_n, x) - \sigma(G_k, x) = \sigma(G_{n,k}, X_k^x)$. 
Thus the first term of the right-hand side of the above inequality can be rewritten as 
\begin{align*}
\mathbb{E} \Big[ \varphi(G_{n,k} \cdot X_k^x)
    \mathds{1}_{ \big\{ \sigma(G_{n-k}, X_k^x) - (n-k) \lambda \in [a_1 - e^{-ak}, a_2 + e^{-ak}] + \sqrt{n}\sigma y
    - (\log \lVert G_k\rVert - k \lambda) \big\} }
     \Big]. 
\end{align*}
Now we fix a sufficiently large constant $C_1 >0$ and we choose 
\begin{align*}
k = \floor{ C_1 y^2 },   
\end{align*}
where $\floor{ y }$ denotes the integer part of $y \in \mathbb R$. 
For any $\varepsilon >0$, there exists a large enough $k_1 \geq 1$ such that for all $k \geq k_1$,
\begin{align*}
[a_1 - e^{-ak}, a_2 + e^{-ak}] \subset I^+_{\varepsilon} : = [a_1 - \varepsilon, a_2 + \varepsilon]. 
\end{align*}
Using the large deviation bounds for $\log \lVert G_k \rVert$ (see \cite{BQ16b} or \cite{XGL18}), 
we see that for any $\delta > 0$, there exists a constant $c>0$ such that
for large enough $k \geq 1$, 
\begin{align*} 
\mathbb{P} \Big( \Big| \log \lVert G_k \rVert - k \lambda \Big| > k \delta \Big) \leq e^{-ck}. 
\end{align*}
Using this bound, it follows that
\begin{align*}
J_n & \leq \mathbb{E} \Big[ \varphi(G_{n,k} \cdot X_k^x)
    \mathds{1}_{ \big\{ \sigma(G_{n-k}, X_k^x) - (n-k) \lambda \in I^+_{\varepsilon} + \sqrt{n}\sigma y
    - (\log \lVert G_k\rVert - k \lambda) \big\} }    
   \mathds{1}_{ \big\{ \big| \log \lVert G_k \rVert - k \lambda \big| \leq k \delta \big\}  }
     \Big]    \notag\\
 & \quad  + e^{-ck} \lVert \varphi \rVert_{\infty}.    
\end{align*}
Taking conditional expectation given the $\sigma$-algebra 
$\mathscr{F}_k = \sigma (g_1, \ldots, g_k)$,
we get
\begin{align*}
J_n & \leq  \mathbb{E} \Big\{ \mathbb{E} \Big[ \varphi(G_{n,k} \cdot X_k^x)
    \mathds{1}_{ \big\{ \sigma(G_{n-k}, X_k^x) - (n-k) \lambda \in I^+_{\varepsilon} + \sqrt{n}\sigma y
    - (\log \lVert G_k \rVert - k \lambda) \big\} }      \notag\\
  & \qquad \qquad  \mathds{1}_{ \big\{ \big| \log \lVert G_k \rVert - k \lambda \big| \leq k \delta \big\}  }
      \Big|  \mathscr{F}_k  \Big]  \Big\}   + e^{-ck} \lVert \varphi \rVert_{\infty}.    
\end{align*}
Applying Theorem \ref{ThmLocal01}, we obtain that, as $n \to \infty$,
uniformly in $x \in \mathbb P^{d-1}$ and $|y| = o(n^{1/6})$,
\begin{align}\label{Pf_LLT_Upp111}
J_n  & \leq  \sup_{|u| \leq k \delta} 
 \exp \Big\{ -\frac{1}{2} \Big(\frac{y \sqrt{n}}{\sqrt{n-k}} - \frac{u}{\sigma \sqrt{n-k}} \Big)^2  \Big\} 
   \frac{ (a_2 - a_1 + 2 \varepsilon) \nu(\varphi)  +  o(1) }{\sigma   \sqrt{2 \pi n}}      
    + e^{-ck} \lVert \varphi \rVert_{\infty}.  
\end{align}
Since $k = \floor{C_1 y^2}$, 
it follows that as $n \to \infty$, 
\begin{align}\label{LLT_Norm_Upper}
J_n \leq  \frac{ e^{- \frac{y^2}{2}} }{\sigma   \sqrt{2 \pi n}}  
      \Big[ (a_2 - a_1 + 2 \varepsilon) \nu(\varphi) + o(1) \Big].
\end{align}

\medskip

We next give the lower bound.
Since the proof is similar to that of the upper bound, we only sketch the main differences. 
By Lemma \ref{Lem_Com_Gn_Gnx}, we get that for any $a>0$, 
there exist $c >0$ and $k_0 \in \mathbb{N}$, such that for all $n \geq k \geq k_0$ and $x = \mathbb R v \in \mathbb P^{d-1}$,  
\begin{align*}
J_n \geq  \mathbb{E} \Big[ \varphi(X_n^x)
    \mathds{1}_{ \big\{ \log \frac{|G_n v|}{|G_k v|} + \log \lVert G_k \rVert - n\lambda 
     \in [a_1 + e^{-ak}, a_2 - e^{-ak}] + \sqrt{n}\sigma y \big\} }
     \Big].  
\end{align*}
With the notation used in the proof of the upper bound, we have
\begin{align*}
J_n & \geq \mathbb{E} \Big[ \varphi(G_{n,k} \cdot X_k^x)
    \mathds{1}_{ \big\{ \sigma(G_{n-k}, X_k^x) - (n-k) \lambda \in I^-_{\varepsilon} + \sqrt{n}\sigma y
    - (\log \lVert G_k \rVert  - k \lambda) \big\} }     
 \mathds{1}_{ \big\{ \big| \log \lVert G_k \rVert - k \lambda \big| \leq k \delta \big\}  }
     \Big], 
\end{align*}
where $I^-_{\varepsilon} : = [a_1 + \varepsilon, a_2 - \varepsilon]. $
Notice that, for any $\varepsilon >0$, there exists a large enough $k_1 \geq 1$ such that for all $k \geq k_1$,
\begin{align*}
I^-_{\varepsilon} \subset [a_1 + e^{-ak}, a_2 - e^{-ak}]. 
\end{align*}
In the same way as in the proof of \eqref{Pf_LLT_Upp111},
we take conditional expectation given $\mathscr{F}_k$
and use Theorem \ref{ThmLocal01} to obtain that as $n \to \infty$,
uniformly in $x \in \mathbb P^{d-1}$ and $|y| = o(n^{1/6})$,
\begin{align*}
J_n  & \geq   \frac{1}{\sigma   \sqrt{2 \pi n}}     
  \Big[ (a_2 - a_1 - 2 \varepsilon) \nu(\varphi) -  o(1) \Big] 
  \inf_{|u| \leq k \delta} 
 \exp \Big\{ -\frac{1}{2} \Big(\frac{y \sqrt{n}}{\sqrt{n-k}} - \frac{u}{\sigma \sqrt{n-k}} \Big)^2  \Big\}.  
\end{align*}
As $k = \floor{ C_1 y^2 }$, elementary calculations lead to
\begin{align}\label{LLT_Norm_Lower}
J_n \geq  \frac{ e^{- \frac{y^2}{2}} }{\sigma   \sqrt{2 \pi n}}  
      \Big[ (a_2 - a_1 + 2 \varepsilon) \nu(\varphi) - o(1) \Big]. 
\end{align}
Since $\varepsilon>0$ can be arbitrary small, combining \eqref{LLT_Norm_Upper} and \eqref{LLT_Norm_Lower}, 
we conclude the proof of Theorem \ref{Thm_LLT_Norm_0a}. 
\end{proof}


\begin{ack}
We would like to thank the referees and the associate editor for their helpful comments and remarks.   
\end{ack}

\begin{funding}
This work was supported by the Centre Henri Lebesgue (CHL, ANR-11-LABX-0020-01), 
 and the National Natural Science Foundation of China (Grants Nos. 11971063 and 11731012). 
\end{funding}


\end{document}